\documentclass[reqno,11pt]{amsart}

\usepackage{bbm}
\usepackage{amsmath}
\usepackage{amssymb}
\usepackage[colorlinks=true]{hyperref}

\hypersetup{urlcolor=red, citecolor=blue}
\makeatletter
\@namedef{subjclassname@2020}{%
	\textup{2020} Mathematics Subject Classification}
\makeatother

\newtheorem{thm}{Theorem}[section]
\newtheorem{lem}[thm]{Lemma}
\newtheorem{prop}[thm]{Proposition}
\newtheorem{cor}[thm]{Corollary}

\theoremstyle{definition}
\newtheorem{defn}[thm]{Definition}
\newtheorem{notation}[thm]{Notation}

\theoremstyle{remark}
\newtheorem{rem}[thm]{Remark}

\theoremstyle{plain}
\newtheorem{thmx}{Theorem}

\newtheorem{corx}[thmx]{Corollary}

\newtheorem*{thmstar}{Theorem}

\def \N {\mathbb N}
\def \R {\mathbb R}
\def \Z {\mathbb Z}
\def \m {\mathcal M}
\def \w {\mathbb{W}}
\def \cA {\mathcal A}

\def \cQ {\mathcal Q}
\def \cP {\mathcal P}
\def \cV {\mathcal V}
\def \lip {{\rm Lip}}
\newcommand{\ind}{\mathbbm{1}}
\newcommand{\Mmax}{\m_{\max}}
\newcommand{\lex}{\mathrm{lex}}
\newcommand{\bfa}{\mathbf{a}}
\newcommand{\bfb}{\mathbf{b}}

\numberwithin{equation}{section}

\begin{document}

\title[Beta-shifts without TPO]{Beta-shifts without typical periodic
optimization}

\author[W. Huang, O. Jenkinson, L. Xu and Y. Zhang]{Wen Huang, Oliver
Jenkinson, Leiye Xu and Yiwei Zhang}

\address[W. Huang]{School of Mathematical Sciences, University of Science
and Technology of China, Hefei, Anhui 230026, P.R. China}
\email{wenh@mail.ustc.edu.cn}

\address[O.~Jenkinson]{School of Mathematical Sciences, Queen Mary
University of London, Mile End Road, London, E1 4NS, UK}
\email{o.jenkinson@qmul.ac.uk}

\address[L. Xu]{School of Mathematical Sciences, University of Science and
Technology of China, Hefei, Anhui 230026, P.R. China}
\email{leoasa@mail.ustc.edu.cn}

\address[Y.~Zhang]{School of Mathematical Sciences and Big Data, Anhui
University of Science and Technology, Huainan, Anhui 232001, P.R. China}
\email{zhangyw@aust.edu.cn}

\date{\today}

\begin{abstract}
We show that beta-shifts whose lexicographically largest point is a
Sturmian sequence do not have typical periodic optimization (TPO): on each
such beta-shift there is a non-empty open set of Lipschitz functions whose
unique maximizing measure is the Sturmian measure, which is not supported on
a periodic orbit.  These are the first known examples of beta-shifts without
TPO, and, since they have the specification property, the first known
examples of shift spaces with specification but without TPO.  The
corresponding beta-transformations do not have TPO in any space of H\"older
functions on the interval.  The main ingredient in the proofs is a rigidity
property of Sturmian subshifts: modulo constants and Lipschitz coboundaries,
the space of Lipschitz functions on such a subshift is one-dimensional.

We also show that the set of $\beta>1$ for which the beta-shift or the
beta-transformation does not have TPO is dense in $(1,\infty)$, but has zero
packing dimension.  The density is proved by a tuning construction,
analogous to Douady--Hubbard tuning in the Mandelbrot set, which produces
non-Sturmian subshifts with the same rigidity property.  Finally, we show
that every beta-shift has typical zero-entropy optimization: for an open
dense set of Lipschitz functions, the maximizing measure is unique and its
support has zero topological entropy.
\end{abstract}

\maketitle

\setcounter{tocdepth}{1}
\tableofcontents

\newpage

\section*{Introduction}

For dynamical systems with some hyperbolicity, and an abundance of periodic
orbits, it has long been conjectured that maximizing measures are
\emph{typically periodic}.  This \emph{typical periodic optimization (TPO)}
phenomenon was observed in numerical experiments in the 1990s, notably by
Hunt \& Ott \cite{HO96,HO96b}, and has since been proved rigorously in
various forms \cite{BZ16,Bousch2000,Bousch2001,CLT01,GSZ25,QS12} (see
\cite{Boc18,Jen06,Jen19} for surveys of the wider area of ergodic optimization).
Much of the rigorous work was motivated by the
conjecture of Yuan \& Hunt \cite{YH99} that, for uniformly hyperbolic
systems, all members of an open dense set of Lipschitz functions have a unique maximizing
measure, supported on a periodic orbit.  The Yuan--Hunt conjecture was
proved by Contreras \cite{Con16} for open distance-expanding maps, and for
general uniformly hyperbolic systems by Huang, Lian, Ma, Xu \& Zhang
\cite{HLMXZ}.
In his 2018 ICM lecture, Bochi \cite{bochi_slides} formulated, after Hunt \&
Ott, the more general conjecture that for chaotic systems and typical
regular functions, the maximizing measure has low complexity, for example
support of zero topological entropy or, more strongly, support on a periodic
orbit.

In \cite{HJXZ} we showed that, in the setting of symbolic dynamics, the
results of \cite{Con16,HLMXZ}
(in particular their application to shifts of finite type) can be
used as a basis for establishing TPO for a much wider class of shift spaces.  It
was also shown there, via an example called the \emph{magic Morse shift},
that TPO can fail, even for shift spaces whose periodic orbit measures are
weak$^*$ dense in the space of all invariant measures.
The magic Morse shift is synchronized, but does not have the \emph{specification} property (see Remark
\ref{rem:magic-morse}), a
condition introduced by Bowen \cite{bowen} which is often regarded as a
hallmark of hyperbolicity; this left open the question of whether all shift
spaces with specification have TPO.

In this paper we show that they do not.  Unlike the \textit{in vitro} example of the magic Morse shift, which
was constructed for the express purpose of exhibiting the failure of TPO,
the shift spaces considered here are classical: they are the
\emph{Sturmian beta-shifts}, that is, the beta-shifts
\cite{Renyi57,Par60,Bertrand86,blanchard} whose lexicographically largest
point is a Sturmian sequence, all of which have the specification property.

A further motivation was the study of TPO for beta-transformations
$T_\beta(t)=\beta t\pmod 1$ on $[0,1]$, whose symbolic models are the
beta-shifts, by Hao, Huang, Jenkinson \& Li \cite{HHJL}.  They showed that
$T_\beta$ has TPO in all H\"older spaces $C^{0,\gamma}([0,1])$ for Lebesgue
almost every $\beta>1$, and for a residual set of $\beta$
\cite[Thms.~D, E]{HHJL}; the remaining parameters are among the so-called \emph{emergent} ones.
These include the simple beta-numbers, for which $T_\beta$ fails to have
TPO only because the functions in some non-empty open set have no maximizing
measure \cite[Rem.~5.21]{HHJL}, while the closely related upper
beta-transformation $U_\beta$ does have TPO \cite[Thm.~I]{HHJL}.  The question of whether
$U_\beta$ has TPO for all emergent $\beta$ was left open \cite[\S 1.3]{HHJL}.

The parameters $\beta$ whose beta-shift is a Sturmian beta-shift are
emergent, and for these we show
that TPO fails in both settings, so that this question has a negative answer
(see Section \ref{sec:statement} for precise statements):

\begin{thmstar}
Let $\beta>1$ be such that the lexicographically largest point of the
beta-shift of $\beta$ is a Sturmian sequence.  Then the beta-shift does not
have TPO for any of its standard metrics, and neither $T_\beta$ nor $U_\beta$
has TPO in any H\"older space $C^{0,\gamma}([0,1])$, $0<\gamma\le1$.  In each
case there is a non-empty open set of functions whose unique maximizing
measure is the Sturmian measure.
\end{thmstar}

The proof is based on a rigidity property of Sturmian subshifts: modulo
constants and Lipschitz coboundaries, the space of Lipschitz functions on a
Sturmian subshift is one-dimensional, spanned by the class of the digit
function $x\mapsto x_0$ (Theorem \ref{thm:D}).  As the Sturmian measure
maximizes the integral of the digit function among invariant measures on the
beta-shift, this rigidity makes it the unique maximizing measure for all
small Lipschitz perturbations of the digit function minus the distance to
the Sturmian subshift.  For the interval maps we also use an inequality
comparing distances in the beta-shift and in $[0,1]$ (Lemma \ref{lem:key}),
which yields a transfer principle between the two settings (Theorem~\ref{thm:F}).

The above theorem applies to uncountably many parameters $\beta$, one for
each irrational rotation number.  Our second main result concerns all
$\beta>1$:

\begin{thmstar}
The set of parameters $\beta>1$ for which the beta-shift does not
have TPO is dense in $(1,\infty)$, and has packing dimension zero.
For all $\beta>1$, the
beta-shift has typical zero-entropy optimization: some open dense set of
Lipschitz functions consists of functions whose maximizing measure is
unique, and has support of zero topological entropy.
\end{thmstar}

So for beta-shifts, Bochi's conjecture holds for all $\beta>1$ in its
zero-entropy form, but fails for a dense set of $\beta$ in its periodic
form.  By Theorems \ref{thm:dimension}, \ref{thm:TZO} and \ref{thm:density},
both parts of the above theorem also hold for $U_\beta$, with H\"older
functions on $[0,1]$ in place of Lipschitz functions on the beta-shift, and
they hold for $T_\beta$ except that the second fails if $\beta$ is a simple
beta-number; the first part sharpens the almost-everywhere result
of \cite[Thm.~E]{HHJL}.  This reflects a close relation between the two
settings: for all $\beta>1$, the beta-shift has TPO if and only if $U_\beta$
does, and if and only if $T_\beta$ does unless $\beta$ is a simple
beta-number (Corollary \ref{cor:G}).

The packing dimension and zero-entropy statements of the second theorem are
based on a version for beta-shifts of a theorem of Yuan \& Hunt
\cite{YH99}, according to which, in the formulation of Bousch
\cite{Bousch2008}, for uniformly hyperbolic systems an invariant measure
which is maximizing for all Lipschitz functions in a non-empty open set must
be a periodic orbit measure.  On a beta-shift such a measure need not be
periodic, as the Sturmian examples show, but its support does have zero
topological entropy; indeed, for some $C>0$, at most $\exp(C\sqrt{n\log n})$
words of length $n$ occur in it (Theorem \ref{thm:robust}).  If the
beta-shift does not have TPO, then the orbit closure of its
lexicographically largest point is the support of such a measure, so has
zero topological entropy, which happens only for a set of parameters $\beta$
of packing dimension zero; moreover typical functions have a unique
maximizing measure, which is either periodic or supported by this orbit
closure.  This accords with the work of Hunt \& Ott \cite{HO96,HO96b}, who
argued that, in the exceptional cases where no periodic orbit is optimal,
the optimal orbits have limit sets of zero topological entropy.

The density of the parameters without TPO is proved by a tuning
construction, a version for beta-shifts of Douady--Hubbard tuning of external
angles of the Mandelbrot set \cite{Dou86,DH85}.  Replacing the digits of the
lexicographically largest point of a Sturmian beta-shift by two suitable
words of equal length gives the lexicographically largest point of what we
call a \emph{tuned Sturmian beta-shift}.  The failure of TPO, and the
rigidity property behind it, extend to these beta-shifts (Theorems
\ref{thm:tuned-rigidity} and \ref{thm:tuned}), and their parameters
accumulate at all simple beta-numbers.
If the two words have length at least $2$, then the orbit closure of the
lexicographically largest point is not topologically conjugate to a
Sturmian subshift (Proposition \ref{prop:tuned-not-sturmian}), so the
rigidity property is not restricted to Sturmian subshifts.

\section{Statement of results}\label{sec:statement}

In this section we fix notation, state our results, and outline their
proofs.

\subsection{Notation, and typical periodic optimization}

Let $\N_0:=\{0,1,2,\ldots\}$, for $m\in\N$ let
$\Sigma_m:=\{0,1,\ldots,m\}^{\N_0}$ be equipped with the product
topology, and write $\Sigma:=\Sigma_1=\{0,1\}^{\N_0}$.  Let $\sigma$ denote the
shift map $(\sigma x)_i:=x_{i+1}$ on each $\Sigma_m$,
let $\le_\lex$ denote the lexicographic order on sequences of non-negative
integers, and write $\max_\lex S$ for
the lexicographically largest element of a set $S$ of such sequences, when
this exists.  
A \emph{subshift} is a closed set $X\subseteq\Sigma_m$ with
$\sigma(X)=X$, for some $m\in\N$, and
$\m(X,\sigma)$ denotes the set of $\sigma$-invariant
Borel probability measures on $X$, equipped with the weak$^*$ topology. 
Let $C(X)$ denote the set of continuous real-valued functions on $X$, equipped with the norm
$\|f\|_{C(X)}:=\max_{x\in X}|f(x)|$.
A
measure $\mu\in\m(X,\sigma)$ is \emph{maximizing} for $f\in C(X)$ if
$\int f\,d\mu=\max\{\int f\,d\nu:\nu\in\m(X,\sigma)\}$, and $\Mmax(f)$
denotes the set of maximizing measures for $f$.

In this paper, shift spaces will always be equipped with one of the standard metrics
\[
  d_\theta(x,y):=\theta^{\min\{i\,:\,x_i\ne y_i\}}\quad(x\ne y),
  \qquad d_\theta(x,x):=0 ,
\]
for some $\theta\in(0,1)$. 
We write $\lip(X)=\lip(X,d_\theta)$ for the Banach space of
Lipschitz functions $f:X\to\R$, with Lipschitz seminorm
$|f|_{\lip(X)}:=\sup_{x\ne y}|f(x)-f(y)|/d_\theta(x,y)$ and norm
$\|f\|_{\lip(X)}:=\|f\|_{C(X)}+|f|_{\lip(X)}$.
The associated
\emph{Wasserstein distance}
(see e.g.~\cite{villani,villani2})
between Borel probability measures $\mu,\nu$ on
$X$ is
\begin{equation}\label{eq:wass}
  W_\theta(\mu,\nu):=\sup\Bigl\{\int f\,d\mu-\int f\,d\nu:
  |f|_{\lip(X, d_\theta)}\le1\Bigr\} .
\end{equation}
Since $d_\theta^{\,\gamma}=d_{\theta^\gamma}$ for $\theta\in(0,1)$ and
$\gamma>0$, the $\gamma$-H\"older functions with respect to $d_\theta$ are
precisely the elements of $\lip(X,d_{\theta^\gamma})$, so a statement proved
for all $\theta\in(0,1)$ is a statement about all H\"older classes.  Further
notation, including Birkhoff sums $S_r$, cylinder sets $[v]_X$ and the sets
$\w_k(X)$ of words of length $k$ occurring in $X$, is collected in Section
\ref{sec:prelim}.

Let $T:Z\to Z$ be a map of a set $Z$.  A function $g:Z\to\R$ is a
\emph{coboundary} for $T$ if
\[
  g=u\circ T-u
\]
for some $u:Z\to\R$, and any such $u$ is called a \emph{transfer function}
for $g$.  If $Z$ is a compact metric space we call $g$ a \emph{continuous
coboundary} if some transfer function for $g$ is continuous, and a
\emph{Lipschitz coboundary} if some transfer function for $g$ is Lipschitz.
A continuous coboundary
$g$ for $T$ integrates to zero against all $T$-invariant Borel
probability measures, so $f$ and $f+g$ have the same maximizing measures.

We now recall the notion of typical periodic optimization.

\begin{defn}\label{def:TPO}
Let $X$ be a subshift equipped with a metric $d_\theta$.  A function
$f\in\lip(X,d_\theta)$ has the \emph{periodic optimization property} if
$\Mmax(f)=\{\mu\}$ for some periodic orbit measure $\mu$, and $\cP(X)$
will denote the set of such $f$.  The subshift $X$ has \emph{typical periodic
optimization}, or \emph{TPO}, if $\cP(X)$ is dense in $\lip(X,d_\theta)$;
equivalently, by \cite[Cor.~6.6]{HJXZ}, if there is an open dense subset of
$\lip(X,d_\theta)$ all members of which have the periodic optimization
property.
To make the metric explicit, we also say that $X$ has TPO with
respect to $\lip(X,d_\theta)$.
\end{defn}

\subsection{Sturmian beta-shifts}

We begin by recalling the maximal Sturmian sequence and the Sturmian
subshift of a given rotation number (for background on Sturmian sequences
see e.g.~\cite{MH40}, \cite[Ch.~2]{Lothaire},
\cite[Ch.~6]{Fogg}).

\begin{defn}
Let $\alpha\in(0,1)$ be irrational.  The \emph{maximal Sturmian sequence of
rotation number $\alpha$} is $\bfa=\bfa(\alpha):=(a_i)_{i\ge0}\in\Sigma$,
where
\begin{equation}\label{eq:a}
  a_i:=\lceil(i+1)\alpha\rceil-\lceil i\alpha\rceil ,
\end{equation}
so that $a_i=1$ if and only if the interval $[i\alpha,(i+1)\alpha)$ contains
an integer.  The \emph{Sturmian subshift of rotation number $\alpha$} is
defined to be the orbit closure $X_\alpha:=\overline{\{\sigma^n\bfa:n\ge0\}}$.
\end{defn}

The Sturmian subshift $X_\alpha$ is minimal, non-periodic\footnote{In some contexts it is convenient to allow $\alpha$ to be rational, in which case the corresponding Sturmian subshifts are periodic orbits, see e.g.~\cite{Bousch2000, BM02, bullettsentenac}. In the present setting, 
if the largest point of a beta-shift is periodic then the beta-shift is a shift of finite type
(see e.g.~\cite{blanchard}), and in this case TPO does hold, by \cite{Con16}. For that reason, in this paper the definition of Sturmian only allows irrational rotation number, as in e.g.~\cite{Lothaire, Fogg}.} and uniquely ergodic; we write
$\mu_\alpha$ for its unique invariant measure, the \emph{Sturmian measure of
rotation number $\alpha$}.  The sequence $\bfa$ is the largest element of
$X_\alpha$ in the lexicographic order, and satisfies
$\sigma^j\bfa<_\lex\bfa$ for all $j\ge1$, so by Parry's criterion
\cite{Par60} there is a unique
$\beta\in(1,2)$ with $\bfa=d_\beta(1)$,
where $d_\beta(1)$ denotes the greedy expansion of $1$ in base $\beta$,
i.e.~the lexicographically largest sequence $(t_i)_{i\ge0}$ of non-negative
integers
with $\sum_{i\ge0}t_i\beta^{-i-1}=1$ (see \cite{Renyi57,Par60}, and
e.g.~\cite{blanchard} and \cite[Ch.~7]{Lothaire}).
The
associated beta-shift is
\begin{equation}\label{eq:Yalpha}
  Y_\alpha:=\{x\in\Sigma:\sigma^jx\le_\lex\bfa\ \text{ for all }j\ge0\},
\end{equation}
which we call the \emph{Sturmian beta-shift of rotation number $\alpha$}.
A sequence is called \emph{Sturmian} if it belongs to $X_\alpha$ for some
irrational $\alpha\in(0,1)$.  The Sturmian beta-shifts are precisely the
beta-shifts whose lexicographically largest point is a Sturmian sequence: if
this point $x$ lies in $X_\alpha$ then, as $\sigma^jx\le_\lex x$ for all
$j\ge0$, all points of the orbit closure $X_\alpha$ of $x$ are
lexicographically at most $x$, so that $x=\bfa$.
The Sturmian beta-shift $Y_\alpha$ is the smallest beta-shift containing
$X_\alpha$, and Chi \& Kwon \cite{ChiKwon} showed that the corresponding
parameter $\beta$ is transcendental.

Since $\bfa$ is not eventually periodic ($X_\alpha$ being minimal and
infinite), $Y_\alpha$ is not sofic \cite{Bertrand86,blanchard}, and in
particular not of finite type.  On the other hand, $Y_\alpha$ has the
\emph{specification property}: there exists $N\in\N$ such that for all words $u,v$
occurring in $Y_\alpha$ there is a word $w$ of length $N$ such that $uwv$ occurs
in $Y_\alpha$ (see Lemma \ref{lem:spec}).
The properties of $X_\alpha$ and $\bfa$ recalled in this subsection, and the
existence and uniqueness of $\beta$, are established in Lemma
\ref{lem:sturmian}.

Our results for Sturmian beta-shifts involve the following functions.

\begin{defn}\label{def:functions}
Let $\alpha\in(0,1)$ be irrational. Let $X_\alpha$
be the
Sturmian subshift,
and $Y_\alpha$ the Sturmian beta-shift, of rotation number $\alpha$,
and let $\theta\in(0,1)$.  Define\footnote{These functions are defined on the whole of $\Sigma$ because the
Birkhoff sums of $\Phi$ are used on $\Sigma$ in Section \ref{sec:cob};
elsewhere we only consider their restrictions to $Y_\alpha$ and to $X_\alpha$,
which we denote by the same symbols.}
$\Phi,\Psi,F:\Sigma\to\R$ by
\begin{equation}\label{eq:functions}
  \Phi(x):=x_0,\qquad  \Psi(x)=\Psi_\theta(x):=d_\theta(x,X_\alpha),\qquad
  F:=\Phi-\Psi .
\end{equation}
We call $\Phi$ the \emph{digit function}, so that $\int\Phi\,d\nu$ is the
frequency of the digit $1$ under $\nu$; we call $\Psi$ the \emph{distance
function}, which vanishes precisely
on $X_\alpha$; and we call their difference $F$ the
\emph{digit-distance difference}.
\end{defn}

We note the following bounds on the Lipschitz seminorms of these functions.

\begin{rem}
All three functions $\Phi$, $\Psi$ and $F$ are Lipschitz: if $\ell$ is the
smallest index at which $x$ and $y$ differ, so that
$d_\theta(x,y)=\theta^{\ell}$, then $|\Phi(x)-\Phi(y)|=|x_0-y_0|$ equals $1$
when $\ell=0$ and $0$ otherwise, while $\Psi$ is a distance function and
$F=\Phi-\Psi$; so
\begin{equation}\label{eq:FLip}
  |\Phi|_{\lip(\Sigma)}=1,\qquad
  |\Psi|_{\lip(\Sigma)}\le1,\qquad
  |F|_{\lip(\Sigma)}\le2 .
\end{equation}
The same bounds hold for the restrictions of these functions to $Y_\alpha$,
with $|\Phi|_{\lip(Y_\alpha)}=1$ since $0^\infty$ and $\bfa$ both lie in
$Y_\alpha$.
\end{rem}

\subsection{Results for Sturmian beta-shifts}

Our main result for Sturmian beta-shifts is the following.

\begin{thmx}[\textbf{Failure of TPO}]\label{thm:A}
Let $\alpha\in(0,1)$ be irrational and let $\theta\in(0,1)$.  Then the
Sturmian beta-shift $Y_\alpha$, equipped with the metric $d_\theta$, does
not have typical periodic optimization.
More precisely, the digit-distance difference function $F$ belongs to $\lip(Y_\alpha,d_\theta)$, and with
$C_\theta:=1+2(1-\theta)^{-3}$ the set
\[
  U:=\bigl\{f\in\lip(Y_\alpha,d_\theta):
  |f-F|_{\lip(Y_\alpha)}<C_\theta^{-1}\bigr\}
\]
is a non-empty open subset of $\lip(Y_\alpha,d_\theta)$, and all $f\in U$
satisfy $\Mmax(f)=\{\mu_\alpha\}$, where the Sturmian measure $\mu_\alpha$ is not
a periodic orbit measure.
\end{thmx}

As Theorem \ref{thm:A} holds for all $\theta\in(0,1)$, typical periodic
optimization fails for all Sturmian beta-shifts in all H\"older classes;
moreover the radius $C_\theta^{-1}$ of the seminorm ball $U$ does not depend on
$\alpha$.  Analogous results hold for Sturmian sequences over other two-letter alphabets
(see Remark
\ref{rem:alphabets}), and for the tuned Sturmian beta-shifts introduced in
Section \ref{sec:density} (see Theorems \ref{thm:tuned-rigidity} and
\ref{thm:tuned}).

Sturmian measures are familiar as maximizing measures for the doubling map
\cite{Bousch2000}, but on the full shift $\Sigma$, the set of $f\in\lip(\Sigma,d_\theta)$ for which $\mu_\alpha$ is a
maximizing measure has empty interior
(by the Yuan--Hunt theorem
recalled in Section \ref{subsec:parameters}, since $\mu_\alpha$ is not 
periodic); the robustness of $\mu_\alpha$ in Theorem
\ref{thm:A} is therefore a consequence of the constraint defining $Y_\alpha$.

Since Sturmian beta-shifts have the specification property, Theorem
\ref{thm:A} implies:

\begin{corx}[\textbf{Specification does not imply TPO}]\label{cor:B}
There exist shift spaces with the specification property which do not have
typical periodic optimization.
\end{corx}

Theorem \ref{thm:A} is a consequence of the following result, by a
stability criterion for maximizing measures (Lemma \ref{lem:stability}).

\begin{thmx}[\textbf{Stability of the Sturmian maximizing
measure}]\label{thm:C}
Let $\alpha\in(0,1)$ be irrational, let $\theta\in(0,1)$, let $Y_\alpha$ be
the Sturmian beta-shift of rotation number $\alpha$ equipped with the metric
$d_\theta$, with Sturmian measure $\mu_\alpha$ and Wasserstein distance
$W_\theta$,
and let $F\in\lip(Y_\alpha,d_\theta)$ be
the digit-distance difference function.
 Then, with
$C_\theta:=1+2(1-\theta)^{-3}$,
\[
  \int F\,d\mu_\alpha-\int F\,d\nu\ \ge\ C_\theta^{-1}\,
  W_\theta(\nu,\mu_\alpha)\qquad\text{for all }\nu\in\m(Y_\alpha,\sigma).
\]
\end{thmx}

The proof of Theorem \ref{thm:C} uses the fact that the digit function $\Phi$ is
maximized by $\mu_\alpha$
(see Section \ref{sec:maximizing} for a quantitative form of this result),
together with the
following rigidity property of Sturmian subshifts.

\begin{thmx}[\textbf{Lipschitz cohomological rigidity for Sturmian
subshifts}]\label{thm:D}
Let $\alpha\in(0,1)$ be irrational, let $\theta\in(0,1)$, let $X_\alpha$ be
the Sturmian subshift of rotation number $\alpha$ equipped with the metric
$d_\theta$, with Sturmian measure $\mu_\alpha$, 
and let $\Phi$ be the digit function.
Then for all
$f\in\lip(X_\alpha,d_\theta)$ there is a unique constant $c_f\in\R$, and a
function $\tilde f\in\lip(X_\alpha,d_\theta)$ which is unique up to an additive
constant, such that
\begin{equation}\label{eq:cohomology}
  f-\int f\,d\mu_\alpha=c_f\,(\Phi-\alpha)+\tilde f\circ\sigma-\tilde f .
\end{equation}
Moreover,
\begin{equation}\label{eq:c-bound}
  |c_f|\le\frac{|f|_{\lip(X_\alpha)}}{(1-\theta)^{2}}
\end{equation}
and
\begin{equation}\label{eq:u-bound}
  |\tilde f|_{\lip(X_\alpha)}
  \le\frac{\theta\,|f|_{\lip(X_\alpha)}}{(1-\theta)^{3}} .
\end{equation}

Consequently the quotient of $\lip(X_\alpha,d_\theta)$ by the subspace spanned
by the constants and the Lipschitz coboundaries is one-dimensional, and is
spanned by the equivalence class of the digit function $\Phi$.  The functional
$f\mapsto c_f$ is the unique linear functional on $\lip(X_\alpha,d_\theta)$
which vanishes on the constants and on the Lipschitz coboundaries and
satisfies $c_\Phi=1$.
\end{thmx}

The one-dimensionality in Theorem \ref{thm:D} is not specific to Sturmian
subshifts: Theorem \ref{thm:tuned-rigidity} establishes it for the tuned
Sturmian subshifts of Section \ref{sec:density}, which in general are not
topologically conjugate to Sturmian subshifts (Proposition \ref{prop:tuned-not-sturmian}).

\begin{rem}\label{rem:uniqueness}
The uniqueness of $c_f$ in Theorem \ref{thm:D} is equivalent, as is readily
checked, to the fact that $\Phi-\alpha$ is not a Lipschitz coboundary on
$X_\alpha$.  It is, however, a continuous coboundary there: by Proposition
\ref{prop:rigidity}, the function $\tau:=\sup_{r\ge0}S_r(\Phi-\alpha)$ is
continuous, takes values in $[0,1]$, and satisfies
$\Phi-\alpha=\tau-\tau\circ\sigma$, but is not H\"older continuous, since for
all $k\in\N$ it maps one of the $k+1$ 
length-$k$ cylinder sets in
$X_\alpha$, which have $d_\theta$-diameter at most $\theta^k$, onto an
interval of length at least $(k+1)^{-1}$.  On a hyperbolic system this could
not happen, since by Liv\v{s}ic's theorem \cite{Liv71,Livsic} a H\"older
function which is a continuous coboundary is a H\"older coboundary.
Corollary \ref{cor:livsic} describes how this fails on $X_\alpha$: if
$f\in\lip(X_\alpha,d_\theta)$ satisfies $\int f\,d\mu_\alpha=0$, then the
continuous transfer functions for $f$ are the functions $\tilde f-c_f\tau+c$,
$c\in\R$, so the failure of Liv\v{s}ic regularity is confined to the single
direction $\R\tau$, along which the transfer function has coefficient $-c_f$.
\end{rem}

\subsection{Sturmian beta-transformations, and a transfer
principle}\label{subsec:transformations}

For $\beta>1$, the \emph{beta-transformation} $T_\beta$ and the \emph{upper
beta-transformation} $U_\beta$ of the interval $I:=[0,1]$ are defined by
$T_\beta(t):=\beta t-\lfloor\beta t\rfloor$ for $t\in I$, and $U_\beta(0):=0$,
$U_\beta(t):=\beta t-\lceil\beta t\rceil+1$ for $t\in(0,1]$ (see
e.g.~\cite{HHJL}).  For $\gamma\in(0,1]$ we write $C^{0,\gamma}(I)$ for
the Banach space of $\gamma$-H\"older functions $g:I\to\R$, with seminorm
$|g|_\gamma:=\sup_{s\ne t}|g(s)-g(t)|/|s-t|^\gamma$ and norm
$\|g\|_\gamma:=\|g\|_{C(I)}+|g|_\gamma$.  For a Borel map $T:I\to I$ and
$g\in C(I)$, we write $\m(I,T)$ for the set of $T$-invariant Borel
probability measures on $I$, and $\Mmax(T,g)$ for the set of those
$\mu\in\m(I,T)$ with $\int g\,d\mu=\sup\{\int g\,d\nu:\nu\in\m(I,T)\}$; as
$T$ need not be continuous, this set may be empty.  The notions of Definition
\ref{def:TPO} are adapted to this setting as in \cite{HHJL}: $(T,g)$ has the
periodic optimization property if $\Mmax(T,g)$ consists of a single periodic
orbit measure, and $T$ has TPO with respect to $C^{0,\gamma}(I)$ if some open
dense subset of $C^{0,\gamma}(I)$ consists of functions $g$ for which $(T,g)$
has the periodic optimization property.  When $\gamma$ is clear from the
context, we simply say that $T$ has TPO.

Recall that $\beta>1$ is a \emph{beta-number} if $d_\beta(1)$ is eventually
periodic, and a \emph{simple beta-number} if $d_\beta(1)$ is eventually
zero.  
For $\beta>1$ we write $d^*_\beta(1)$ for the quasi-greedy expansion
of $1$ in base $\beta$, i.e.~the lexicographically largest sequence
$(t_i)_{i\ge0}$ of non-negative integers with infinitely many non-zero
terms such that
$\sum_{i\ge0}t_i\beta^{-i-1}=1$ (see e.g.~\cite{Par60} and
\cite[Ch.~7]{Lothaire}); it equals $d_\beta(1)$ unless $\beta$ is a
simple beta-number, in which case it is periodic (cf.~Remark
\ref{rem:general}), and by \cite[Prop.~3.9(i)]{HHJL} it is the upper
$\beta$-expansion of $1$ in the sense of \cite[Def.~3.2]{HHJL}, with
first term $\lceil\beta\rceil-1$.
Let
\[
  S_\beta:=\{x\in\Sigma_{\lceil\beta\rceil-1}:\sigma^jx\le_\lex d^*_\beta(1)\ \text{ for all }j\ge0\}
\]
be the beta-shift of $\beta$, so that $S_\beta=Y_\alpha$ when
$d_\beta(1)=\bfa$, and define $h_\beta:\Sigma_{\lceil\beta\rceil-1}\to\R$ by
$h_\beta(x):=\sum_{i\ge0}x_i\beta^{-i-1}$.

It is proved in \cite{HHJL} that $U_\beta$ has TPO with respect to
$C^{0,\gamma}(I)$, for all $\gamma\in(0,1]$, whenever $\beta>1$ is a
beta-number, and that both $T_\beta$ and $U_\beta$ have TPO whenever $\beta$ is
\emph{non-emergent} 
\cite[Thms.~I, J]{HHJL}, where, by \cite[Prop.~5.3]{HHJL}, $\beta>1$ is
emergent if and only if the orbit closure of $d^*_\beta(1)$ under the shift
is minimal.
For emergent $\beta$ they showed that
the only possible obstruction to TPO is the \emph{critical set}
$\mathrm{Crit}^\gamma(\beta)$, consisting of those functions for which the
$U_\beta$-orbit of $1$ is maximizing \cite[Def.~5.8, Thm.~5.22]{HHJL},
leaving open the question of TPO for emergent $\beta$ which are not
beta-numbers.  The Sturmian parameters are of this kind: if
$d_\beta(1)=\bfa$ is the maximal Sturmian sequence of irrational rotation
number $\alpha$, then $d^*_\beta(1)=\bfa$, its orbit closure $X_\alpha$ is
minimal, and $\bfa$ is not eventually periodic, so $\beta$ is emergent and
not a beta-number.  The emergent parameter of \cite[Ex.~5.5]{HHJL}, for
which $d_\beta(1)$ is the digit $1$ followed by the Fibonacci word, is the
case $\alpha=(3-\sqrt5)/2$.

Our main result for Sturmian beta-transformations is the following.

\begin{thmx}[\textbf{Failure of TPO for Sturmian
beta-transformations}]\label{thm:E}
Let $\alpha\in(0,1)$ be irrational, let $\beta\in(1,2)$ be such that
$d_\beta(1)$ is the maximal Sturmian sequence of rotation number $\alpha$,
and let $\gamma\in(0,1]$.  Then neither $T_\beta$ nor $U_\beta$ has typical
periodic optimization with respect to $C^{0,\gamma}(I)$.  More precisely,
let $\Lambda_\alpha$ be the closure of the $U_\beta$-orbit of $1$, and define
$\widehat F\in C^{0,\gamma}(I)$ by
\[
  \widehat F(t):=(\beta-1)t-\operatorname{dist}(t,\Lambda_\alpha)^\gamma .
\]
Then there is exactly one $\widehat\mu_\alpha\in\m(I,U_\beta)$ with
$\widehat\mu_\alpha(\Lambda_\alpha)=1$, namely
$\widehat\mu_\alpha=\mu_\alpha\circ h_\beta^{-1}$, it is not a periodic orbit
measure,
and there exists $\delta>0$ such that all $\psi\in C^{0,\gamma}(I)$ with
$|\psi-\widehat F|_\gamma<\delta$ satisfy
\[
  \Mmax(U_\beta,\psi)=\Mmax(T_\beta,\psi)=\{\widehat\mu_\alpha\} .
\]
\end{thmx}

Simple beta-numbers are emergent \cite[Cor.~5.4]{HHJL}, so by
\cite[Thm.~I]{HHJL} and Theorem \ref{thm:E}, both TPO and non-TPO occur for
$U_\beta$ among emergent parameters.

The proof of Theorem \ref{thm:E} (see Section \ref{sec:interval}) leads to
a general transfer principle between beta-shifts and beta-transformations.
If $\beta>1$ is not a simple beta-number, $\gamma\in(0,1]$ and
$\theta:=\beta^{-\gamma}$, then $g\mapsto g\circ h_\beta$ is an injective
bounded linear map from $C^{0,\gamma}(I)$ to $\lip(S_\beta,d_\theta)$, and
the maximizing measures of $g$ for $U_\beta$ are the images under $h_\beta$
of the maximizing measures of $g\circ h_\beta$ for $\sigma:S_\beta\to
S_\beta$ (see Section \ref{subsec:coding}).  However, the image of this map
has empty interior (see Section \ref{subsec:transfer}), so neither TPO for
$S_\beta$ with respect to $\lip(S_\beta,d_\theta)$, nor TPO for $U_\beta$
with respect to $C^{0,\gamma}(I)$, immediately implies the other. Nevertheless the
following holds.

\begin{thmx}\label{thm:F}
Suppose $\beta>1$ is not a simple beta-number. Let $\gamma\in(0,1]$ and
$\theta:=\beta^{-\gamma}$.  Let $K\subseteq S_\beta$ be a non-empty closed
set with $\sigma(K)\subseteq K$, such that for some $N\in\N$ the word $0^N$
does not occur in any element of $K$, and let $\Lambda:=h_\beta(K)$.  Then the
following are equivalent:
\begin{enumerate}
\item[(a)] some non-empty open $\cV\subseteq\lip(S_\beta,d_\theta)$ is such
that each $f\in\cV$ has a maximizing measure $\mu$ with $\mu(K)=1$;
\item[(b)] some non-empty open $\cV'\subseteq C^{0,\gamma}(I)$ is such that
each $g\in\cV'$ has some $\nu\in\Mmax(U_\beta,g)$ with $\nu(\Lambda)=1$.
\end{enumerate}
Moreover, $\cV$ and $\cV'$ can then be chosen such that all maximizing
measures of all $f\in\cV$ give full measure to $K$, and all members of
$\Mmax(U_\beta,g)=\Mmax(T_\beta,g)$, $g\in\cV'$, give full measure to
$\Lambda$.  If $K$ contains no periodic orbit and (a) or (b) holds, then
$S_\beta$ does not have TPO with respect to $\lip(S_\beta,d_\theta)$, and
neither $U_\beta$ nor $T_\beta$ has TPO with respect to $C^{0,\gamma}(I)$.
\end{thmx}

We shall apply Theorem \ref{thm:F} with $K$ equal to the Markov boundary
of $S_\beta$, a notion introduced by Thomsen \cite{Thomsen06}.  For a
subshift $X$ and $w\in\w(X):=\bigcup_{n\ge0}\w_n(X)$, let
$F_X(w):=\{v\in\w(X):wv\in\w(X)\}$ denote the \emph{follower set} of $w$.
The \emph{Markov boundary} $\partial_MX$ of $X$ is the set of those
$x\in X$ such that, for every word $w$ occurring in $x$, the set
$\{F_X(uw):uw\in\w(X)\}$ is infinite (see \cite[Def.~8.7]{HJXZ}).  Roughly
speaking, along points of $\partial_MX$ the admissible continuations of a
word are never determined by a bounded portion of what precedes it.  The
Markov boundary is a closed shift-invariant subset of $X$, which is empty
if and only if $X$ is sofic \cite[Lem.~2.1, Thm.~2.3]{Thomsen06}.
Moreover, the only possible obstruction to TPO for $X$ is located in
$\partial_MX$: an open dense subset of $\lip(X)$ 
consists of functions
which either have the periodic optimization property, or have a maximizing
measure supported by $\partial_MX$ \cite[Thm.~9.13]{HJXZ}, so that $X$ has TPO whenever
$\partial_MX$ has TPO \cite[Thm.~1.3]{HJXZ}.

A beta-shift $S_\beta$ is sofic if and only if $\beta$ is a beta-number
\cite{Bertrand86,blanchard}, and otherwise $\partial_MS_\beta$ is the
$\omega$-limit set of $d^*_\beta(1)$, by Lemma \ref{lem:markov-boundary},
a version of a theorem of Thomsen \cite[Thm.~2.8]{Thomsen05} on
two-sided beta-shifts which are not of finite type.\footnote{Thomsen's
theorem identifies the natural extension of the $\omega$-limit set of
$d^*_\beta(1)$ with the set of points in which no synchronizing word occurs
(the \emph{derived shift space}), a word $w\in\w(X)$ being
\emph{synchronizing} if $F_X(uw)=F_X(w)$ whenever $uw\in\w(X)$.  This set
contains the Markov boundary, and for sofic beta-shifts which are not of
finite type it is a periodic orbit \cite[Prop.~2.6]{Thomsen05}, whereas the
Markov boundary is empty.}  Applying Theorem \ref{thm:F} with
$K=\partial_MS_\beta$, and combining with Theorem \ref{thm:nonemergent}
below, leads to Corollary \ref{cor:G}, which shows that TPO for the
beta-shift $S_\beta$ is equivalent to TPO for $U_\beta$, and, unless
$\beta$ is a simple beta-number, to TPO for $T_\beta$.

\subsection{TPO for general parameters}\label{subsec:general}

By \cite[Thm.~J]{HHJL}, $T_\beta$ and $U_\beta$ have TPO with respect to
$C^{0,\gamma}(I)$, for all $\gamma\in(0,1]$, whenever $\beta>1$ is
non-emergent, and the set of emergent parameters has zero Lebesgue measure
and is meagre \cite[Cor.~5.7]{HHJL}.  The analogous statement holds for
beta-shifts:

\begin{thmx}[\textbf{TPO for non-emergent beta-shifts}]\label{thm:nonemergent}
Let $\beta>1$ be non-emergent.  Then the beta-shift $S_\beta$ has TPO
with respect to $\lip(S_\beta,d_\theta)$, for all $\theta\in(0,1)$.
\end{thmx}

Theorem \ref{thm:nonemergent} shows in particular that the set of
$\beta>1$ such that $S_\beta$ has TPO with respect to
$\lip(S_\beta,d_\theta)$ for all $\theta\in(0,1)$ has full Lebesgue measure
and is residual.  Its proof (see Section \ref{subsec:nonemergent}) follows
the strategy used in \cite{HHJL} for $U_\beta$ (see \cite[Lem.~5.23]{HHJL}),
namely to find maximizing measures supported by suitable subsystems; for
beta-shifts these subsystems can be taken to be the beta-shifts
$S_{\beta'}$ of simple beta-numbers $\beta'<\beta$, which are shifts of
finite type, so that the structural theory of \cite{HJXZ} applies directly.

Combining Theorem \ref{thm:nonemergent} with the transfer principle of
Theorem \ref{thm:F}, we obtain the following comparison between typical
periodic optimization for beta-shifts and  beta-transformations.

\begin{corx}[\textbf{TPO for beta-shifts and beta-transformations}]
\label{cor:G}
Let $\beta>1$ and $\gamma\in(0,1]$, and equip the beta-shift
$S_\beta$ with the metric $d_{\beta^{-\gamma}}$.  Then $S_\beta$ has TPO if
and only if $U_\beta$ has TPO with respect to $C^{0,\gamma}(I)$.  If $\beta$
is not a simple beta-number, then these properties are also equivalent to
$T_\beta$ having TPO with respect to $C^{0,\gamma}(I)$; if $\beta$ is a
simple beta-number, then $S_\beta$ and $U_\beta$ have TPO, whereas $T_\beta$
does not.
\end{corx}

For non-emergent $\beta$ both $S_\beta$ and $U_\beta$ have TPO, by Theorem
\ref{thm:nonemergent} and \cite[Thm.~J]{HHJL}, so the substance of
Corollary \ref{cor:G} lies in the emergent case.  If $\beta$ is emergent
and not a beta-number, then the Markov boundary of $S_\beta$ is the orbit
closure of $d_\beta(1)$, and the structural theorems of \cite{HHJL} and
\cite{HJXZ} recalled in Section \ref{subsec:transformations} identify it as
the only possible obstruction to TPO in each of the two settings; Theorem
\ref{thm:F}, applied with $K$ equal to this Markov boundary, then shows that
TPO fails in one setting if and only if it fails in the other (see Section
\ref{subsec:transfer}).

\subsection{Parameters without TPO, and zero-entropy optimization}\label{subsec:parameters}

By Theorems \ref{thm:A} and \ref{thm:E}, the parameters $\beta>1$ such
that $d_\beta(1)$ is a Sturmian sequence form an uncountable set (one
$\beta$ for each irrational $\alpha\in(0,1)$), for which none of $S_\beta$,
$T_\beta$ and $U_\beta$ has TPO.  The following result shows that,
nevertheless, in each of these three settings the parameters without TPO
form a set of zero packing dimension.

\begin{thmx}[\textbf{Parameters without TPO}]\label{thm:dimension}
Each of the following three sets has packing dimension zero, and hence
Hausdorff dimension zero:
\begin{enumerate}
\item[(i)] the set of $\beta>1$ such that, for some $\theta\in(0,1)$, the beta-shift
$S_\beta$ does not have TPO with respect to $\lip(S_\beta,d_\theta)$;
\item[(ii)] the set of $\beta>1$ such that, for some
$\gamma\in(0,1]$, the upper beta-transformation $U_\beta$ does not have TPO with respect to
$C^{0,\gamma}(I)$;
\item[(iii)] the set of $\beta>1$ such that, for some
$\gamma\in(0,1]$, the beta-transformation $T_\beta$ does not have TPO with respect to
$C^{0,\gamma}(I)$.
\end{enumerate}
\end{thmx}

Part (iii) of Theorem \ref{thm:dimension} sharpens the
result of \cite[Thm.~E]{HHJL} that $T_\beta$ has TPO for Lebesgue almost
every $\beta>1$.  By Theorem \ref{thm:nonemergent} and \cite[Thm.~J]{HHJL},
the three sets in Theorem \ref{thm:dimension} consist of emergent
parameters.
Theorem
\ref{thm:dimension} is proved by combining Theorem \ref{thm:nonemergent}, Corollary
\ref{cor:G}, the structural results of \cite{HHJL, HJXZ},
and the following theorem.

\begin{thmx}[\textbf{Sub-exponential complexity of robustly maximizing
measures}]\label{thm:robust}
Let $\beta>1$ and $\theta\in(0,1)$, and let $\mu\in\m(S_\beta,\sigma)$.
If the set of $f\in\lip(S_\beta,d_\theta)$ for which $\mu$ is a maximizing
measure has non-empty interior in $\lip(S_\beta,d_\theta)$, then there exists
$C>0$ such that
\[
  \bigl|\w_n(\mathrm{supp}(\mu))\bigr|\le\exp\bigl(C\sqrt{n\log n}\,\bigr)
  \qquad\text{for all }n\ge2 ,
\]
and in particular $\mathrm{supp}(\mu)$ has zero topological entropy.
\end{thmx}

For uniformly hyperbolic systems, Yuan \& Hunt \cite[Thm.~4.6]{YH99} showed
that if an orbit which is not periodic, and which is measure-recurrent (that
is, it generates an invariant measure whose support contains it), is optimal
for a Lipschitz function, then there are arbitrarily small Lipschitz
perturbations of this function for which it is no longer optimal.  In terms
of invariant
measures (see Bousch \cite{Bousch2008}, who gave a new proof of this
theorem), the result is that if the set of Lipschitz functions for which a
given invariant measure is maximizing has non-empty interior, then this
measure is periodic.  By Theorem \ref{thm:A}, this
conclusion fails for beta-shifts, even those with specification, and Theorem \ref{thm:robust} shows that a weaker conclusion
persists: the support of such a measure has zero topological entropy.  In
fact such a measure is either periodic, or the unique
invariant measure on the orbit closure of $d_\beta(1)$, which is then
minimal, so that $\beta$ is emergent (see Remark \ref{rem:robust}).
Theorem \ref{thm:robust} may be compared with the suggestion of Hunt \&
Ott \cite{HO96,HO96b} recalled in the Introduction (cf.~\cite[p.~1208]{YH99}).

We shall say that a dynamical system has \emph{typical zero-entropy
optimization} with respect to a Banach space of real-valued functions if
some open dense subset of this space consists of functions having a unique
maximizing measure $\mu$, such that all invariant measures supported by
$\mathrm{supp}(\mu)$ have zero entropy.  
For a continuous map, and in particular
for a subshift, this condition on $\mu$ means that $\mathrm{supp}(\mu)$ has
zero topological entropy, by the variational principle \cite[Ch.~8]{Wal82}.
As the only invariant measure supported by a periodic orbit
is the corresponding periodic orbit measure, which has zero entropy, TPO
implies typical zero-entropy optimization.  The following result shows that
typical zero-entropy optimization holds for all $\beta>1$, including those
for which TPO fails. 

\begin{thmx}[\textbf{Typical zero-entropy optimization}]\label{thm:TZO}
Let $\beta>1$.  Then the beta-shift $S_\beta$ has typical zero-entropy optimization
with respect to $\lip(S_\beta,d_\theta)$, for all $\theta\in(0,1)$, and the upper beta-transformation
$U_\beta$ has typical zero-entropy optimization with respect to
$C^{0,\gamma}(I)$, for all $\gamma\in(0,1]$, as does the beta-transformation $T_\beta$ unless $\beta$
is a simple beta-number.
\end{thmx}

The exception in Theorem \ref{thm:TZO} is necessary: if $\beta$ is a simple
beta-number then, for all $\gamma\in(0,1]$, some non-empty open subset of
$C^{0,\gamma}(I)$ consists of functions having no maximizing measure for
$T_\beta$ \cite[Rem.~5.21]{HHJL}.  The proof of Theorem \ref{thm:TZO} shows
moreover that the open dense set can be chosen so that the unique maximizing
measure of each of its members is either a periodic orbit measure, or the
unique invariant measure on the orbit closure of $d_\beta(1)$ under
$\sigma$, respectively its image under $h_\beta$, which is the unique
invariant measure on the closure of the orbit of $1$ under $U_\beta$.  The
second case is needed only if TPO fails, and then the orbit closure of
$d_\beta(1)$ satisfies the complexity bound of Theorem \ref{thm:robust} (see
Lemma \ref{lem:emergent} and Remark \ref{rem:robust}).

Although the parameters without TPO form a set of packing dimension zero,
they are dense in $(1,\infty)$:

\begin{thmx}[\textbf{Density of parameters without TPO}]\label{thm:density}
Let $J\subseteq(1,\infty)$ be a non-empty open interval.  Then $J$ contains
uncountably many $\beta$ such that, for all $\theta\in(0,1)$, the beta-shift
$S_\beta$ does not have TPO with respect to $\lip(S_\beta,d_\theta)$, and,
for all $\gamma\in(0,1]$, neither $T_\beta$ nor $U_\beta$ has TPO with
respect to $C^{0,\gamma}(I)$.  In particular, each of the three sets in
Theorem \ref{thm:dimension} is dense in $(1,\infty)$.
\end{thmx}

As these three sets consist of emergent parameters, which form a meagre
set \cite[Cor.~5.7]{HHJL}, each of them is a dense meagre subset of
$(1,\infty)$ of packing dimension zero.  For $T_\beta$, the density of the
set in Theorem \ref{thm:dimension}(iii) also follows from the density of the
simple beta-numbers \cite{Par60}, for which $T_\beta$ does not have TPO
\cite[Rem.~5.21]{HHJL}, whereas the parameters in Theorem \ref{thm:density}
are not beta-numbers: they are \emph{tuned Sturmian parameters}, in the
sense of Definition \ref{def:tuned}, obtained from maximal Sturmian sequences
by a version, for beta-shifts, of Douady--Hubbard tuning of external angles
of the Mandelbrot set (see \cite{DH85,Dou86,bullettsentenac}), and Theorem
\ref{thm:tuned} shows that their beta-shifts, the \emph{tuned Sturmian
beta-shifts}, do not have TPO.  By contrast with Theorem \ref{thm:density},
the set of $\beta$ such that $d_\beta(1)$ is a Sturmian sequence has nowhere
dense closure (see Remark \ref{rem:sturmian-closure}).

So, for beta-shifts, the conjecture of Bochi \cite{bochi_slides} recalled in
the Introduction holds for all $\beta>1$ in its zero-entropy form (Theorem
\ref{thm:TZO}), but fails in its periodic form for a dense set of parameters
$\beta$, of packing dimension zero (Theorems \ref{thm:dimension} and
\ref{thm:density}).

\subsection{The proof strategy}

We now outline the proofs of our results, beginning with those for
Sturmian beta-shifts.

The key observation underlying Theorem \ref{thm:D} is that each cylinder
set $C=[v]_{X_\alpha}$ of $X_\alpha$ is coded by an arc $I_v$ of the circle
whose endpoints lie in $\Z\alpha+\Z$, and whose length is $\mu_\alpha(C)$.
The centred\footnote{The terminology \emph{centred} reflects the fact that
subtracting the constant $\mu_\alpha(C)$ ensures the vanishing of the
integral of the function with respect to $\mu_\alpha$.} indicator function
$\ind_{I_v}-|I_v|$ of that arc is therefore a coboundary for the rotation
$R_\alpha t:=t+\alpha$, with an explicit piecewise affine transfer function
(Lemma \ref{lem:coboundary}).
Composing with the Sturmian coding (Definition \ref{def:coding})
carries this identity
from the circle to $X_\alpha$, where it says that the centred indicator
function $\ind_C-\mu_\alpha(C)$ is an integer multiple of $\Phi-\alpha$,
modulo a coboundary whose transfer function is, up to sign, the sum of the
Birkhoff sums $S_r(\Phi-\alpha)$ over a range of consecutive values of $r$
(Proposition \ref{prop:centred-cylinder}).

Theorem \ref{thm:D} follows by approximating an arbitrary
$f\in\lip(X_\alpha,d_\theta)$ by the functions $f_n$ that are constant on
each cylinder set of length $n$, and then writing $f-\int f\,d\mu_\alpha$
as the sum of the successive differences $f_{n+1}-f_n$.  Here the Sturmian
hypothesis is used, to exploit the fact that $X_\alpha$ has
precisely $k+1$ words of length $k$: passing from length $n$ to length
$n+1$ therefore subdivides just one cylinder set,
so that, for a suitable choice of the values of $f_n$, the difference
$f_{n+1}-f_n$ is a constant multiple of the indicator function of a single
cylinder set, and this constant is $O(\theta^{n})$ since $f$ is Lipschitz.
The theorem follows by
applying Proposition
\ref{prop:centred-cylinder} to each term and summing.

The proof of Theorem \ref{thm:C} from Theorem \ref{thm:D} (see Section
\ref{sec:proofs}) exploits the particular form of the function
$F=\Phi-\Psi$.  To bound $\int f\,d\nu-\int f\,d\mu_\alpha$ for a Lipschitz
function $f$ on $Y_\alpha$, we subtract from $f$ the decomposition of its
restriction to $X_\alpha$ given by Theorem \ref{thm:D}, after extending
$\tilde f$ to a Lipschitz function on $Y_\alpha$.  The term
$c_f(\Phi-\alpha)$ of this decomposition is absorbed by the term $\Phi$ in
$F$, since $\int\Phi\,d\nu\le\alpha$ for all $\nu\in\m(Y_\alpha,\sigma)$;
this maximizing property of $\mu_\alpha$ is readily proved by counting the
digits $1$ along the first return loops of a graph presentation of
$Y_\alpha$ (Section \ref{sec:maximizing}).  The remainder vanishes on
$X_\alpha$, so is bounded by a multiple of $\Psi$, and is therefore
absorbed by the term $-\Psi$ in $F$.
The proofs of Theorems \ref{thm:A} and \ref{thm:C} use only the existence
statements of Theorem \ref{thm:D}, together with the bounds
\eqref{eq:c-bound} and \eqref{eq:u-bound}.  The argument is formulated
abstractly in Lemma \ref{lem:rigid-stability}, which is used again for the
tuned Sturmian beta-shifts of Section \ref{sec:density}.

Theorem \ref{thm:E} is deduced from Theorem \ref{thm:C} by means of the map
$h_\beta$, which maps $Y_\alpha$ onto $I$ and induces a bijection between the
invariant measures of $\sigma$ and of $U_\beta$ (see Section
\ref{subsec:coding}).  This coding does not preserve distances to $X_\alpha$: for
example $10^\infty\notin X_\alpha$, whereas
$h_\beta(10^\infty)=h_\beta(0\bfa)$ lies in $h_\beta(X_\alpha)$.
However $\sigma(10^\infty)=0^\infty$, and $h_\beta(0^\infty)=0$ is far from
$h_\beta(X_\alpha)$; in general, the distance from $x$ to $X_\alpha$ is
bounded, on average with respect to invariant measures, by a multiple of a
power of the distance from $h_\beta(x)$ to $h_\beta(X_\alpha)$ (Lemma
\ref{lem:key}).
The same inequality, combined with a H\"older extension
argument, 
gives Theorem \ref{thm:F}, which together with the structural results of
\cite{HHJL} and \cite{HJXZ} gives Corollary \ref{cor:G} for emergent
$\beta$; for non-emergent $\beta$, Corollary \ref{cor:G} follows from
Theorem \ref{thm:nonemergent}.

Theorems \ref{thm:nonemergent}, \ref{thm:dimension}, \ref{thm:robust} and
\ref{thm:TZO} are proved in Section \ref{sec:dimension}.
For $g\in\lip(S_\beta,d_\theta)$ with maximizing
measure $\mu$, and $x\in S_\beta$, let $u_g(x)$ be the supremum of the
Birkhoff sums $S_j(g-\int g\,d\mu)(z)$ over all $j\ge0$ and $z\in S_\beta$
with $\sigma^jz=x$.  
A word which can precede a point $y\in S_\beta$ can also precede all points
of $S_\beta$ which are lexicographically smaller than $y$, and this gives a
one-sided Lipschitz inequality for $u_g$ with respect to the lexicographic
order (Lemma \ref{lem:subaction}(iii)).  Moreover, if $vb0^\infty\in S_\beta$
for a word $v$ and a digit $b\ge1$, then the periodic point
$(v(b-1))^\infty$ also belongs to
$S_\beta$, and together with the one-sided inequality this implies that
$u_g$ is bounded; it satisfies $g-\int g\,d\mu\le u_g\circ\sigma-u_g$, with
equality $\mu$-almost everywhere (Lemma \ref{lem:subaction}).
By a theorem of Morris
\cite{Mor07}, the boundedness of $u_g$ implies that $g$ has a maximizing
measure supported by a minimal set.  If $\beta$ is non-emergent, such a
minimal set does not contain $d^*_\beta(1)$, so is contained in one of the
shifts of finite type $S_{\beta'}$, $\beta'<\beta$, and Theorem
\ref{thm:nonemergent} follows from the structural theory of \cite{HJXZ}.

Theorem \ref{thm:robust} is proved by a counting argument.  If $\mu$ is a
maximizing measure for all functions in a non-empty open set containing $f$,
then $\mu$ is also a maximizing measure for the functions $g$ obtained by
adding to $f$ small multiples of the indicator functions of cylinder sets of
length $k$.  
By the one-sided inequality of Lemma \ref{lem:subaction}(iii), the functions
$u_g$ are approximately non-increasing, with respect to the lexicographic
order, on small cylinder sets; this allows the support of $\mu$ to be
partitioned into a number of sets which grows at most exponentially in $k$,
on each of which all of the functions $u_g$ are approximately constant.
Along the orbits of $\mu$-almost all points, the Birkhoff sums of
$g-\int g\,d\mu$ are differences of values of $u_g$, and it follows that the
words of length $n$ occurring in the support of $\mu$ fall into a number of
classes, which again grows at most exponentially in $k$, such that the words
in each class contain each word of length $k$ equally often.
Each class
therefore has at most $n^{n/k}\lceil\beta\rceil^k$ elements, as its words
have a common set of at most $n$ subwords of length $k$, and choosing $k$ of
order $\sqrt{n\log n}$ gives Theorem \ref{thm:robust}.
For emergent $\beta$ such that $S_\beta$ does not have TPO, the structural
results of \cite{HJXZ} show that the orbit closure of $d_\beta(1)$, which is
the Markov boundary of $S_\beta$, is uniquely ergodic, and that its invariant
measure is a maximizing measure for all functions in a non-empty open set;
so this orbit closure has zero topological entropy, by Theorem
\ref{thm:robust}.  The parameters $\beta$ with this property form a set of
packing dimension zero (Lemma \ref{lem:zerodim}), and Theorem
\ref{thm:dimension} follows from Theorem \ref{thm:nonemergent}, using
Corollary \ref{cor:G} for $T_\beta$ and $U_\beta$.  Theorem \ref{thm:TZO} is
deduced using the same ingredients, together with the structural theorem
\cite[Thm.~5.22]{HHJL} for $U_\beta$.

Theorem \ref{thm:density} is proved in Section \ref{sec:density} by a
tuning construction, a version for beta-shifts of Douady--Hubbard tuning of
external angles of the Mandelbrot set \cite{Dou86,DH85} (see also
Bullett \& Sentenac 
\cite[Part~2, \S 2]{bullettsentenac}).
If $\beta_1$ is a simple beta-number, $w$ is the shortest word such that
$d^*_{\beta_1}(1)=w^\infty$, and $p$ is the length of $w$, then replacing
each digit $1$ of the maximal Sturmian sequence $\bfa$ of rotation number
$\alpha$ by $w$, and each digit $0$ by the word $w^-$ obtained by
decreasing the last non-zero digit of $w$ by one, gives the expansion
$d_\beta(1)$ of a tuned Sturmian parameter $\beta<\beta_1$, which tends to
$\beta_1$ as $\alpha$ tends to $1$; as simple beta-numbers are dense in
$(1,\infty)$ \cite{Par60}, so are the tuned Sturmian parameters.  The orbit
closure of $d_\beta(1)$, the tuned Sturmian subshift, is built from $p$
copies of the Sturmian subshift $X_\alpha$, and Theorem \ref{thm:D} extends
to it, with the indicator function of the cylinder set $[w]$ in place of the
digit function (Theorem \ref{thm:tuned-rigidity}). 
Along orbit
segments of $S_\beta$ which stay close to the tuned Sturmian subshift, the
blocks $w$ and $w^-$ follow each other
in the same manner
as the digits $1$ and $0$ do in
points of the Sturmian beta-shift $Y_\alpha$, which gives a version of
Proposition \ref{prop:Phimax} (Proposition \ref{prop:tuned-max}).
So
the argument for Theorem \ref{thm:C} shows that the invariant measure on the
tuned Sturmian subshift is the unique maximizing measure for all functions
in a non-empty open set (Theorem \ref{thm:tuned}).

\subsection{Organisation of the paper}
Section \ref{sec:prelim} records further notation, the criterion for a
maximizing measure to persist under small Lipschitz perturbations, and some
preliminaries concerning Sturmian sequences and the shifts $X_\alpha$ and
$Y_\alpha$. Section \ref{sec:maximizing} describes
the graph presentation of $Y_\alpha$,
identifies the Markov boundary of a non-sofic beta-shift,
and proves that $\mu_\alpha$ maximizes
the digit function.  Section \ref{sec:cob} establishes the coboundary
identity for cylinder sets, first on the circle and then on $X_\alpha$.
Section \ref{sec:cohomology} proves Theorem \ref{thm:D} and the
accompanying rigidity statement. Section \ref{sec:proofs} proves
Theorems \ref{thm:A} and \ref{thm:C}, and Corollary \ref{cor:B}.  Section
\ref{sec:interval} proves Theorems \ref{thm:E} and \ref{thm:F}, and Corollary
\ref{cor:G}; it also indicates how Theorems \ref{thm:A}, \ref{thm:C},
\ref{thm:D} and \ref{thm:E} extend to Sturmian sequences 
over other two-letter alphabets
(Remark \ref{rem:alphabets}).
Section \ref{sec:dimension} proves Theorems
\ref{thm:nonemergent}, \ref{thm:dimension}, \ref{thm:robust} and
\ref{thm:TZO}.  Section \ref{sec:density} introduces tuned Sturmian
beta-shifts, proves analogues of Theorems \ref{thm:A}, \ref{thm:C} and
\ref{thm:D} for them (Theorems \ref{thm:tuned-rigidity} and
\ref{thm:tuned}), and deduces Theorem \ref{thm:density}.

\subsection*{Acknowledgments} {W. Huang is partially supported by the National Key R\&D Program of China (Nos. 2024YFA1013602, 2024YFA1013600) and the National Natural Science Foundation of China grant (No. 12426201). L. Xu is partially supported by the National Key R\&D Program of China (Nos. 2024YFA1013602, 2024YFA1013600) and the National Natural Science Foundation (Nos. 12371197, 12426201). Y. Zhang is partially supported by the National Natural Science Foundation (Nos.~12671216, 12271432), and USTC-AUST Mathematics Basic Discipline Research Center.}

\section{Preliminaries}\label{sec:prelim}

In this section we collect further notation, a criterion for an invariant
measure to remain the unique maximizing measure under small Lipschitz
perturbations, and basic properties of Sturmian sequences and of the shifts
$X_\alpha$ and $Y_\alpha$.

\subsection{Further notation}

The following notation is used throughout the paper.

\begin{notation}
For $m\in\N$, $x\in\Sigma_m$ and $0\le r\le s$ we write $x_{[r,s)}:=x_rx_{r+1}\cdots
x_{s-1}$, defining this to be the empty word when $r=s$.  For a finite word
$v=v_0v_1\cdots v_{n-1}$ we write $|v|:=n$ for its length, and
$v_{[i,j)}:=v_i\cdots v_{j-1}$ for $0\le i\le j\le n$.  Juxtaposition
denotes concatenation, of two words or of a word followed by a sequence;
$v^k$ denotes the concatenation of $k$ copies of $v$, and
$v^\infty:=vvv\cdots$.
We say that a word $v$ \emph{occurs} in
$X$ if $v=x_{[r,s)}$ for some $x\in X$ and $0\le r\le s$, and let $\w_k(X)$
denote the set of words of length $k$ occurring in $X$.
We write
$|v|_1$ for the number of occurrences of the digit $1$ in $v$, and $[v]:=\{x\in\Sigma_m:x_{[0,|v|)}=v\}$ for the associated cylinder set,
the value of $m$ being clear from the context;
for a subshift $X\subseteq\Sigma_m$ we write $[v]_X:=[v]\cap X$.
For a Borel set
$C\subseteq X$ and a Borel probability measure $\mu$ on $X$ we
let $\ind_C$ denote the indicator function of $C$, and
call
$\ind_C-\mu(C)$ the \emph{centred indicator function}, the constant
being subtracted so that the integral with respect to $\mu$ vanishes.

For $g:X\to\R$ and $r\in\N_0$ we write $S_rg:=\sum_{h=0}^{r-1}g\circ\sigma^h$
for the Birkhoff sum of length $r$, with $S_0g:=0$.

For $t\in\R$, $\{t\}$ denotes the fractional part of $t$.  We identify the
circle $\R/\Z$ with $[0,1)$ by $t\mapsto\{t\}$, and write
$R_\alpha t:=t+\alpha$ for the rotation by $\alpha$.
We write $\mathrm{Leb}$ for Lebesgue measure on $\R/\Z$, identified
with $[0,1)$.
\end{notation}

The Wasserstein distance \eqref{eq:wass} metrises the weak$^*$ topology on
the set of Borel probability measures on $X$; see \cite{villani,villani2}.

\subsection{A criterion for robust maximization}

The particular form of Theorem \ref{thm:C} is motivated by the following
criterion, which together with its converse (Remark \ref{rem:converse})
asserts that an invariant measure persists as the unique maximizing measure
under all small Lipschitz perturbations if and only if the amount by which
its integral exceeds that of any other invariant measure is at least
proportional to the Wasserstein
distance between the two.

\begin{lem}\label{lem:stability}
Let $X$ be a subshift equipped with a metric $d_\theta$, let
$f\in\lip(X,d_\theta)$, let $\mu\in\m(X,\sigma)$, and let $\delta>0$ satisfy
\begin{equation}\label{eq:gap}
  \int f\,d\mu-\int f\,d\nu\ \ge\ \delta\,W_\theta(\nu,\mu)
  \qquad\text{for all }\nu\in\m(X,\sigma).
\end{equation}
Then $\Mmax(f+g)=\{\mu\}$ for all $g\in\lip(X,d_\theta)$ with
$|g|_{\lip(X)}<\delta$.
\end{lem}

\begin{proof}
Let $|g|_{\lip(X)}<\delta$ and let $\nu\in\m(X,\sigma)\setminus\{\mu\}$, so
that $W_\theta(\nu,\mu)>0$.  If $g$ is constant then
$\int g\,d\mu-\int g\,d\nu=0$, and otherwise applying \eqref{eq:wass} to
$g/|g|_{\lip(X)}$ gives
$\int g\,d\mu-\int g\,d\nu\ge-|g|_{\lip(X)}W_\theta(\nu,\mu)$; in either
case, by \eqref{eq:gap},
\[
  \int(f+g)\,d\mu-\int(f+g)\,d\nu
  \ \ge\ \bigl(\delta-|g|_{\lip(X)}\bigr)W_\theta(\nu,\mu)\ >\ 0 .
\]
So $\mu$ is the unique maximizing measure for $f+g$.
\end{proof}

Lemma \ref{lem:stability} admits the following converse.

\begin{rem}\label{rem:converse}
If $\mu\in\Mmax(f+g)$ for all $g\in\lip(X,d_\theta)$ with
$|g|_{\lip(X)}<\delta$, then \eqref{eq:gap} holds.  Indeed, suppose
\eqref{eq:gap} fails for some $\nu$ (necessarily with $\nu\ne\mu$ and hence
$W_\theta(\nu,\mu)>0$), and choose $\delta'\in[0,\delta)$ with
$\int f\,d\mu-\int f\,d\nu<\delta'W_\theta(\nu,\mu)$.  By \eqref{eq:wass}
there is a function $h$ with $|h|_{\lip(X)}\le1$ such that
$\int h\,d\nu-\int h\,d\mu$ is close enough to $W_\theta(\nu,\mu)$ that
$g:=\delta'h$ satisfies $\int(f+g)\,d\nu>\int(f+g)\,d\mu$; then
$|g|_{\lip(X)}<\delta$ and $\mu\notin\Mmax(f+g)$.  So the stability of
$\mu$ as a maximizing measure is equivalent to an inequality of the form
\eqref{eq:gap}.
\end{rem}

\subsection{Sturmian sequences and Sturmian beta-shifts}\label{subsec:sturmian}

We now collect basic properties of Sturmian sequences, and of the shifts
$X_\alpha$ and $Y_\alpha$.

Throughout the rest of the paper, unless stated otherwise, $\alpha\in(0,1)$ is irrational, $\bfa$ is
the maximal Sturmian sequence \eqref{eq:a} of rotation number $\alpha$,
$X_\alpha$ and $Y_\alpha$ are the associated Sturmian subshift and Sturmian
beta-shift, and $\Phi$ is the digit function \eqref{eq:functions}.

As $\Phi(x)=x_0$, summing \eqref{eq:a} gives
\begin{equation}\label{eq:prefixsum}
  S_j\Phi(\bfa)=\sum_{i<j}a_i=\lceil j\alpha\rceil\qquad(j\ge0).
\end{equation}
In the terminology of \cite[Ch.~2]{Lothaire}, $\bfa$ is the upper
mechanical word of slope $\alpha$ and intercept $0$, so that
$\bfa=1c_\alpha$, where $c_\alpha$ is the characteristic word of slope
$\alpha$.

Sturmian sequences arise as codings of orbits of the rotation $R_\alpha$, as
follows.

\begin{defn}\label{def:coding}
For $t\in\R$, let $\bfb(t)=\bfb_\alpha(t)\in\Sigma$ be given by
\begin{equation}\label{eq:coding}
  \bfb(t)_i:=\lfloor t+(i+1)\alpha\rfloor-\lfloor t+i\alpha\rfloor
  \qquad(i\ge0).
\end{equation}
The function $\bfb=\bfb_\alpha:\R\to\Sigma$ will be referred to as the \emph{Sturmian coding} of rotation number
$\alpha$.
\end{defn}

For a fixed $i\ge0$, the function $\bfb(\cdot)_i$ takes only the values $0$
and $1$, satisfies $\bfb(t+1)_i=\bfb(t)_i$, and is right continuous with
exactly two discontinuities in each interval of length $1$, at the points
congruent to $-i\alpha$ and $-(i+1)\alpha$ modulo $1$.  Hence the
coordinatewise one-sided limits
\[
  \bfb^{\pm}(t):=\lim_{s\to t^{\pm}}\bfb(s)
\]
exist in $\Sigma$ for all $t$, with $\bfb^{+}(t)=\bfb(t)$, and the
functions $\bfb^{\pm}$ are $1$-periodic.  More generally, for a function
$g$ on $\R$ having one-sided limits everywhere we define $g(t^{+})$ and
$g(t^{-})$ by $g(t^{\pm}):=\lim_{s\to t^{\pm}}g(s)$, so that in particular
$g(t^{+})=g(t)$ when $g$ is right continuous.  In any identity involving
several such expressions, together with $\bfb^{\pm}$, the same choice of
sign is made throughout.

The basic properties of the Sturmian coding are as follows.

\begin{lem}\label{lem:coding}
Let $\alpha\in(0,1)$ be irrational, let $\bfa$ be the maximal Sturmian
sequence of rotation number $\alpha$, let $\bfb$ be the Sturmian coding, and let $I_0:=[0,1-\alpha)$ and
$I_1:=[1-\alpha,1)$, regarded as subsets of $\R/\Z$.  Then for all
$t\in\R$ and $i\ge0$, the digit $\bfb(t)_i$ equals $1$ if
$t+i\alpha\in I_1$ modulo $1$, and equals $0$ otherwise.  Moreover $\bfa=\bfb^{-}(0)$, and $a_i=\bfb(0)_i$
for all $i\ge1$, whereas $a_0=1$ and $\bfb(0)_0=0$.  Finally,
\begin{equation}\label{eq:shiftcoding}
  \sigma\bfb(t)=\bfb(R_\alpha t),\qquad
  \sigma\bfb^{\pm}(t)=\bfb^{\pm}(R_\alpha t).
\end{equation}
\end{lem}
\begin{proof}
Writing $s:=\{t+i\alpha\}$, \eqref{eq:coding} gives
$\bfb(t)_i=\lfloor s+\alpha\rfloor$, which equals $1$ if $s+\alpha\ge1$ and
$0$ otherwise, since $0\le s+\alpha<2$.  As $s+\alpha\ge1$ if and only if
$s\ge1-\alpha$, i.e.~if and only if $t+i\alpha\in I_1$ modulo $1$, the
first statement follows; and \eqref{eq:shiftcoding} is immediate from
\eqref{eq:coding}.  For $i\ge1$
neither $i\alpha$ nor $(i+1)\alpha$ is an integer, so
$\lceil j\alpha\rceil=\lfloor j\alpha\rfloor+1$ for $j\in\{i,i+1\}$ and
hence $a_i=\bfb(0)_i$, whereas $a_0=\lceil\alpha\rceil=1$ and
$\bfb(0)_0=\lfloor\alpha\rfloor=0$.  Finally, given $n\in\N$, choose $\eta>0$ smaller
than the distance from each of $\alpha,2\alpha,\ldots,n\alpha$ to the
nearest integer; then $\lfloor-\eta+i\alpha\rfloor=\lceil
i\alpha\rceil-1$ for $i\le n$, so that
$\bfb(-\eta)_i=a_i$ for $i<n$; letting $\eta\downarrow0$ and $n\to\infty$
gives $\bfb^{-}(0)=\bfa$.
\end{proof}

The words of length $k$ occurring in $X_\alpha$ are indexed by the
intervals into which the $k+1$ points $\{-j\alpha\}$, $0\le j\le k$, cut
the circle:

\begin{notation}\label{not:cuts}
For $k\in\N$ and $v\in\{0,1\}^k$ set
\begin{equation}\label{eq:Iv}
  I_v:=\{t\in[0,1):\bfb(t)_{[0,k)}=v\},
\end{equation}
and for
$j\ge0$ set
\[
  c_j:=\{-j\alpha\} .
\]
We call $c_0,\ldots,c_k$ the \emph{cuts of level $k$}, i.e.~the
first $k+1$ terms of the single sequence $(c_j)_{j\ge0}$.
When $I_v$ is the argument of an
indicator function on $\R$, it is identified with the $1$-periodic set
$I_v+\Z$.
\end{notation}

The cuts of level $k$ determine the intervals $I_v$, as follows.

\begin{lem}\label{lem:cuts}
Let $\alpha\in(0,1)$ be irrational, let $k\in\N$, and let the intervals
$I_v$, $v\in\{0,1\}^k$, and the cuts $c_j$ be as in Notation
\ref{not:cuts}.  Then:
\begin{enumerate}
\item[(i)] the cuts of level $k$ are $k+1$ distinct points of $[0,1)$, one of
which is $c_0=0$,
and they divide $[0,1)$ into $k+1$ intervals of positive length, each of the
form $[c,c')$ with $c$ a cut of level $k$ and $c'$ either a cut of level $k$
or $1$;
\item[(ii)] each of these $k+1$ intervals equals $I_v$ for exactly one
$v\in\{0,1\}^k$, distinct intervals corresponding to distinct words, and
$I_v=\emptyset$ for all other $v$.
\end{enumerate}
\end{lem}

\begin{proof}
Summing \eqref{eq:coding} over $i<j$ gives
$\sum_{i<j}\bfb(t)_i=\lfloor t+j\alpha\rfloor$ for $t\in[0,1)$, so the word
$\bfb(t)_{[0,k)}$ is determined by the $k$ integers
$\lfloor t+j\alpha\rfloor$, $1\le j\le k$.  As $t$ increases through
$[0,1)$, the $j$th of these is non-decreasing, and increases by one when
$t$ passes $c_j$.  Since $\alpha$ is irrational the points
$c_0,\ldots,c_k$ are distinct, which gives (i).
Each of these $k$ integers is a right continuous step function of $t$, so all
of them are constant on each of the intervals $[c,c')$ in (i), and hence so is
$\bfb(t)_{[0,k)}$;
and if $t<t'$ lie in different intervals then some $c_j$ with $1\le j\le k$
separates them, so that $\lfloor t+j\alpha\rfloor<\lfloor t'+j\alpha\rfloor$
and the two words differ,
thus giving (ii).
\end{proof}

In terms of the Sturmian coding, the standard properties of $X_\alpha$,
$Y_\alpha$ and $\mu_\alpha$ are as follows.

\begin{lem}\label{lem:sturmian}
Let $\alpha\in(0,1)$ be irrational, with $\bfa$ the maximal Sturmian
sequence of rotation number $\alpha$, and let $X_\alpha$ and $Y_\alpha$ be the
Sturmian subshift, and Sturmian beta-shift, of rotation number $\alpha$.
Let $\bfb$ be the Sturmian coding, let the intervals $I_v$ be as in Notation
\ref{not:cuts}, and let $\Phi$ be the digit function.  Then:
\begin{enumerate}
\item[(i)] $X_\alpha=\{\bfb^{\pm}(t):t\in\R\}$, and $X_\alpha$ is minimal,
non-periodic and uniquely ergodic, with invariant measure
$\mu_\alpha=\bfb_*\mathrm{Leb}$;
\item[(ii)] $\w_k(X_\alpha)=\{v\in\{0,1\}^k:I_v\ne\emptyset\}$ has exactly
$k+1$ elements, and $\mu_\alpha([v]_{X_\alpha})=|I_v|$ for all
$v\in\{0,1\}^k$, with the convention $|I_v|=0$ when $I_v=\emptyset$;
\item[(iii)] $\int\Phi\,d\mu_\alpha=|I_1|=\alpha$;
\item[(iv)] all $v\in\w_k(X_\alpha)$ satisfy
$\lfloor k\alpha\rfloor\le|v|_1\le\lceil k\alpha\rceil$, and consequently
$|S_k\Phi(x)-k\alpha|<1$ for all $x\in X_\alpha$ and $k\in\N_0$;
\item[(v)] $\bfa=\max_\lex X_\alpha$ and $\sigma^j\bfa<_\lex\bfa$ for all
$j\ge1$; consequently there is a unique $\beta\in(1,2)$ with
$\bfa=d_\beta(1)$, the set $Y_\alpha$ of \eqref{eq:Yalpha} is the beta-shift
of that $\beta$, and $X_\alpha\subseteq Y_\alpha$.
\end{enumerate}
\end{lem}

\begin{proof}
These are standard; see \cite{MH40} and \cite[Ch.~2]{Lothaire}, or
\cite[Ch.~6]{Fogg}.  For the first equality in (i), note that
\eqref{eq:coding} exhibits $\bfb(t)$ as the lower mechanical word of slope
$\alpha$ and intercept $t$, while
$\bfb^{-}(t)_i=\lceil t+(i+1)\alpha\rceil-\lceil t+i\alpha\rceil$ exhibits
$\bfb^{-}(t)$ as the upper mechanical word of the same slope and intercept;
so $\{\bfb^{\pm}(t):t\in\R\}$ is the set of all mechanical words of slope
$\alpha$, which by the theorem of Morse \& Hedlund
\cite[Thm.~2.1.13]{Lothaire} is the set of all Sturmian words of slope
$\alpha$.  The factors of such a word depend only on its slope
\cite[Prop.~2.1.18]{Lothaire}, so that set is the subshift generated by any
one of its elements, which is $X_\alpha$.  Part (ii) is the count of
the $k+1$ intervals in Lemma \ref{lem:cuts} together with unique ergodicity
of $R_\alpha$ applied to the intervals $I_v$.  Part (iv) follows from
$|v|_1=\lfloor t+k\alpha\rfloor-\lfloor t\rfloor$ for $t\in I_v$.  For (v),
the identity $\sum_{i<j}\bfb(t)_i=\lfloor t+j\alpha\rfloor$ shows that for
$0\le t<s<1$, at the first index where the partial sums of $\bfb(t)$ and
$\bfb(s)$ differ the digit of $\bfb(t)$ is the smaller; hence
$\bfb(t)\le_\lex\bfb(s)$, and letting $s\uparrow1$ and using (i) gives
$\bfa=\bfb^{-}(0)=\max_\lex X_\alpha$.  As $X_\alpha$ is non-periodic the
inequality $\sigma^j\bfa\le_\lex\bfa$ is strict for $j\ge1$, so Parry's
characterisation \cite{Par60} of the sequences arising as $d_\beta(1)$
gives a unique $\beta>1$ with $\bfa=d_\beta(1)$.  The first term of
$d_\beta(1)$ is $\lfloor\beta\rfloor$, so $\lfloor\beta\rfloor=a_0=1$, and
hence $\beta\in(1,2)$.  As $\bfa$ has infinitely many non-zero terms,
$d^*_\beta(1)=\bfa$, so $Y_\alpha$ is the beta-shift of $\beta$; and
$X_\alpha\subseteq Y_\alpha$ because $\sigma^jx\le_\lex\bfa$ for all
$x\in X_\alpha$ and $j\ge0$.
\end{proof}

\section{The maximizing property of the digit function}\label{sec:maximizing}

In this section we describe a graph presentation of $Y_\alpha$, identify the
Markov boundary of a non-sofic beta-shift, deduce that $Y_\alpha$ has the
specification property, and prove that $\mu_\alpha$ is a maximizing measure
for the digit function $\Phi$ (Proposition \ref{prop:Phimax}).

The following graph presentation for beta-shifts does not need any Sturmian
hypothesis:

\begin{defn}\label{def:graph}
Let $\bfa=(a_i)_{i\ge0}$ be a sequence of non-negative integers such that
$a_0\ge1$ and
$\sigma^j\bfa\le_\lex\bfa$ for all $j\ge1$, so that $\bfa\in\Sigma_{a_0}$.  We define the \emph{graph of
$\bfa$} to mean the directed labelled graph $G(\bfa)$ with vertex set
$\N_0$ whose edges are, for each $r\in\N_0$, an edge from $r$ to $r+1$
labelled $a_r$, together with an edge from $r$ to $0$ labelled $c$ for
each integer $c$ with $0\le c<a_r$.  A \emph{first return loop} of $G(\bfa)$ is a path from the
vertex $0$ to itself, of length at least $1$, visiting the vertex $0$ only
at its two endpoints; we identify such a path with the word formed by its
labels.  Finally, set
\begin{equation*}
  Y(\bfa):=\{x\in\Sigma_{a_0}:\sigma^jx\le_\lex\bfa\ \text{ for all }j\ge0\} .
\end{equation*}
\end{defn}

The edges leaving a given vertex carry distinct labels, so for all
$x\in\Sigma_{a_0}$ there is at most one path in $G(\bfa)$ starting at the vertex
$0$ whose sequence of labels is $x$; that is, $G(\bfa)$ is
right-resolving.  When such a path exists we write
$r_0(x)=0,r_1(x),r_2(x),\ldots$ for its sequence of vertices, defined for
as long as the path can be continued.

Definition \ref{def:graph} applies to all beta-shifts, as follows.

\begin{rem}\label{rem:general}
For all $\beta>1$, the hypotheses of Definition \ref{def:graph} are
satisfied by the quasi-greedy expansion $\bfa=d^*_\beta(1)$, by Parry's
criterion \cite{Par60}, and $Y(\bfa)$ is then the beta-shift $S_\beta$, as
the first term of $d^*_\beta(1)$ is $\lceil\beta\rceil-1$; here
$d^*_\beta(1)=d_\beta(1)$ unless $\beta$ is a simple beta-number, in which
case $d^*_\beta(1)$ is periodic.
So Lemma
\ref{lem:presentation} below is a statement about all beta-shifts.  If
$\bfa$ is not eventually periodic then $G(\bfa)$ is the Fischer automaton of
$Y(\bfa)$ in the sense of Blanchard \& Hansel \cite{BH86}, and in general the
presentation is standard; see
\cite[\S 5]{PS05}, \cite[\S 5.1]{Tho12} and \cite[\S 3.2]{GK}.  The short
proof of Lemma \ref{lem:presentation} is included because the precise form
of $G(\bfa)$ is used at various points below.
\end{rem}

The following lemma describes $Y(\bfa)$ in terms of paths in $G(\bfa)$,
and identifies the first return loops of $G(\bfa)$.

\begin{lem}\label{lem:presentation}
Let $\bfa=(a_i)_{i\ge0}$ be a sequence of non-negative integers such that
$a_0\ge1$ and
$\sigma^j\bfa\le_\lex\bfa$ for all $j\ge1$, and let $G(\bfa)$ and $Y(\bfa)$
be as in Definition \ref{def:graph}.  Then $Y(\bfa)$ is the set of
$x\in\Sigma_{a_0}$ labelling an infinite path in $G(\bfa)$ that starts at the
vertex $0$.  Moreover, the first return loops of $G(\bfa)$ are precisely the
words
\begin{equation}\label{eq:Rn}
  R_{n,c}:=a_0a_1\cdots a_{n-2}\,c\qquad
  (n\ge1,\ 0\le c<a_{n-1}),
\end{equation}
the word $R_{n,c}$ having length $n$; in particular $R_{1,0}=0$ is a first
return loop.
\end{lem}

\begin{proof}
For $r\in\N_0$ let $\cQ_r$ denote the set of $y\in\Sigma_{a_0}$ labelling an
infinite path in $G(\bfa)$ starting at the vertex $r$, and let
$(s_j)_{j\ge0}$, with $s_0=r$, be the vertex sequence of that path.  We
claim that $y\in\cQ_r$ implies $y\le_\lex\sigma^r\bfa$.  If $s_j\ne0$ for
all $j\ge1$ then inductively $s_j=r+j$ and $y_j=a_{r+j}$, so
$y=\sigma^r\bfa$.  Otherwise let $h\ge1$ be the smallest integer such that $s_h=0$; then
$y_j=a_{r+j}$ for $j\le h-2$, while the $h$th edge runs from $s_{h-1}=r+h-1$
to $0$ and is therefore labelled by some $c<a_{r+h-1}$, so
$y_{h-1}<a_{r+h-1}$ and
$y<_\lex\sigma^r\bfa$.

If $x$ labels an infinite path starting at $0$ then
$\sigma^jx\in\cQ_{r_j(x)}$, so $\sigma^jx\le_\lex\sigma^{r_j(x)}\bfa
\le_\lex\bfa$ by hypothesis, and $x\in Y(\bfa)$.  Conversely, let
$x\in Y(\bfa)$ and suppose $r_j(x)$ were undefined for some smallest $j$,
necessarily $j\ge1$.  Writing $\rho:=r_{j-1}(x)$, no edge labelled
$x_{j-1}$ leaves the vertex $\rho$.  The labels of the edges leaving $\rho$
are $0,1,\ldots,a_\rho$, so $x_{j-1}>a_\rho$.  As
$x_{j-1}\le a_0$ we have $\rho\ge1$, so $\rho=r_{j-2}(x)+1$ and
$x_{j-2}=a_{r_{j-2}(x)}$; iterating downwards gives
$r_{j-1-\ell}(x)=\rho-\ell$ and $x_{j-1-\ell}=a_{\rho-\ell}$ for
$1\le\ell\le\rho$.  Setting $p:=j-1-\rho$ this gives $x_{p+i}=a_i$ for
$0\le i<\rho$, while $x_{p+\rho}=x_{j-1}>a_\rho$, so that
$\sigma^p x>_\lex\bfa$, contradicting that $x\in Y(\bfa)$.

For the last statement,
following the edges labelled
$a_0,\ldots,a_{n-2}$ from the vertex $0$
leads to the vertex $n-1$ without returning to $0$, and a final symbol $c$
then returns to $0$
if and only if $c<a_{n-1}$; all first return loops
arise in this way.  Since $a_0\ge1$, the word $R_{1,0}=0$ is a first return
loop.
\end{proof}

The following lemma identifies the Markov boundary (see Section \ref{subsec:transformations}) of a beta-shift which is not sofic.  If
$\mathbf e$ satisfies the hypotheses of Definition \ref{def:graph}
and is not eventually periodic, then $\sigma^j\mathbf e<_\lex\mathbf e$ for
all $j\ge1$, so $Y(\mathbf e)$ is a beta-shift by Parry's criterion
\cite{Par60}, which is not sofic \cite{Bertrand86,blanchard}; all non-sofic
beta-shifts arise in this way.

\begin{lem}\label{lem:markov-boundary}
Let $\mathbf e=(e_i)_{i\ge0}$ be a sequence of non-negative integers such
that $e_0\ge1$ and
$\sigma^j\mathbf e\le_\lex\mathbf e$ for all $j\ge1$, and suppose that
$\mathbf e$ is not eventually periodic.  Then the Markov boundary
$\partial_MY(\mathbf e)$ of the beta-shift $Y(\mathbf e)$ is the
$\omega$-limit set of $\mathbf e$.
\end{lem}

\begin{proof}
For a word $v$ occurring in $Y(\mathbf e)$, let $\ell(v)$ be the length of
the longest suffix of $v$ which is a prefix of $\mathbf e$, and let $x\in
Y(\mathbf e)$.  As $v$ occurs in $Y(\mathbf e)$, each suffix of $v$ which is
not a prefix of $\mathbf e$ is lexicographically smaller than the prefix of
$\mathbf e$ of the same length, so $vx\in Y(\mathbf e)$ if and only if
$x\le_\lex\sigma^k\mathbf e$ for all $k\ge1$ such that the suffix of $v$ of
length $k$ is a prefix of $\mathbf e$.  For such $k$ we have $k\le\ell(v)$,
and $\sigma^{\ell(v)}\mathbf e\le_\lex\sigma^k\mathbf e$ because
$\sigma^{\ell(v)-k}\mathbf e\le_\lex\mathbf e$ and these two sequences have
the same first $k$ symbols.  Thus $vx\in Y(\mathbf e)$ if and only if
$x\le_\lex\sigma^{\ell(v)}\mathbf e$, so the follower set of $v$, i.e.~the set of words $u$ such that $vu$ occurs in $Y(\mathbf e)$, consists of the
prefixes of those $x\in Y(\mathbf e)$ with $x\le_\lex\sigma^{\ell(v)}\mathbf
e$, and depends only on $\ell(v)$.  The follower sets of the prefixes of
$\mathbf e$ are pairwise distinct: the sequences $\sigma^n\mathbf e$,
$n\ge0$, are pairwise distinct as $\mathbf e$ is not eventually periodic,
and if $\sigma^n\mathbf e<_\lex\sigma^k\mathbf e$ then the shortest prefix of
$\sigma^k\mathbf e$ which is not a prefix of $\sigma^n\mathbf e$ lies in the
follower set of $\mathbf e_{[0,k)}$ but not in that of $\mathbf e_{[0,n)}$.

Now let $w$ be a word.  If $w$ occurs in $\mathbf e$ at arbitrarily large
positions $p$, then the words $\mathbf e_{[0,p)}w=\mathbf e_{[0,p+|w|)}$
have infinitely many distinct follower sets.  If instead all occurrences of
$w$ in $\mathbf e$ start before some position $B$, then $\ell(uw)<B+|w|$ for
all words $uw$ occurring in $Y(\mathbf e)$, so these words have only finitely
many distinct follower sets.  The words occurring in $\mathbf e$ at
arbitrarily large positions are precisely the words occurring in the
$\omega$-limit set of $\mathbf e$, so by \cite[Def.~8.7]{HJXZ} this
$\omega$-limit set is $\partial_MY(\mathbf e)$.
\end{proof}

We now compare Theorem \ref{thm:A} with the structural theory of
\cite{HJXZ} recalled in Section \ref{subsec:transformations}.

\begin{rem}\label{rem:markov-boundary}
Let $\alpha\in(0,1)$ be irrational.  As $X_\alpha$ is minimal and contains
$\bfa$, the $\omega$-limit set of $\bfa$ is $X_\alpha$, so
$\partial_MY_\alpha=X_\alpha$ by Lemma \ref{lem:markov-boundary}.  As
$X_\alpha$ contains no periodic orbit, it does not have TPO, so
\cite[Thm.~1.3]{HJXZ}, according to which a shift space has TPO whenever its
Markov boundary has TPO, does not apply to $Y_\alpha$; and Theorem
\ref{thm:A} shows that the set of Lipschitz functions having a maximizing
measure supported by $\partial_MY_\alpha$ has non-empty interior, that is,
$Y_\alpha$ is not \emph{fragile} in the sense of \cite[Def.~10.2]{HJXZ}.
The failure of TPO for $Y_\alpha$ is nevertheless not determined by its
Markov boundary alone: all beta-shifts $S\ne Y_\alpha$ with
$\partial_MS=X_\alpha$ have TPO with respect to $\lip(S,d_\theta)$, for all
$\theta\in(0,1)$.  To see this, let $\beta>1$ be such that $S=S_\beta$.  As
$\partial_MS$ is non-empty, $S$ is not sofic, so $\beta$ is not a
beta-number, and the $\omega$-limit set of $d^*_\beta(1)$ is
$\partial_MS=X_\alpha$, by Lemma \ref{lem:markov-boundary}.  In particular
$\bfa\in S$, so $\bfa\le_\lex d^*_\beta(1)$, and the inequality is strict as
$S\ne Y_\alpha$, so $d^*_\beta(1)\notin X_\alpha$.  Hence the orbit closure
of $d^*_\beta(1)$ properly contains the closed invariant set $X_\alpha$, so
is not minimal, and $\beta$ is non-emergent \cite[Prop.~5.3]{HHJL}.  So $S$
has TPO by Theorem \ref{thm:nonemergent}, and the beta-transformations
$T_\beta$ and $U_\beta$ also have TPO \cite[Thm.~J]{HHJL}; by contrast, if
$\beta_\alpha$ denotes the parameter with $S_{\beta_\alpha}=Y_\alpha$, then
neither $T_{\beta_\alpha}$ nor $U_{\beta_\alpha}$ has TPO (Theorem
\ref{thm:E}).  For example, $Y(1\bfa)$ (see Definition \ref{def:graph}) is
such a beta-shift: as $\bfa$ begins with $1^k0$ for some $k\in\N$, we have
$\bfa<_\lex1\bfa$, so $\sigma^j(1\bfa)<_\lex1\bfa$ for all $j\ge1$, and the
$\omega$-limit set of $1\bfa$ is $X_\alpha$.
\end{rem}

For the maximal Sturmian sequence $\bfa=\bfa(\alpha)$ of \eqref{eq:a}, with
$\alpha\in(0,1)$ irrational, the hypotheses of Definition \ref{def:graph}
hold, since $a_0=\lceil\alpha\rceil=1$ and $\sigma^j\bfa<_\lex\bfa$ for all
$j\ge1$ by Lemma \ref{lem:sturmian}(v), and $Y(\bfa)=Y_\alpha$ by
\eqref{eq:Yalpha}; in this case we shall write $G_\alpha$ instead of
$G(\bfa)$.

We now record the specification property of Sturmian beta-shifts,  which gives Corollary
\ref{cor:B} and is used in the proof of Theorem \ref{thm:E}.

\begin{lem}\label{lem:spec}
Let $\alpha\in(0,1)$ be irrational, let $\bfa$ be the maximal Sturmian
sequence of rotation number $\alpha$, and let $Y_\alpha$ be the
Sturmian beta-shift of rotation number $\alpha$.  Then there is $N\in\N$ such
that the word $0^N$ does not occur in $\bfa$.  For all such $N$, and all words
$u,v$ occurring in $Y_\alpha$, the word $u0^Nv$ occurs in $Y_\alpha$, and
$(u0^N)^\infty\in Y_\alpha$.  In particular $Y_\alpha$ has the specification
property.
\end{lem}

\begin{proof}
If $0^N$ occurred in $\bfa$ for all $N\in\N$ then $0^\infty$ would lie in
$X_\alpha$, so $X_\alpha=\{0^\infty\}$ by minimality, a contradiction; so the first statement is proved.
Now fix 
$N\in\N$ such
that $0^N$ does not occur in $\bfa$.
All words occurring in
$Y_\alpha$ are prefixes of points of $Y_\alpha$, since
$\sigma(Y_\alpha)\subseteq Y_\alpha$, so by Lemma \ref{lem:presentation}
each of them labels a path in $G_\alpha$ starting at the vertex $0$.  An edge labelled $0$
leaves each vertex $r$: it runs to $0$ if $a_r\ge1$, and to $r+1$ if $a_r=0$;
in particular the edge labelled $0$ leaving the vertex $0$ is a loop.  A path
labelled $0^N$ which never visits the vertex $0$ uses only edges of the form
$r\to r+1$, so its label is a word occurring in $\bfa$, which is impossible;
hence all paths labelled $0^N$ end at the vertex $0$.  So if $u$ labels a path
starting at $0$ then $u0^N$ labels a path from $0$ to $0$, and the sequences
$u0^Nv0^\infty$ and $(u0^N)^\infty$ label infinite paths starting at the
vertex $0$; by Lemma \ref{lem:presentation} they lie in $Y_\alpha$.  (The
specification property of $Y_\alpha$ can also be deduced from the theorem of
Bertrand-Mathis \cite{Bertrand86,Bertrand88} that a beta-shift has the
specification property if and only if the strings of zeros in $d_\beta(1)$
have bounded length; see also \cite[Prop.~3.5]{Sch97} and
\cite{blanchard,PS05}.)
\end{proof}

The numbers of digits $1$ in the first return loops of $G_\alpha$ are as
follows.

\begin{lem}\label{lem:returnwords}
Let $\alpha\in(0,1)$ be irrational, and let $\bfa$ be the maximal Sturmian
sequence of rotation number $\alpha$.  Then $a_{n-1}=1$ for infinitely many
$n\in\N$, the first return loops of the graph $G_\alpha=G(\bfa)$ of
Definition \ref{def:graph} are the words
$R_{n,0}$ of \eqref{eq:Rn} with $n\in\N$ and $a_{n-1}=1$, and
\begin{equation}\label{eq:defect}
  |R_{n,0}|_1=\lfloor n\alpha\rfloor\qquad(n\in\N,\ a_{n-1}=1) .
\end{equation}
\end{lem}

\begin{proof}
The digit $1$ occurs infinitely often in $\bfa$, by \eqref{eq:prefixsum}
and $\alpha>0$.  As $a_{n-1}\in\{0,1\}$ for all $n\in\N$, the first return
loops $R_{n,c}$, $0\le c<a_{n-1}$, of Lemma \ref{lem:presentation} are the
words $R_{n,0}$ with $a_{n-1}=1$.  For such $n$, \eqref{eq:prefixsum} gives
$|R_{n,0}|_1=\sum_{i\le n-2}a_i=\lceil(n-1)\alpha\rceil$, which equals
$\lfloor n\alpha\rfloor$ when $n=1$, both being $0$; and for $n\ge2$ it
equals $\lceil n\alpha\rceil-1=\lfloor n\alpha\rfloor$, by \eqref{eq:a} and
since $n\alpha\notin\Z$.
\end{proof}

By \eqref{eq:defect}, the digit $1$ has frequency $\lfloor n\alpha\rfloor/n<\alpha$
in each first return loop $R_{n,0}$ of $G_\alpha$.  As every point of
$Y_\alpha$ labels a path in $G_\alpha$ starting at the vertex $0$, this
leads to the following result, which shows that no invariant measure on
$Y_\alpha$ gives the digit $1$ a frequency larger than $\alpha$.

\begin{prop}\label{prop:Phimax}
Let $\alpha\in(0,1)$ be irrational, let $Y_\alpha$ be the Sturmian
beta-shift of rotation number $\alpha$, with Sturmian measure $\mu_\alpha$,
and let $\Phi$ be the digit function.  Then
\begin{equation}\label{eq:uniform}
  S_n\Phi(x)=\sum_{i<n}x_i\ \le\ n\alpha+1
  \qquad\text{for all }x\in Y_\alpha\text{ and all }n\in\N .
\end{equation}
Consequently $\int\Phi\,d\nu\le\alpha$ for all $\nu\in\m(Y_\alpha,\sigma)$,
with equality for $\nu=\mu_\alpha$, so that $\mu_\alpha\in\Mmax(\Phi)$.
\end{prop}

\begin{proof}
Let $x\in Y_\alpha$.
By Lemma \ref{lem:presentation}, whose hypotheses hold for $\bfa$ as noted above, $x$ labels
an infinite path in $G_\alpha$ from the
vertex $0$, with vertex sequence $(r_j(x))_{j\ge0}$.  Fix $n\in\N$ and let
$0=q_0<q_1<\cdots<q_s\le n$ be an enumeration of those $j\le n$ with $r_j(x)=0$, a
finite and non-empty list.  For $1\le i\le s$ the segment
$x_{[q_{i-1},q_i)}$ is a first return loop, so
$x_{[q_{i-1},q_i)}=R_{n_i,0}$ with $n_i:=q_i-q_{i-1}$ and
$a_{n_i-1}=1$, by Lemma \ref{lem:returnwords}.

For the final segment, set $k:=n-q_s\ge0$.  For $q_s<j\le n$ we have
$r_j(x)\ne0$, and the only edge from a vertex $h$ which does not lead to
$0$ is the one labelled $a_h$ leading to $h+1$; so inductively
$r_{q_s+h}(x)=h$ and $x_{q_s+h}=a_h$ for $0\le h\le k-1$, that is
$x_{[q_s,n)}=a_{[0,k)}$, whose number of digits $1$ is
$\lceil k\alpha\rceil$ by \eqref{eq:prefixsum}.

Combining these with \eqref{eq:defect},
\[
  \sum_{i<n}x_i=\sum_{i=1}^{s}\lfloor n_i\alpha\rfloor+\lceil k\alpha\rceil
  \le\sum_{i=1}^{s}n_i\alpha+k\alpha+1=n\alpha+1 ,
\]
which is \eqref{eq:uniform}.

Now let $\nu\in\m(Y_\alpha,\sigma)$.  Integrating \eqref{eq:uniform} and
using invariance of $\nu$ gives
$n\int\Phi\,d\nu=\int S_n\Phi\,d\nu\le n\alpha+1$, so
$\int\Phi\,d\nu\le\alpha+1/n$ for all $n\in\N$, and therefore
$\int\Phi\,d\nu\le\alpha$.  Finally $\mu_\alpha\in\m(Y_\alpha,\sigma)$ by
Lemma \ref{lem:sturmian}(v), and $\int\Phi\,d\mu_\alpha=\alpha$ by Lemma
\ref{lem:sturmian}(iii).
\end{proof}

\section{Decomposition of centred indicator functions}\label{sec:cob}

In this section we prove Proposition \ref{prop:centred-cylinder},
decomposing the centred indicator function of a cylinder set of $X_\alpha$
as the sum of an integer multiple of $\Phi-\alpha$, and a coboundary.

We begin with some notation.

\begin{notation}\label{not:cob}
Let $\alpha\in(0,1)$ be irrational, let $X_\alpha$ be the Sturmian subshift
of rotation number $\alpha$, 
let $\Phi$ be the digit function, let the
cuts $c_j$ and the intervals $I_v$ be as in Notation \ref{not:cuts}, and
let $k\in\N$.
By Lemmas \ref{lem:cuts}
and \ref{lem:sturmian}(ii),
each $v\in\w_k(X_\alpha)$ determines an interval
$I_v$ whose endpoints are cuts of level $k$, and we write
\begin{equation}\label{eq:endpoints}
  I_v=[c_i,c_j),
\end{equation}
where $i=i(v)$ and $j=j(v)$ satisfy $0\le i,j\le k$ and $i\ne j$, and where
the right endpoint is to be taken as $1$ in the case $j=0$.
Set
\[
  \gamma_{v,r}:=
  \begin{cases}
    \phantom{-}1 & \text{if } i\le r<j,\\
    -1 & \text{if } j\le r<i,\\
    \phantom{-}0 & \text{otherwise,}
  \end{cases}
  \qquad(0\le r\le k-1),
\]
\begin{equation}\label{eq:Bv}
  B_v:=\sum_{r=0}^{k-1}\gamma_{v,r} ,
\end{equation}
and
$$\phi_v:=\ind_{I_v}-|I_v|,$$
a $1$-periodic function on $\R$.

If
$v\in\{0,1\}^k\setminus\w_k(X_\alpha)$ then we set $B_v:=0$, $|I_v|:=0$,
$\gamma_{v,r}:=0$, and $\phi_v:=0$.

Finally, for $r\in\N_0$ define $H_r:\R\to[0,1)$ by
$H_r(t):=\{t+r\alpha\}$, define functions on $\Sigma$ by
\begin{equation}\label{eq:Kr}
  K_0:=0,\qquad K_r:=S_r(\alpha-\Phi)\quad(r\ge1),
\end{equation}
and set
\[
  G_v:=\sum_{r=0}^{k-1}\gamma_{v,r}H_r,\qquad
  V_v:=\sum_{r=0}^{k-1}\gamma_{v,r}K_r .
\]
\end{notation}

The quantities in Notation \ref{not:cob} have the following elementary
properties.

\begin{lem}\label{lem:Bvbounds}
Let $\alpha$ and $k$ be as in Notation \ref{not:cob}, let $v\in\{0,1\}^k$,
and let $\gamma_{v,r}$, $B_v$, $K_r$ and $V_v$ be as defined there.  If
$v\in\w_k(X_\alpha)$, with $i$ and $j$ as in \eqref{eq:endpoints}, then the
non-zero $\gamma_{v,r}$ all have the same sign, and
\begin{equation}\label{eq:Bvalue}
  B_v=j-i,\qquad |B_v|\le k .
\end{equation}
For all $r\in\N_0$, the function $K_r$ depends on the
coordinates $x_0,\ldots,x_{r-1}$ only, the function $V_v$ depends on the
coordinates $x_0,\ldots,x_{k-2}$ only, and
\begin{equation}\label{eq:Krbound}
  \|K_r\|_{C(\Sigma)}\le r\max(\alpha,1-\alpha)\le r ,
\end{equation}
\begin{equation}\label{eq:Vvbound}
  \|V_v\|_{C(\Sigma)}\le\sum_{r=0}^{k-1}r=\tfrac12k(k-1) .
\end{equation}
\end{lem}

\begin{proof}
Let $v\in\w_k(X_\alpha)$.  Since $i\ne j$, at most one of the first two
cases in the definition of $\gamma_{v,r}$ arises for any given $r$, and
only one of the two arises at all, so the non-zero $\gamma_{v,r}$ all have
the same sign.  That case contributes $+1$ for each $r$ with $i\le r<j$, or
$-1$ for each $r$ with $j\le r<i$, so $B_v=j-i$ in either case, and
$|B_v|\le k$ because $0\le i,j\le k$.

For the remaining statements, $K_r$ is by \eqref{eq:Kr} a sum of $r$ terms
each taking the two values $\alpha$ and $\alpha-1$, and the $h$th of them
depends on the coordinate $x_h$ only, which gives \eqref{eq:Krbound}; the
statements about $V_v$ follow, since $V_v$ is a linear combination of
$K_0,\ldots,K_{k-1}$ with coefficients in $\{-1,0,1\}$.
\end{proof}

\subsection{The identity on the circle}

Each of the functions $\phi_v$ is a coboundary for the rotation $R_\alpha$, as
described by the following
lemma (see also Remark~\ref{rem:brs}).

\begin{lem}\label{lem:coboundary}
Let $\alpha\in(0,1)$ be irrational, let $X_\alpha$ be the Sturmian subshift
of rotation number $\alpha$, let $k\in\N$, let $v\in\{0,1\}^k$, and let
$\phi_v$, $H_r$ and $G_v$ be as in Notation \ref{not:cob}.  If
$v\in\w_k(X_\alpha)$, with $i$ and $j$ as in \eqref{eq:endpoints}, then
\[
  \phi_v=H_j-H_i=G_v\circ R_\alpha-G_v .
\]
If $v\notin\w_k(X_\alpha)$ then $\phi_v=G_v=0$.
\end{lem}

\begin{proof}
The second statement is immediate from Notation \ref{not:cob}, so let
$v\in\w_k(X_\alpha)$.  For $0\le p<q\le1$ and $t\in[0,1)$,
\begin{equation}\label{eq:indicator}
  \ind_{[p,q)}(t)-(q-p)=\{t-q\}-\{t-p\} ,
\end{equation}
as is checked separately on the three intervals $[0,p)$, $[p,q)$ and
$[q,1)$, on which the right-hand side takes the values $p-q$, $1-q+p$ and
$p-q$ respectively.  Note that $c_i<c_j$ whenever $j\ne0$, since $I_v$ is a
non-empty interval, and that $c_j$ is taken as $1$ when $j=0$; so we may
take $p=c_i$ and $q=c_j$ in \eqref{eq:indicator}.  Using
$c_i\equiv-i\alpha$ and $c_j\equiv-j\alpha$ modulo $1$ gives
$\phi_v=H_j-H_i$; here the case $j=0$ is covered because
$\{t-1\}=\{t\}=H_0(t)$.

For the second equality, $H_r(R_\alpha t)=\{t+(r+1)\alpha\}=H_{r+1}(t)$, so
by Notation \ref{not:cob},
\[
  G_v\circ R_\alpha-G_v=\sum_{r=0}^{k-1}\gamma_{v,r}(H_{r+1}-H_r) .
\]
If $i<j$ then $\gamma_{v,r}=1$ exactly for $i\le r<j$ and the sum is
$H_j-H_i$; if $j<i$ then $\gamma_{v,r}=-1$ exactly for $j\le r<i$ and the
sum is $-(H_i-H_j)$, which is again $H_j-H_i$.
\end{proof}

We note the relation of Lemma \ref{lem:coboundary} to the theory of
bounded remainder sets.

\begin{rem}\label{rem:brs}
Lemma \ref{lem:coboundary} is a special case of a classical result.  An
interval whose length lies in $\Z\alpha+\Z$ is a \emph{bounded remainder
set} for $R_\alpha$, meaning that its centred indicator function is an
$R_\alpha$-coboundary with bounded transfer function; this is due to Hecke
\cite{Hecke} and Ostrowski \cite{Ost27,Ost30},
while Kesten \cite{Kesten}
proved the converse, confirming a conjecture of Erd\H{o}s \& Sz\"usz.  The
intervals $I_v$ satisfy the condition, since
$|I_v|=c_j-c_i\equiv(i-j)\alpha$ modulo $1$.
What is needed below, however, is not the boundedness of a transfer
function but the explicit form of $G_v$, namely
$G_v=\pm\sum_{\min\{i,j\}\le r<\max\{i,j\}}H_r$, the sign being that of
$j-i$: it is this form which transports to $X_\alpha$ in Proposition
\ref{prop:centred-cylinder}.
For more on bounded
remainder sets, including transfer functions constructed in this way, see
\cite{GL}.
\end{rem}

\subsection{The identity on $X_\alpha$}

Lemma \ref{lem:coboundary} is a statement about functions on the circle.
The corresponding result on $X_\alpha$ is Proposition \ref{prop:centred-cylinder}
below, in which the transfer function $V_v$ is, up to sign, the sum of the
Birkhoff sums $S_r(\Phi-\alpha)$ over a range of consecutive values of $r$,
and in which a term that is an integer multiple of $\Phi-\alpha$ itself
appears alongside the coboundary.
What connects the
two is the following comparison between the functions $K_r$ on $X_\alpha$
and the functions $H_r$ on $\R$.

\begin{lem}\label{lem:Krcoding}
Let $\alpha\in(0,1)$ be irrational,
let $\bfb$ be the Sturmian coding,
and let $H_r$ and $K_r$ be as in Notation
\ref{not:cob}.  Then for all $r\in\N_0$ and $t\in\R$,
\[
  K_r\bigl(\bfb^{\pm}(t)\bigr)=H_r(t^{\pm})-H_0(t^{\pm}) .
\]
\end{lem}

\begin{proof}
The case $r=0$ is trivial.
Taking $k=1$ and $v=1$ in Notation
\ref{not:cuts} gives $I_1=[1-\alpha,1)=[c_1,c_0)$, so that $i=1$ and
$j=0$ in \eqref{eq:endpoints}, and
Lemma \ref{lem:coboundary} gives $\phi_1=H_0-H_1$.  Now
$\Phi(\bfb^{\pm}(t))=\bfb^{\pm}(t)_0=\ind_{I_1}(t^{\pm})$ by Lemma
\ref{lem:coding}, and $|I_1|=\alpha$, so that
$(\alpha-\Phi)(\bfb^{\pm}(t))=-\phi_1(t^{\pm})$, and therefore
\[
  (\alpha-\Phi)\bigl(\bfb^{\pm}(t)\bigr)=H_1(t^{\pm})-H_0(t^{\pm}) .
\]
Now $H_1(R^h_\alpha t)=H_{h+1}(t)$ and $H_0(R^h_\alpha t)=H_h(t)$, so by
\eqref{eq:shiftcoding} and \eqref{eq:Kr},
\[
  K_r\bigl(\bfb^{\pm}(t)\bigr)
  =\sum_{h=0}^{r-1}(\alpha-\Phi)\bigl(\bfb^{\pm}(R^h_\alpha t)\bigr)
  =\sum_{h=0}^{r-1}\bigl(H_{h+1}(t^{\pm})-H_h(t^{\pm})\bigr)
  =H_r(t^{\pm})-H_0(t^{\pm}) . \qedhere
\]
\end{proof}

Using Lemma \ref{lem:Krcoding}, we now transport the identity of Lemma
\ref{lem:coboundary} from the circle to $X_\alpha$.

\begin{prop}\label{prop:centred-cylinder}
Let $\alpha\in(0,1)$ be irrational, let $X_\alpha$ be the Sturmian subshift
of rotation number $\alpha$, with Sturmian measure $\mu_\alpha$, let $\Phi$
be the digit function, let $k\in\N$ and $v\in\{0,1\}^k$, let $B_v$ and
$V_v$ be as in Notation \ref{not:cob}, and set $C:=[v]_{X_\alpha}$.  Then
the centred indicator function of $C$ satisfies
\begin{equation}\label{eq:centred-cylinder}
  \ind_{C}-\mu_\alpha(C)=-B_v(\Phi-\alpha)+V_v\circ\sigma-V_v
  \qquad\text{on }X_\alpha .
\end{equation}
\end{prop}

\begin{proof}
We may assume $v\in\w_k(X_\alpha)$, since otherwise both sides vanish.
The key identity is \eqref{eq:Gvpullback} below, which expresses the circle
transfer function $G_v$ of Lemma \ref{lem:coboundary}, evaluated at
$t^{\pm}$, as the sum of $V_v(\bfb^{\pm}(t))$ and the multiple
$B_vH_0(t^{\pm})$.
The first term yields the coboundary
$V_v\circ\sigma-V_v$ of \eqref{eq:centred-cylinder}, and the second
yields the multiple of $\Phi-\alpha$, since $H_0$ evaluated in this way is
a transfer function for $\alpha-\Phi$.

Let $t\in\R$.  By Lemma \ref{lem:Krcoding},
$H_r(t^{\pm})=K_r(\bfb^{\pm}(t))+H_0(t^{\pm})$ for $0\le r<k$; multiplying
by $\gamma_{v,r}$, summing over $r$, and using Notation \ref{not:cob} and
\eqref{eq:Bv}, we obtain
\begin{equation}\label{eq:Gvpullback}
  G_v(t^{\pm})=V_v\bigl(\bfb^{\pm}(t)\bigr)+B_vH_0(t^{\pm}) .
\end{equation}
Since \eqref{eq:Gvpullback} holds for all $t\in\R$, it holds in particular
with $R_\alpha t$ in place of $t$; moreover
$\bfb^{\pm}(R_\alpha t)=\sigma\bfb^{\pm}(t)$ by \eqref{eq:shiftcoding}.
Writing $x:=\bfb^{\pm}(t)$, this gives
\begin{equation}\label{eq:Gvpullbackshift}
  G_v\bigl((R_\alpha t)^{\pm}\bigr)
  =V_v(\sigma x)+B_vH_0\bigl((R_\alpha t)^{\pm}\bigr) ,
\end{equation}
and subtracting \eqref{eq:Gvpullback} from \eqref{eq:Gvpullbackshift} gives
\begin{equation}\label{eq:Gvdifference}
  G_v\bigl((R_\alpha t)^{\pm}\bigr)-G_v(t^{\pm})
  =\bigl(V_v(\sigma x)-V_v(x)\bigr)
   +B_v\bigl(H_0\bigl((R_\alpha t)^{\pm}\bigr)-H_0(t^{\pm})\bigr) .
\end{equation}
The left-hand side of \eqref{eq:Gvdifference} is $\phi_v(t^{\pm})$, by
Lemma \ref{lem:coboundary} together with the fact that $R_\alpha$ is
increasing, so that if $s\to t^{\pm}$ then $R_\alpha s\to(R_\alpha t)^{\pm}$.
Using that
$H_1=H_0\circ R_\alpha$,
together with
Lemma
\ref{lem:Krcoding} and \eqref{eq:Kr},
we see that
\[
H_0\bigl((R_\alpha t)^{\pm}\bigr)-H_0(t^{\pm})
=
H_1(t^{\pm})-H_0(t^{\pm})
=K_1(x)= (\alpha-\Phi)(x),
\]
so
\eqref{eq:Gvdifference} becomes
\begin{equation}\label{eq:phiv}
  \phi_v(t^{\pm})
  =\bigl(V_v(\sigma x)-V_v(x)\bigr)+B_v(\alpha-\Phi)(x) .
\end{equation}
Finally, $\phi_v(t^{\pm})=\ind_{I_v}(t^{\pm})-|I_v|=\ind_C(x)-\mu_\alpha(C)$,
by Notation \ref{not:cob}, \eqref{eq:Iv} and Lemma \ref{lem:sturmian}(ii),
so \eqref{eq:phiv} gives
\[
  \ind_C(x)-\mu_\alpha(C)
  =-B_v(\Phi-\alpha)(x)+(V_v\circ\sigma-V_v)(x) .
\]
By Lemma \ref{lem:sturmian}(i) all points of $X_\alpha$ are of the form
$\bfb^{\pm}(t)$ for some $t\in\R$, so this proves
\eqref{eq:centred-cylinder}.
\end{proof}

We conclude this section with an intrinsic description, used in Corollary
\ref{cor:cf}, of the quantities $B_v$ and $\mu_\alpha([v]_{X_\alpha})$ in
Proposition \ref{prop:centred-cylinder}.

\begin{rem}\label{rem:intrinsic}
For $\alpha$ and $k$ as in Proposition \ref{prop:centred-cylinder}, and
$v\in\w_k(X_\alpha)$, the integer $B_v$ and the measure $\mu_\alpha([v]_{X_\alpha})$ are
determined by the Birkhoff sums $K_0,\ldots,K_k$ on $[v]$, in the following way.
Note that $K_0,\ldots,K_k$ are constant on $[v]$, since $K_r$ depends only on
$x_0,\ldots,x_{r-1}$.  For $t\in I_v$, Lemma \ref{lem:Krcoding} gives
$K_r(\bfb(t))=H_r(t)-H_0(t)$, which equals $-c_r$ if $c_r\le t$ and $1-c_r$ if
$c_r>t$.  These $k+1$ values are distinct, those of the first kind lying in
$(-1,0]$ and those of the second kind in $(0,1)$.  No cut of level $k$ lies in
the interior of $I_v$, so the smallest value is attained only at $r=i$ and the
greatest only at $r=j$ (the greatest value being $0$, attained at $r=0$, if
$j=0$), and their difference is $1-|I_v|$.
Writing $r_-(v)$ and $r_+(v)$ for the
indices in $\{0,1,\ldots,k\}$ at which $K_r$ is smallest and greatest on
$[v]$, we therefore have
\[
  B_v=r_+(v)-r_-(v),\qquad
  \mu_\alpha\bigl([v]_{X_\alpha}\bigr)
  =1-\Bigl(\max_{0\le r\le k}K_r-\min_{0\le r\le k}K_r\Bigr),
\]
the second of which is the formula of Bousch \& Mairesse
\cite[Prop.~1.2]{BM02} for the measure of a cylinder set under a Sturmian
measure, in the form given in
\cite[Lem.~3.13]{AJ}.
\end{rem}

\section{The Lipschitz cohomology of the Sturmian subshift $X_\alpha$}\label{sec:cohomology}

In this section we prove Theorem \ref{thm:D}, and deduce consequences for
the continuous transfer functions of Lipschitz functions on $X_\alpha$.

\subsection{Construction of $c_f$ and $\tilde f$}\label{subsec:construction}

Theorem \ref{thm:D} is proved in two parts.  Here we construct
$c_f$ and $\tilde f$ satisfying \eqref{eq:cohomology}, \eqref{eq:c-bound}
and \eqref{eq:u-bound}.  The two uniqueness statements are deferred to
Section \ref{subsec:rigidity}, and are equivalent to the statement that
$\Phi-\alpha$ is not a Lipschitz coboundary, which uses the analysis of
the function $\tau$ carried out there.

\begin{proof}[Proof of the existence statements of Theorem \ref{thm:D}]

Set $L:=|f|_{\lip(X_\alpha)}$.  The proof is in five steps.  Steps 1 and 2
expand $f$, up to an additive constant, as a series of centred indicator
functions of cylinder sets, with coefficients decaying like $\theta^{n}$;
this is where the Sturmian hypothesis is used, through
the fact that $\w_n(X_\alpha)$ has exactly $n+1$ elements.  Step 3 applies
Proposition \ref{prop:centred-cylinder} to each term of that series and sums the results,
which proves \eqref{eq:cohomology} and \eqref{eq:c-bound}.
Steps 4 and 5 rewrite the transfer function as a series in
$(\alpha-\Phi)\circ\sigma^{h}$ and bound its Lipschitz seminorm, which
gives \eqref{eq:u-bound}.
Once the uniqueness statements are available,
\eqref{eq:c-bound} and \eqref{eq:u-bound} are upgraded from statements
about the particular $c_f$ and $\tilde f$ constructed here to statements
about invariants of $f$.

\emph{Step 1: a refining sequence of partitions, and representatives.}
Set $\cA_0:=\{X_\alpha\}$ and, for $n\in\N$,
\[
  \cA_n:=\bigl\{[v]_{X_\alpha}:v\in\w_n(X_\alpha)\bigr\},
\]
a partition of $X_\alpha$ into $|\w_n(X_\alpha)|$ non-empty clopen sets.  We
claim that, for all $n\in\N_0$, exactly one atom of $\cA_n$ is the union of
two atoms of $\cA_{n+1}$, while all other atoms of $\cA_n$ are themselves
atoms of $\cA_{n+1}$.

Indeed, this is immediate for $n=0$, since $\cA_0$ has the single atom
$X_\alpha$ and $\cA_1$ the two atoms $[0]_{X_\alpha}$ and $[1]_{X_\alpha}$.
For $n\in\N$, each $w\in\w_{n+1}(X_\alpha)$ has a unique prefix of length
$n$, lying in $\w_n(X_\alpha)$; and each $v\in\w_n(X_\alpha)$ has at least
one extension $vb\in\w_{n+1}(X_\alpha)$, since $v$ occurs in some
$y\in X_\alpha$, say $v=y_{[i,i+n)}$, and then $y_{[i,i+n+1)}$ is such an
extension.  Grouping the words of length $n+1$ according to their prefix
therefore expresses $|\w_{n+1}(X_\alpha)|$ as a sum over
$v\in\w_n(X_\alpha)$ of terms equal to $1$ or $2$; these two numbers being
$n+2$ and $n+1$ by Lemma \ref{lem:sturmian}(ii), exactly one term equals
$2$.  For all other $v$, with unique extension $vb$, all points of
$[v]_{X_\alpha}$ have $(n+1)$st digit $b$, so $[v]_{X_\alpha}=[vb]_{X_\alpha}$
and this atom of $\cA_n$ is also an atom of $\cA_{n+1}$; the claim follows.

We choose a representative point in each atom of each partition $\cA_n$,
inductively.  The single atom of $\cA_0$ is $X_\alpha$; let its
representative be any point of $X_\alpha$.  Suppose representatives have
been chosen for the atoms of $\cA_n$, let $A$ be the unique atom of $\cA_n$
that is the union of two atoms of $\cA_{n+1}$, and let $C_n$ be whichever
of those two atoms does not contain the representative of $A$; the other
one contains it, and retains it as its own representative.  All atoms
of $\cA_{n+1}$ other than $C_n$ then already contain a representative,
which we retain; as the representative of $C_n$ we choose a new point.
Since $C_n$ is an atom of $\cA_{n+1}$, there is a unique
$v_n\in\w_{n+1}(X_\alpha)$ with $C_n=[v_n]_{X_\alpha}$.

\emph{Step 2: the expansion of $f$ into centred indicator functions.}
Let $f_n:X_\alpha\to\R$ be constant on each atom of $\cA_n$, with value
there the value of $f$ at the chosen representative of that atom.
Since all atoms of $\cA_{n+1}$ other than $C_n$ retain the representative
of the atom of $\cA_n$ containing it, $f_{n+1}$ and $f_n$ agree off $C_n$,
and there is $d_n\in\R$ with
\begin{equation}\label{eq:successive-difference}
  f_{n+1}-f_n=d_n\ind_{C_n} .
\end{equation}
Here $d_n$ is the difference of the values of $f$ at the representative of
$C_n$ and at the representative of $A$.  These two points both lie in $A$,
an atom of $\cA_n$, so they agree on their coordinates $0,1,\ldots,n-1$ and
are therefore at distance at most $\theta^{n}$, giving
\begin{equation}\label{eq:dn}
  |d_n|\le L\theta^{n} .
\end{equation}
Similarly, each $x\in X_\alpha$ agrees on its first $n$ coordinates with the
representative of the atom of $\cA_n$ containing it, so
\begin{equation}\label{eq:fn-uniform}
  \|f-f_n\|_{C(X_\alpha)}\le L\theta^{n} .
\end{equation}

Summing the differences \eqref{eq:successive-difference} gives
\begin{equation}\label{eq:partial-sum}
  f_N=f_0+\sum_{n=0}^{N-1}d_n\ind_{C_n}\qquad(N\ge1) .
\end{equation}
Since $\cA_0$ has a single atom, $f_0$ is constant; integrating
\eqref{eq:partial-sum} with respect to $\mu_\alpha$, and subtracting this
from \eqref{eq:partial-sum},
we have
\begin{equation}\label{eq:expression-minus-integral}
  f_N-\int f_N\,d\mu_\alpha
  =\sum_{n=0}^{N-1}d_n\bigl(\ind_{C_n}-\mu_\alpha(C_n)\bigr) .
\end{equation}
By \eqref{eq:fn-uniform}, the left-hand side of \eqref{eq:expression-minus-integral} converges uniformly to
$f-\int f\,d\mu_\alpha$ as $N\to\infty$, hence
\begin{equation}\label{eq:centered-expansion}
  f-\int f\,d\mu_\alpha
  =\sum_{n=0}^{\infty}d_n\bigl(\ind_{C_n}-\mu_\alpha(C_n)\bigr)
\end{equation}
uniformly on $X_\alpha$.

\emph{Step 3: decomposition of each term.}
The word $v_n$ has length $n+1$, so Proposition \ref{prop:centred-cylinder},
applied with $k=n+1$ and $C=C_n$, gives an integer $B_n:=B_{v_n}$ and a
function $V_n:=V_{v_n}$ with
\begin{equation}\label{eq:Cn-decomp}
  \ind_{C_n}-\mu_\alpha(C_n)=-B_n(\Phi-\alpha)+V_n\circ\sigma-V_n
  \qquad\text{on }X_\alpha,
\end{equation}
where $\gamma_{n,r}:=\gamma_{v_n,r}$ and, by Notation \ref{not:cob},
\eqref{eq:Bvalue} and \eqref{eq:Vvbound},
\begin{equation}\label{eq:BnVn}
  |B_n|\le n+1,\qquad V_n=\sum_{r=0}^{n}\gamma_{n,r}K_r,\qquad
  \gamma_{n,r}\in\{-1,0,1\},\qquad
  \|V_n\|_{C(\Sigma)}\le\tfrac12n(n+1) .
\end{equation}
Define
\begin{equation}\label{eq:cfuf}
  c_f:=-\sum_{n=0}^{\infty}d_nB_n,\qquad
  \tilde f:=\sum_{n=0}^{\infty}d_nV_n .
\end{equation}
The first series converges absolutely, by \eqref{eq:dn} and
\eqref{eq:BnVn}, and
\[
  |c_f|\le L\sum_{n=0}^{\infty}(n+1)\theta^{n}=\frac{L}{(1-\theta)^{2}},
\]
which is \eqref{eq:c-bound}.  The second series converges
absolutely in $C(\Sigma)$, since by \eqref{eq:dn} and \eqref{eq:BnVn}
\begin{equation}\label{eq:uf-absolute}
  \sum_{n=0}^{\infty}|d_n|\,\|V_n\|_{C(\Sigma)}
  \le\frac L2\sum_{n=0}^{\infty}n(n+1)\theta^{n}<\infty ,
\end{equation}
so $\tilde f$ is a well defined continuous function on $\Sigma$, and in
particular $\sum_nd_n(V_n\circ\sigma)$ converges uniformly to
$\tilde f\circ\sigma$, uniform convergence being preserved by composition
with $\sigma$.

Now substitute \eqref{eq:Cn-decomp} into \eqref{eq:centered-expansion}.
Each of the three resulting series, namely those with terms
$-d_nB_n(\Phi-\alpha)$, $d_nV_n\circ\sigma$ and $-d_nV_n$, converges
absolutely and uniformly on $X_\alpha$ by \eqref{eq:dn}, \eqref{eq:BnVn} and
\eqref{eq:uf-absolute}, so they may be separated and summed individually.
This gives
\[
  f-\int f\,d\mu_\alpha
  =\Bigl(-\sum_{n\ge0}d_nB_n\Bigr)(\Phi-\alpha)
  +\Bigl(\sum_{n\ge0}d_nV_n\Bigr)\circ\sigma-\sum_{n\ge0}d_nV_n
  =c_f(\Phi-\alpha)+\tilde f\circ\sigma-\tilde f
\]
on $X_\alpha$, which is \eqref{eq:cohomology}.  It remains to bound
$|\tilde f|_{\lip(X_\alpha)}$, and in particular to show that $\tilde f$ is
Lipschitz.

\emph{Step 4: rewriting $\tilde f$ as a series in
$(\alpha-\Phi)\circ\sigma^{h}$.}
By \eqref{eq:BnVn} and \eqref{eq:Kr},
\[
  V_n=\sum_{r=1}^{n}\gamma_{n,r}K_r
  =\sum_{r=1}^{n}\gamma_{n,r}\sum_{h=0}^{r-1}(\alpha-\Phi)\circ\sigma^{h} .
\]
 Substituting this into the
definition \eqref{eq:cfuf} of $\tilde f$ gives a triple sum whose terms
have supremum norm at most $|d_n|$, by \eqref{eq:Krbound}, and which has at
most $\tfrac12n(n+1)$ terms for each $n$; by \eqref{eq:uf-absolute} it
therefore converges absolutely in $C(\Sigma)$, and may be summed in any
order.  Gathering the coefficient of each $(\alpha-\Phi)\circ\sigma^{h}$,
and noting that $d_n\gamma_{n,r}$ contributes to the index $h$ exactly when
$r>h$, we see that
\begin{equation}\label{eq:bh-expansion}
  \tilde f=-\sum_{h=0}^{\infty}b_h(\Phi-\alpha)\circ\sigma^{h},\qquad
  b_h:=\sum_{r=h+1}^{\infty}\ \sum_{n=r}^{\infty}d_n\gamma_{n,r} .
\end{equation}
Since $|\gamma_{n,r}|\le1$, two applications of \eqref{eq:dn} give
\begin{equation}\label{eq:bh-bound}
  |b_h|\le L\sum_{r=h+1}^{\infty}\ \sum_{n=r}^{\infty}\theta^{n}
  =\frac{L}{1-\theta}\sum_{r=h+1}^{\infty}\theta^{r}
  =\frac{L\theta^{h+1}}{(1-\theta)^{2}} .
\end{equation}

\emph{Step 5: the Lipschitz bound.}
Let $x,y\in X_\alpha$ with $x\ne y$, and let $\ell\ge0$ be the least index
with $x_\ell\ne y_\ell$, so that $d_\theta(x,y)=\theta^{\ell}$.  Since
$(\Phi-\alpha)(\sigma^{h}z)=z_h-\alpha$, the $h$th terms of
\eqref{eq:bh-expansion} at $x$ and at $y$ differ by
$|b_h|\,|x_h-y_h|$, which vanishes for $h<\ell$ and is at most $|b_h|$ for
$h\ge\ell$.  So by \eqref{eq:bh-bound},
\[
  |\tilde f(x)-\tilde f(y)|\le\sum_{h=\ell}^{\infty}|b_h|
  \le\frac{L}{(1-\theta)^{2}}\sum_{h=\ell}^{\infty}\theta^{h+1}
  =\frac{\theta L}{(1-\theta)^{3}}\,d_\theta(x,y) ,
\]
and taking the supremum over $x\ne y$ gives
$\tilde f\in\lip(X_\alpha,d_\theta)$ together with \eqref{eq:u-bound}.
\end{proof}

\subsection{Rigidity, and the uniqueness of $c_f$}\label{subsec:rigidity}

The uniqueness statements of Theorem \ref{thm:D} are consequences of the
following properties of the supremum of the Birkhoff sums of $\Phi-\alpha$.

\begin{prop}\label{prop:rigidity}
Let $\alpha\in(0,1)$ be irrational, let $X_\alpha$ be the Sturmian subshift
of rotation number $\alpha$, with Sturmian measure $\mu_\alpha$,
let $\bfb$ be the Sturmian coding, let $\Phi$ be the digit function, and
let $H_0$ be as in Notation \ref{not:cob}.  Define
\[
  \tau:=\sup_{r\ge0}S_r(\Phi-\alpha)\qquad\text{on }X_\alpha .
\]
Then:
\begin{enumerate}
\item[(i)] $\tau(\bfb^{\pm}(t))=H_0(t^{\pm})$ for all $t\in\R$; in
particular $0\le\tau\le1$;
\item[(ii)] $\alpha-\Phi=\tau\circ\sigma-\tau$ on $X_\alpha$;
\item[(iii)] for all $k\in\N$ and $w\in\w_k(X_\alpha)$ the image
$\tau\bigl([w]_{X_\alpha}\bigr)$ is a closed interval of length
$\mu_\alpha([w]_{X_\alpha})$; consequently $\tau$ is continuous on
$X_\alpha$, but $\tau\notin\lip(X_\alpha,d_\theta)$ for all
$\theta\in(0,1)$;
\item[(iv)] for all $\theta\in(0,1)$ there is no
$u\in\lip(X_\alpha,d_\theta)$ with $\Phi-\alpha=u\circ\sigma-u$ on
$X_\alpha$.
\end{enumerate}
Consequently, for all $\theta\in(0,1)$, the two uniqueness statements of
Theorem \ref{thm:D} hold,
the functional $f\mapsto c_f$ is as described there, and the quotient of
$\lip(X_\alpha,d_\theta)$ by the constants and the Lipschitz coboundaries
is one-dimensional, spanned by the equivalence class of $\Phi$.
\end{prop}

\begin{proof}
(i) By \eqref{eq:Kr} and Lemma \ref{lem:Krcoding},
\[
  S_r(\Phi-\alpha)\bigl(\bfb^{\pm}(t)\bigr)
  =-K_r\bigl(\bfb^{\pm}(t)\bigr)
  =H_0(t^{\pm})-H_0\bigl((R^r_\alpha t)^{\pm}\bigr) ,
\]
the last equality because $H_r=H_0\circ R^r_\alpha$ on $\R$, and
$R^r_\alpha$ is translation by $r\alpha$, so that if $s\to t^{\pm}$ then
$R^r_\alpha s\to(R^r_\alpha t)^{\pm}$.
Each $H_0((R^r_\alpha t)^{\pm})$ lies in $[0,1]$, and since $\alpha$ is
irrational the orbit $\{\{t+r\alpha\}:r\ge0\}$ is dense in $[0,1)$, so the
infimum of $H_0((R^r_\alpha t)^{\pm})$ over $r\ge0$ is $0$.  Taking the
supremum over $r\ge0$ therefore gives $\tau(\bfb^{\pm}(t))=H_0(t^{\pm})$,
and the bounds $0\le\tau\le1$ follow.

(ii) By (i) and \eqref{eq:shiftcoding},
$\tau(\sigma\bfb^{\pm}(t))-\tau(\bfb^{\pm}(t))
=H_0((R_\alpha t)^{\pm})-H_0(t^{\pm})$, which equals
$H_1(t^{\pm})-H_0(t^{\pm})=K_1(\bfb^{\pm}(t))$ by Lemma \ref{lem:Krcoding},
as $H_1=H_0\circ R_\alpha$; and $K_1=\alpha-\Phi$.

(iii) Let $x,y\in X_\alpha$ with $x_{[0,k)}=y_{[0,k)}=:w$, necessarily in
$\w_k(X_\alpha)$, and write $I_w=[c_i,c_j)$ as in \eqref{eq:endpoints}.
If $x=\bfb^{+}(t)$ then $\bfb(t)_{[0,k)}=w$, so $\{t\}\in I_w$ and
$\tau(x)=\{t\}$.  If $x=\bfb^{-}(t)$ with $t\notin\Z$ then $\{t\}$ lies in the
closure of $I_w$ and $\tau(x)=H_0(t^{-})=\{t\}$; and if $x=\bfb^{-}(t)$ with
$t\in\Z$ then $x=\bfa$, so $w=\bfa_{[0,k)}$, whose interval has right endpoint
$1$, and $\tau(x)=1$. In
all cases $\tau(x)$ lies in the closure of $I_w$ taken in $[0,1]$, so
$|\tau(x)-\tau(y)|\le|I_w|$.  The level-$k$ partitions refine as $k$
increases, the cuts of level $k$ being among those of level $k+1$, and
$\{c_j:j\ge0\}$ is dense in $[0,1)$ since $\alpha$ is irrational; so
$\max_w|I_w|\to0$ as $k\to\infty$, and $\tau$ is continuous.

Conversely, for $t\in I_w$ the point $\bfb^{+}(t)$ lies in
$[w]_{X_\alpha}$ with $\tau(\bfb^{+}(t))=t$ by (i), while $\bfb^{-}(c_j)$
lies in $[w]_{X_\alpha}$ with $\tau(\bfb^{-}(c_j))=c_j$, the right endpoint
being taken as $1$ when $j=0$.  So $\tau([w]_{X_\alpha})$ is the closed
interval with endpoints $c_i$ and $c_j$, whose length is $|I_w|$, that is
$\mu_\alpha([w]_{X_\alpha})$ by Lemma \ref{lem:sturmian}(ii).  Since
$\w_k(X_\alpha)$ has $k+1$ elements, whose cylinder sets have
$\mu_\alpha$-measures summing to $1$, some $w\in\w_k(X_\alpha)$ has
$\mu_\alpha([w]_{X_\alpha})\ge(k+1)^{-1}$;
choosing $x,y\in[w]_{X_\alpha}$ with $\tau(x)$ and $\tau(y)$ the two
endpoints of the interval $\tau([w]_{X_\alpha})$ gives
$d_\theta(x,y)\le\theta^{k}$ and
$|\tau(x)-\tau(y)|=\mu_\alpha([w]_{X_\alpha})\ge(k+1)^{-1}$, so that
$|\tau|_{\lip(X_\alpha)}\ge\theta^{-k}(k+1)^{-1}$ for all $k$.

(iv) Suppose $u\in\lip(X_\alpha,d_\theta)$ satisfies
$\Phi-\alpha=u\circ\sigma-u$.  With (ii) this gives
$(u+\tau)\circ\sigma=u+\tau$, so the continuous function $u+\tau$ is
$\sigma$-invariant.  The set on which a continuous $\sigma$-invariant
function attains its maximum is non-empty, closed, and $\sigma$-invariant,
hence equals $X_\alpha$ by minimality; so $u+\tau$ is constant and $\tau$ is
Lipschitz, contradicting (iii).

Finally, we prove the uniqueness statements.  As noted in Remark \ref{rem:uniqueness}, the uniqueness of $c_f$ is equivalent to the statement
that $\Phi-\alpha$ is not a Lipschitz coboundary on $X_\alpha$, which is
(iv).  Given $c_f$, the function $\tilde f$ is unique up to an additive
constant: if $\tilde f$ and $\tilde f'$ both satisfy
\eqref{eq:cohomology} then $(\tilde f'-\tilde f)\circ\sigma=\tilde
f'-\tilde f$, so $\tilde f'-\tilde f$ is constant by the argument used in
(iv).

For the functional, let $f,g\in\lip(X_\alpha,d_\theta)$.  Adding the
decompositions \eqref{eq:cohomology} of $f$ and of $g$ gives a decomposition for
$f+g$, so $c_{f+g}=c_f+c_g$ by uniqueness, and similarly $c_{\lambda
f}=\lambda c_f$; the functional vanishes on the constants and on
$h\circ\sigma-h$ for $h\in\lip(X_\alpha,d_\theta)$, because in those two
cases \eqref{eq:cohomology} holds with $c_f=0$; and $c_\Phi=1$ because
\eqref{eq:cohomology} holds for $f=\Phi$ with $c_f=1$ and $\tilde f=0$.
Since, by the existence statements of Theorem \ref{thm:D}, all
$f\in\lip(X_\alpha,d_\theta)$ are sums of
$c_f\Phi$, a constant and a Lipschitz coboundary, these properties determine
the functional, and the equivalence class of $\Phi$ spans the quotient.  This
class is non-zero: if $\Phi=c+u\circ\sigma-u$ with $c\in\R$ and
$u\in\lip(X_\alpha,d_\theta)$, then integrating with respect to $\mu_\alpha$
gives $c=\alpha$, contradicting (iv).
\end{proof}

We are now able to complete the proof of Theorem \ref{thm:D}.

\begin{proof}[Proof of Theorem \ref{thm:D}]
The constant $c_f$ and the function $\tilde f$ constructed in Section
\ref{subsec:construction} satisfy \eqref{eq:cohomology},
\eqref{eq:c-bound} and \eqref{eq:u-bound}, and Proposition
\ref{prop:rigidity} gives the two uniqueness statements, together with the
properties of the functional $f\mapsto c_f$ and the description of the
quotient.
\end{proof}

By the uniqueness of $c_f$, the series in \eqref{eq:cfuf} gives $c_f$ for
all choices of the representatives in Step 1 of Section
\ref{subsec:construction}.

\begin{rem}\label{rem:linsine}
Part (ii) of Proposition \ref{prop:rigidity} is a particular case of a general
result, since the Birkhoff sums of $\Phi-\alpha$ are uniformly bounded,
Sturmian sequences being balanced (Lemma \ref{lem:sturmian}(iv)).
That a continuous function with uniformly bounded Birkhoff sums over a
minimal system is a continuous coboundary is due to Browder
\cite[Thm.~1]{Browder}, generalising Gottschalk \& Hedlund
\cite[Thm.~14.11]{GH}, and that the supremum of the Birkhoff sums then gives a
continuous transfer function is due to Lin \& Sine \cite[Thm.~9]{linsine}; see
also Lyubich \cite{lyubich}.
The direct argument above is included because what is needed is the
identification of $\tau$ in part (i): parts (ii) and (iii) follow from it
and together imply part (iv), from which the uniqueness statements of
Theorem \ref{thm:D} are deduced.
\end{rem}

We now deduce two consequences of Theorem \ref{thm:D} and Proposition
\ref{prop:rigidity}.  The first identifies $c_f$ when $f$ is the indicator
function of a cylinder set, and relates $\tilde f$ to the supremum of the
Birkhoff sums of $f-\int f\,d\mu_\alpha$.

\begin{cor}\label{cor:cf}
Let $\alpha\in(0,1)$ be irrational, let $\theta\in(0,1)$, let $X_\alpha$ be
the Sturmian subshift of rotation number $\alpha$ equipped with the metric
$d_\theta$, with Sturmian measure $\mu_\alpha$, 
let $\Phi$ be the digit function, let $\tau$ be as in Proposition
\ref{prop:rigidity}, and let $f\mapsto c_f$ be the functional of Theorem
\ref{thm:D}.  Then:
\begin{enumerate}
\item[(i)] for all $k\in\N$ and $v\in\{0,1\}^k$ the indicator function of
the cylinder set $[v]_{X_\alpha}$ has $c_{\ind_{[v]_{X_\alpha}}}=-B_v$,
with $B_v$ as in Notation \ref{not:cob}; and if $v\in\w_k(X_\alpha)$ then
\[
  -B_v=r_-(v)-r_+(v) ,
\]
with $r_\pm(v)$ as in Remark \ref{rem:intrinsic};
\item[(ii)] for all $f\in\lip(X_\alpha,d_\theta)$ the function
$\tau_f:=\sup_{n\ge0}S_n\bigl(f-\int f\,d\mu_\alpha\bigr)$ is real-valued and
continuous on $X_\alpha$, with $\tau_\Phi=\tau$, and $\tilde f-c_f\tau+\tau_f$
is constant for all $\tilde f$ as in Theorem \ref{thm:D}; consequently
$c\,\tau-\tau_f$ is continuous for all $c\in\R$, it is Lipschitz if and only
if $c=c_f$, and for $c\ne c_f$ it is not H\"older continuous.
\end{enumerate}
\end{cor}

\begin{proof}
(i)  The function $V_v$ of Notation \ref{not:cob} depends on finitely many
coordinates, so is Lipschitz, and Proposition \ref{prop:centred-cylinder} is
therefore an instance of \eqref{eq:cohomology} with
$f=\ind_{[v]_{X_\alpha}}$, $c_f=-B_v$ and $\tilde f=V_v$.  The first
equality follows from the uniqueness of $c_f$, and the second from Remark
\ref{rem:intrinsic}.

(ii)  By Proposition \ref{prop:rigidity}(ii) we have
$\Phi-\alpha=-(\tau\circ\sigma-\tau)$, so \eqref{eq:cohomology} may be
rewritten as
\[
  f-\int f\,d\mu_\alpha=w\circ\sigma-w,\qquad w:=\tilde f-c_f\tau ,
\]
with $w$ continuous.  Hence
$S_n(f-\int f\,d\mu_\alpha)=w\circ\sigma^{n}-w$ for all $n\ge0$, so that
$\tau_f(x)=\sup_{n\ge0}w(\sigma^{n}x)-w(x)$.  For each $x\in X_\alpha$ the
orbit $\{\sigma^nx:n\ge0\}$ is dense, by minimality, and $w$ is continuous,
so $\sup_{n\ge0}w(\sigma^{n}x)=\max w$; thus $\tau_f=\max w-w$, which is
real-valued and continuous, and $\tilde f=w+c_f\tau=c_f\tau-\tau_f+\max w$, so
that $c_f\tau-\tau_f$ is Lipschitz.  For the remaining statements,
$c\,\tau-\tau_f$ is continuous for all $c\in\R$ because $\tau$ and $\tau_f$
are, and
\[
  c\,\tau-\tau_f=(c-c_f)\tau+\bigl(\tilde f-\max w\bigr) .
\]
For $\theta'\in[\theta,1)$ we have $d_\theta\le d_{\theta'}$ pointwise, so
$\tilde f-\max w$ is $d_{\theta'}$-Lipschitz, and hence $c\,\tau-\tau_f$ is
$d_{\theta'}$-Lipschitz if and only if $(c-c_f)\tau$ is, which by
Proposition \ref{prop:rigidity}(iii) happens only when $c=c_f$.  Since the
H\"older classes relative to $d_\theta$ are the $d_{\theta'}$-Lipschitz
classes, $\theta'\in(0,1)$, and the $d_{\theta'}$-Lipschitz class is
contained in the $d_\theta$-Lipschitz class for $\theta'\in(0,\theta)$, it
follows that $c\,\tau-\tau_f$ is non-H\"older when $c\ne c_f$.
\end{proof}

The second describes the continuous transfer functions of Lipschitz
functions on $X_\alpha$ with zero integral, and shows how Liv\v{s}ic
regularity fails on $X_\alpha$.

\begin{cor}[Failure of Liv\v{s}ic regularity on $X_\alpha$]\label{cor:livsic}
Let $\alpha\in(0,1)$ be irrational, let $\theta\in(0,1)$, let $X_\alpha$ be
the Sturmian subshift of rotation number $\alpha$ equipped with the metric
$d_\theta$, with Sturmian measure $\mu_\alpha$,
let $\Phi$ be the digit function,
and let $\tau$ be as in Proposition
\ref{prop:rigidity}.  Let $f\in\lip(X_\alpha,d_\theta)$ satisfy
$\int f\,d\mu_\alpha=0$, and let $c_f$ and $\tilde f$ be as in Theorem
\ref{thm:D}.  Then $f$ is a continuous coboundary on $X_\alpha$, all
continuous transfer functions for $f$ differ from $\tilde f-c_f\tau$ by
constants, and:
\begin{enumerate}
\item[(i)] if $c_f=0$ then the continuous transfer functions for $f$ are
Lipschitz;
\item[(ii)] if $c_f\ne0$ then no continuous transfer function for $f$ is
H\"older continuous.
\end{enumerate}
In particular $\Phi-\alpha$ is a Lipschitz function on $X_\alpha$ which is
a continuous coboundary but not a H\"older coboundary.
\end{cor}

\begin{proof}
By Proposition \ref{prop:rigidity}(ii) we have
$\Phi-\alpha=\tau-\tau\circ\sigma$, so that \eqref{eq:cohomology} becomes
\[
  f=c_f(\Phi-\alpha)+\tilde f\circ\sigma-\tilde f=w\circ\sigma-w,
  \qquad w:=\tilde f-c_f\tau ,
\]
and $w$ is continuous because $\tilde f$ is Lipschitz and $\tau$ is
continuous.  If $w'$ is another continuous transfer function for $f$ then
$w-w'$ is a continuous $\sigma$-invariant function on the minimal system
$X_\alpha$, hence constant.  Since $\tilde f$ is Lipschitz, $w$ is
Lipschitz if and only if $c_f\tau$ is, which by Proposition
\ref{prop:rigidity}(iii) happens only when $c_f=0$; and for $c_f\ne0$ the
function $w$ is not H\"older continuous, by the argument given for
Corollary \ref{cor:cf}(ii).  The last statement is the case $f=\Phi-\alpha$,
for which $c_f=1$ by Theorem \ref{thm:D}.
\end{proof}

\section{Stability of the Sturmian maximizing measure}\label{sec:proofs}

In this section we prove Theorems \ref{thm:A} and \ref{thm:C}, and
Corollary \ref{cor:B}.  We first isolate the mechanism behind Theorem
\ref{thm:C}: a decomposition of the kind given by Theorem \ref{thm:D}, on a
closed invariant subset $X$ of a subshift $Y$, together with a function
$\phi$ on $Y$ whose integral is maximized by an invariant measure on $X$,
gives an inequality of the form \eqref{eq:gap} for the function
$\phi-d_\theta(\cdot,X)$.  The same argument is used again in Section
\ref{sec:density}.
We shall use the fact that $\sigma$ is
Lipschitz on each $(\Sigma_m,d_\theta)$, with
\begin{equation}\label{eq:Lip-shift}
  d_\theta(\sigma x,\sigma y)\le\theta^{-1}d_\theta(x,y)
  \qquad(x,y\in\Sigma_m) .
\end{equation}

\begin{lem}\label{lem:rigid-stability}
Let $Y$ be a subshift equipped with a metric $d_\theta$, with Wasserstein
distance $W_\theta$, let $X\subseteq Y$
be a non-empty closed set with $\sigma(X)=X$, let $\mu\in\m(X,\sigma)$,
regarded also as an element of $\m(Y,\sigma)$, and let
$\phi\in\lip(Y,d_\theta)$ satisfy $\int\phi\,d\nu\le\int\phi\,d\mu$ for all
$\nu\in\m(Y,\sigma)$.  Let $A,B\ge0$ be such that for all
$f\in\lip(X,d_\theta)$ there are $c\in\R$ and $u\in\lip(X,d_\theta)$ with
$|c|\le A|f|_{\lip(X)}$, $|u|_{\lip(X)}\le B|f|_{\lip(X)}$ and
\begin{equation}\label{eq:abstract-decomp}
  f-\int f\,d\mu=c\Bigl(\phi-\int\phi\,d\mu\Bigr)+u\circ\sigma-u
  \qquad\text{on }X .
\end{equation}
Set $\Psi_X:=d_\theta(\cdot,X)$ on $Y$, and
$C:=\max\{A,\,1+A|\phi|_{\lip(Y)}+(1+\theta^{-1})B\}$.  Then
\begin{equation}\label{eq:abstract-gap}
  \int(\phi-\Psi_X)\,d\mu-\int(\phi-\Psi_X)\,d\nu\ \ge\ C^{-1}\,
  W_\theta(\nu,\mu)\qquad\text{for all }\nu\in\m(Y,\sigma) .
\end{equation}
\end{lem}

\begin{proof}
Let $\nu\in\m(Y,\sigma)$ and $f\in\lip(Y,d_\theta)$, and set
$L:=|f|_{\lip(Y)}$.  The restriction of $f$ to $X$ has Lipschitz seminorm at
most $L$, so there are $c\in\R$ and $u\in\lip(X,d_\theta)$ satisfying
\eqref{eq:abstract-decomp}, with $|c|\le AL$ and $|u|_{\lip(X)}\le BL$.  By
McShane's extension theorem \cite{McShane} (see \cite[Thm.~1.33]{W2018}),
$u$ extends from the subset $X$ of the metric space $(Y,d_\theta)$ to a
function on $Y$ with the same Lipschitz seminorm, which we continue to
denote by $u$.  Define $r:Y\to\R$ by
\begin{equation}\label{eq:residual}
  r:=f-\int f\,d\mu-c\Bigl(\phi-\int\phi\,d\mu\Bigr)-(u\circ\sigma-u) ,
\end{equation}
so that $r$ vanishes on $X$, by \eqref{eq:abstract-decomp}.  By
\eqref{eq:Lip-shift} we have
$|u\circ\sigma|_{\lip(Y)}\le\theta^{-1}|u|_{\lip(Y)}$, so
$|r|_{\lip(Y)}\le L+|c|\,|\phi|_{\lip(Y)}+(1+\theta^{-1})BL\le CL$.  
For
all $y\in Y$ and $x\in X$ it follows that
$|r(y)|=|r(y)-r(x)|\le CL\,d_\theta(y,x)$, and taking the infimum over
$x\in X$ gives
\begin{equation}\label{eq:residual-distance}
  |r(y)|\le CL\,\Psi_X(y)\qquad(y\in Y) .
\end{equation}
Integrating \eqref{eq:residual} with respect to $\nu$, and using that
$\int(u\circ\sigma-u)\,d\nu=0$, gives
\[
  \int f\,d\nu-\int f\,d\mu
  =c\Bigl(\int\phi\,d\nu-\int\phi\,d\mu\Bigr)+\int r\,d\nu .
\]
As $\int\phi\,d\mu-\int\phi\,d\nu\ge0$ and $|c|\le AL\le CL$, it follows
from \eqref{eq:residual-distance} that
\[
  \int f\,d\nu-\int f\,d\mu
  \le CL\Bigl(\Bigl(\int\phi\,d\mu-\int\phi\,d\nu\Bigr)+\int\Psi_X\,d\nu\Bigr) .
\]
Both $\int\phi\,d\mu-\int\phi\,d\nu$ and $\int\Psi_X\,d\nu$ are
non-negative, and as $\int\Psi_X\,d\mu=0$ their sum is the left-hand side
of \eqref{eq:abstract-gap}.  Hence $\int f\,d\nu-\int f\,d\mu$ is at most $C$
times this left-hand side whenever $L\le1$, and taking the supremum over all
such $f$ gives \eqref{eq:abstract-gap}, by \eqref{eq:wass}.
\end{proof}

We are now able to prove Theorem \ref{thm:C}, by combining Lemma
\ref{lem:rigid-stability} with Proposition \ref{prop:Phimax} and Theorem
\ref{thm:D}.

\begin{proof}[Proof of Theorem \ref{thm:C}]
We apply Lemma \ref{lem:rigid-stability} with $Y:=Y_\alpha$,
$X:=X_\alpha$, $\mu:=\mu_\alpha$ and $\phi:=\Phi$, for which
$\Psi_X=\Psi$ and $\phi-\Psi_X=F$.  Here $X_\alpha\subseteq Y_\alpha$ by
Lemma \ref{lem:sturmian}(v), and $\sigma(X_\alpha)=X_\alpha$ as $X_\alpha$
is minimal, by Lemma \ref{lem:sturmian}(i).  By Proposition
\ref{prop:Phimax}, $\int\Phi\,d\nu\le\int\Phi\,d\mu_\alpha=\alpha$ for all
$\nu\in\m(Y_\alpha,\sigma)$.  So the decomposition \eqref{eq:cohomology}
of Theorem \ref{thm:D} is of the form \eqref{eq:abstract-decomp}, with
$c:=c_f$ and $u:=\tilde f$, and by \eqref{eq:c-bound} and
\eqref{eq:u-bound} we may take $A:=(1-\theta)^{-2}$ and
$B:=\theta(1-\theta)^{-3}$.  As $|\Phi|_{\lip(Y_\alpha)}\le1$, by
\eqref{eq:FLip}, the constant $C$ of Lemma \ref{lem:rigid-stability}
satisfies
\[
  C\le1+A+(1+\theta^{-1})B
  =1+\frac{1}{(1-\theta)^{2}}+\frac{1+\theta}{(1-\theta)^{3}}
  =1+\frac{2}{(1-\theta)^{3}}=C_\theta ,
\]
so \eqref{eq:abstract-gap} gives the inequality of Theorem \ref{thm:C}.
\end{proof}

Theorem \ref{thm:A} now follows from Theorem \ref{thm:C} and Lemma
\ref{lem:stability}.

\begin{proof}[Proof of Theorem \ref{thm:A}]
Theorem \ref{thm:C} is the hypothesis \eqref{eq:gap} of Lemma
\ref{lem:stability}, with $f=F$, $\mu=\mu_\alpha$ and
$\delta=C_\theta^{-1}$; so $\Mmax(F+g)=\{\mu_\alpha\}$ for all
$g\in\lip(Y_\alpha,d_\theta)$ with $|g|_{\lip(Y_\alpha)}<C_\theta^{-1}$,
that is $\Mmax(f)=\{\mu_\alpha\}$ for all $f\in U$.

The set $U$ is non-empty, since $F\in U$, and open in
$\lip(Y_\alpha,d_\theta)$, since
$|f+h-F|_{\lip(Y_\alpha)}\le|f-F|_{\lip(Y_\alpha)}+\|h\|_{\lip(Y_\alpha)}$
for all $h$.
The measure $\mu_\alpha$ is not a periodic orbit measure, as its support
is the infinite minimal set $X_\alpha$.  Hence $U\cap\cP(Y_\alpha)=\emptyset$, so
$\cP(Y_\alpha)$ is not dense in $\lip(Y_\alpha,d_\theta)$, therefore $Y_\alpha$
does not have typical periodic optimization.
\end{proof}

Corollary \ref{cor:B} is now immediate.

\begin{proof}[Proof of Corollary \ref{cor:B}]
Let $\alpha\in(0,1)$ be irrational.  The Sturmian beta-shift $Y_\alpha$ has
the specification property by Lemma \ref{lem:spec}, and does not have
typical periodic optimization by Theorem~\ref{thm:A}.
\end{proof}

Finally, we verify the properties of the magic Morse shift stated in the
Introduction, in particular that it does not have specification.

\begin{rem}\label{rem:magic-morse}
The magic Morse shift of \cite[\S 12]{HJXZ} is synchronized but does not
have the specification property.
To recall its definition, let $\mathbf m=0110100110010110\ldots$ be the
Morse sequence, that is, the fixed point beginning with $0$ of the
substitution $0\mapsto01$, $1\mapsto10$, and for $k\ge0$ let $p_k$ be its
prefix of length $2^k$, so that $p_{k+1}=p_k\overline{p_k}$, where
$\overline{v_1\cdots v_n}:=\bar{v}_1\cdots\bar{v}_n$, with $\bar{0}:=1$ and
$\bar{1}:=0$.
The magic Morse shift is the closure $X$
in $\{0,1,\flat\}^{\N_0}$ of the set of all sequences
$\sigma^n(p_{k_1}\flat\,p_{k_2}\flat\,p_{k_3}\flat\cdots)$ with
$n,k_1,k_2,\ldots\ge0$.  It is readily checked that $X$ is transitive, and
the symbol $\flat$ is a synchronizing word for $X$ \cite[\S 12]{HJXZ}, so
$X$ is synchronized.

To see that $X$ does not have specification, let $N\in\N$.  All words occurring in $X$ occur in the sequences
$\sigma^n(p_{k_1}\flat\,p_{k_2}\flat\cdots)$, so in such a word each maximal
block of symbols $0$ and $1$ that is immediately followed by $\flat$ is a
suffix of some $p_k$.  Hence if $u$ is a word occurring in $\mathbf m$,
and $uw\flat$ occurs in $X$ for some word $w$ of length at most $N$, then,
considering the first symbol $\flat$ after $u$, we see that $u$ occurs in
some $p_k$, ending within distance $N$ of the end of $p_k$.  Since
$p_{k+1}=p_k\overline{p_k}$, for all $\ell\in\N$ the words $p_k$ with
$2^k\ge\ell$ have at most two distinct suffixes of length $\ell$; so for all
$L\in\N$ at most $2(N+1)$ words $u$ of length $L$ arise in this way.  On the
other hand the orbit closure of $\mathbf m$ is infinite, so the number of
words of length $L$ occurring in $\mathbf m$ tends to infinity with $L$.
Hence for some $L$ there is a word $u$ of length $L$ occurring in $X$ such
that $uw\flat$ occurs in $X$ for no word $w$ of length at most $N$.  As $N$
was arbitrary, $X$ does not have the specification property.
\end{rem}

\section{Sturmian beta-transformations}\label{sec:interval}

In this section we prove Theorems \ref{thm:E} and \ref{thm:F}, and Corollary
\ref{cor:G}, with $T_\beta$, $U_\beta$, $C^{0,\gamma}(I)$, $S_\beta$ and
$h_\beta$ as in Section \ref{subsec:transformations}.  Theorem \ref{thm:E}
is deduced from Theorem \ref{thm:C} by means of an inequality (Lemma
\ref{lem:key}) comparing distances in the beta-shift with distances in $I$,
and the same inequality underlies Theorem \ref{thm:F}.

\subsection{Coding, and a key inequality}\label{subsec:coding}

In this subsection we assume that $\beta>1$ is not a simple
beta-number, so that $d^*_\beta(1)=d_\beta(1)$ has infinitely many non-zero
digits and $S_\beta=Y(d_\beta(1))$ (cf.~Definition \ref{def:graph}); we also
fix $\gamma\in(0,1]$, set $\theta:=\beta^{-\gamma}$, and equip $S_\beta$
with the metric $d_\theta$.
With $h_\beta$ as in Section
\ref{subsec:transformations}, we write 
\begin{equation}\label{eq:Hbeta_def}
H_\beta(\nu):=\nu\circ h_\beta^{-1}
\end{equation}
for $\nu\in\m(S_\beta,\sigma)$.  Up to
a shift of indices (reflecting the fact that sequences in \cite{HHJL} are indexed from $1$ rather than $0$),
$S_\beta$ and $h_\beta$ are the beta-shift $\mathcal S_\beta=X_\beta$ and the
map $h_\beta$ of \cite[Def.~3.7, Prop.~3.14, Rem.~3.15(ii)]{HHJL}.  We use
the following facts from \cite{HHJL}:
\begin{enumerate}
\item[(C1)] $h_\beta(S_\beta)=I$ \cite[Prop.~3.9(x)]{HHJL};
\item[(C2)] $H_\beta$ is a bijection from $\m(S_\beta,\sigma)$ onto
$\m(I,U_\beta)=\m(I,T_\beta)$, which maps periodic orbit measures onto
periodic orbit measures \cite[Props.~3.12(iv), 3.27(i),(iv)]{HHJL} (for the
last statement note that $H_\beta^{-1}$ is induced by the injective map
$\pi^*_\beta$, and $\sigma\circ\pi^*_\beta=\pi^*_\beta\circ U_\beta$
\cite[Prop.~3.9(iii),(vi)]{HHJL}); consequently \cite[Prop.~3.27(iii)]{HHJL}
\[
  \Mmax(U_\beta,g)=\Mmax(T_\beta,g)=H_\beta\bigl(\Mmax(g\circ h_\beta)\bigr)
  \qquad(g\in C(I)),
\]
where $\Mmax(g\circ h_\beta)$ is the set of maximizing measures for
$g\circ h_\beta$ with respect to $\sigma:S_\beta\to S_\beta$;
\item[(C3)] $U_\beta^n(1)=h_\beta(\sigma^nd_\beta(1))$ for all $n\ge0$
\cite[Prop.~3.9(i),(iii),(iv)]{HHJL}.
\end{enumerate}
If $x,y\in S_\beta$ satisfy $x_{[0,\ell)}=y_{[0,\ell)}$, then
$h_\beta(x)-h_\beta(y)=\beta^{-\ell}\bigl(h_\beta(\sigma^\ell x)-h_\beta(\sigma^\ell y)\bigr)$,
where $\sigma^\ell x,\sigma^\ell y\in S_\beta$, so (C1) gives
$|h_\beta(x)-h_\beta(y)|\le\beta^{-\ell}$,
and therefore
\begin{equation}\label{eq:pullback-Lip}
  |g\circ h_\beta|_{\lip(S_\beta)}\le|g|_\gamma
  \qquad\text{for all }g\in C^{0,\gamma}(I) ;
\end{equation}
if moreover $x_\ell>y_\ell$, then (C1) gives
\begin{equation}\label{eq:first-difference}
  h_\beta(x)-h_\beta(y)
  =\beta^{-\ell-1}\bigl(x_\ell-y_\ell+h_\beta(\sigma^{\ell+1}x)-h_\beta(\sigma^{\ell+1}y)\bigr)
  \ge\beta^{-\ell-1}h_\beta(\sigma^{\ell+1}x) .
\end{equation}

The distance from $x\in S_\beta$ to a closed invariant set $K$ cannot in
general be bounded in terms of the distance from $h_\beta(x)$ to
$h_\beta(K)$: in the Sturmian case, $10^\infty\notin X_\alpha$, whereas
$h_\beta(10^\infty)=h_\beta(0\bfa)\in h_\beta(X_\alpha)$.
However, $\sigma(10^\infty)=0^\infty$, and $h_\beta(0^\infty)=0$ is far from
$h_\beta(K)$ when no element of $K$ contains a long run of zeros; the
following lemma shows that such a bound does hold on average with respect to
invariant measures.

\begin{lem}\label{lem:key}
Suppose $\beta>1$ is not a simple beta-number, let $\gamma\in(0,1]$ and
$\theta:=\beta^{-\gamma}$, and let $H_\beta$ be as in \eqref{eq:Hbeta_def}.
Let $K\subseteq S_\beta$ be a non-empty closed set with $\sigma(K)\subseteq
K$, let $N\in\N$ be such that the word $0^N$ does not occur in any element of $K$,
and set $\Lambda:=h_\beta(K)$, $\Psi_K:=d_\theta(\cdot,K)$,
$\rho_\Lambda:=\operatorname{dist}(\cdot,\Lambda)^\gamma$,
$A_K:=(\beta^{N+1}/(\beta-1))^\gamma$ and
$C_K:=\beta^{\gamma(N+2)}+A_K/(1-\theta)$.  Then
$\Lambda\subseteq[\beta^{-N},1]$,
$d_\theta(x,y)\le\beta^{\gamma(N+1)}|h_\beta(x)-h_\beta(y)|^\gamma$ for all
$x,y\in K$, and
\begin{equation}\label{eq:key}
  \int\Psi_K\,d\nu\le C_K\int\rho_\Lambda\circ h_\beta\,d\nu
  \qquad\text{for all }\nu\in\m(S_\beta,\sigma) .
\end{equation}
In particular, if $\nu\in\m(S_\beta,\sigma)$ and $H_\beta(\nu)(\Lambda)=1$,
then $\nu(K)=1$.
\end{lem}

\begin{proof}
Each $z\in K$ has a non-zero digit among $z_0,\ldots,z_{N-1}$, so
$h_\beta(z)\ge\beta^{-N}$, and $h_\beta(z)\le1$ by (C1).  If $x,y\in K$ first
differ at $\ell$, say with $x_\ell>y_\ell$, then $\sigma^{\ell+1}x\in K$, and
\eqref{eq:first-difference} gives
$|h_\beta(x)-h_\beta(y)|\ge\beta^{-\ell-1-N}=\beta^{-N-1}d_\theta(x,y)^{1/\gamma}$.

Now let $x\in S_\beta\setminus K$, choose $z\in K$ with
$|h_\beta(x)-h_\beta(z)|=\operatorname{dist}(h_\beta(x),\Lambda)=:\delta$, and
let $\ell$ be the least index with $x_\ell\ne z_\ell$, so that
$\Psi_K(x)\le\theta^\ell$.  If $z_\ell>x_\ell$, then \eqref{eq:first-difference}
gives $\delta\ge\beta^{-\ell-1}h_\beta(\sigma^{\ell+1}z)\ge\beta^{-\ell-1-N}$.
If $x_\ell>z_\ell$ and $x_{[\ell+1,\ell+N+2)}$ contains a non-zero digit, then
\eqref{eq:first-difference} gives
$\delta\ge\beta^{-\ell-1}h_\beta(\sigma^{\ell+1}x)\ge\beta^{-\ell-N-2}$.  In
both cases $\theta^\ell=\beta^{-\gamma\ell}\le\beta^{\gamma(N+2)}\delta^\gamma
=\beta^{\gamma(N+2)}\rho_\Lambda(h_\beta(x))$.  In the remaining case
$x_\ell>z_\ell$ and $x_{[\ell+1,\ell+N+2)}=0^{N+1}$, so
$h_\beta(\sigma^{\ell+1}x)=\beta^{-N-1}h_\beta(\sigma^{\ell+N+2}x)\le\beta^{-N-1}$,
and as $\Lambda\subseteq[\beta^{-N},1]$ we get
\[
  \rho_\Lambda\bigl(h_\beta(\sigma^{\ell+1}x)\bigr)\ge(\beta^{-N}-\beta^{-N-1})^\gamma
  =A_K^{-1},
\]
so that
$\theta^\ell\le A_K\,\theta^{\ell}\rho_\Lambda(h_\beta(\sigma^{\ell+1}x))$.  Hence in all cases
\[
  \Psi_K(x)\le\beta^{\gamma(N+2)}\rho_\Lambda\bigl(h_\beta(x)\bigr)
  +A_K\sum_{k\ge1}\theta^{k-1}\rho_\Lambda\bigl(h_\beta(\sigma^kx)\bigr),
\]
and this holds trivially for $x\in K$.  Integrating with respect to $\nu$,
and using the $\sigma$-invariance of $\nu$, gives \eqref{eq:key}.  Finally,
if $H_\beta(\nu)(\Lambda)=1$ then $\int\rho_\Lambda\circ
h_\beta\,d\nu=\int\rho_\Lambda\,dH_\beta(\nu)=0$, so $\int\Psi_K\,d\nu=0$, and
$\nu(K)=1$ as $K$ is closed.
\end{proof}

\subsection{Proof of Theorem \ref{thm:E}}

We are now able to prove Theorem \ref{thm:E}, by combining Theorem
\ref{thm:C} with Lemma \ref{lem:key}.

\begin{proof}[Proof of Theorem \ref{thm:E}]
As $d_\beta(1)=\bfa$ has infinitely many digits $1$, $\beta$ is not a
simple beta-number, and $S_\beta=Y_\alpha$.  Let $N$
be as in Lemma \ref{lem:spec}.  As all words occurring in $X_\alpha$ occur
in $\bfa$, Lemma \ref{lem:key} applies to $K:=X_\alpha$, for which
$\Psi_K=\Psi$; and by (C3), $h_\beta(X_\alpha)$ is the closure of
$\{h_\beta(\sigma^n\bfa):n\ge0\}=\{U_\beta^n(1):n\ge0\}$, that is,
$\Lambda=\Lambda_\alpha$.  Let $\widehat\mu_\alpha:=H_\beta(\mu_\alpha)$, so that
$\widehat\mu_\alpha(\Lambda_\alpha)\ge\mu_\alpha(X_\alpha)=1$, and let
$\rho:=\rho_{\Lambda_\alpha}$.  If $\nu\in\m(I,U_\beta)$ and
$\nu(\Lambda_\alpha)=1$ then, by (C2) and Lemma \ref{lem:key},
$\nu=H_\beta(\nu')$ with $\nu'(X_\alpha)=1$, so $\nu'=\mu_\alpha$ by unique
ergodicity, and $\nu=\widehat\mu_\alpha$.  Moreover $\widehat\mu_\alpha$ has
no atoms, since $h_\beta$ is injective on $X_\alpha$ (Lemma \ref{lem:key}) and
$\mu_\alpha$ has no atoms 
(an atom of an invariant probability measure being a periodic point, of
which the infinite minimal set $X_\alpha$ has none); so $\widehat\mu_\alpha$ is not a periodic orbit measure.  Also
$\widehat F\in C^{0,\gamma}(I)$, as $|\mathrm{id}_I|_\gamma\le1$ and
$|\rho|_\gamma\le1$.

Let $\nu\in\m(Y_\alpha,\sigma)$.  Since $\beta
h_\beta(x)=x_0+h_\beta(\sigma x)$, we have $\Phi=(\beta-1)h_\beta+h_\beta-
h_\beta\circ\sigma$, so $\int\Phi\,d\nu=(\beta-1)\int h_\beta\,d\nu$.  As
$\rho\circ h_\beta$ and $\Psi$ vanish on $X_\alpha$, and
$\int\Phi\,d\mu_\alpha=\alpha$, it follows that
\[
  \widehat D(\nu):=\int\widehat F\,d\widehat\mu_\alpha-\int\widehat
  F\,dH_\beta(\nu)=\Bigl(\alpha-\int\Phi\,d\nu\Bigr)+\int\rho\circ h_\beta\,d\nu ,
\]
while $\int F\,d\mu_\alpha-\int F\,d\nu=(\alpha-\int\Phi\,d\nu)+\int\Psi\,d\nu$.
Since $\alpha-\int\Phi\,d\nu\ge0$ (Proposition \ref{prop:Phimax}) and $C_K\ge1$,
Lemma \ref{lem:key} and Theorem \ref{thm:C} give
\[
  \widehat D(\nu)\ge C_K^{-1}\Bigl(\int F\,d\mu_\alpha-\int F\,d\nu\Bigr)
  \ge(C_KC_\theta)^{-1}\,W_\theta(\nu,\mu_\alpha) .
\]
Now set $\delta:=(C_KC_\theta)^{-1}$, let $\psi\in
C^{0,\gamma}(I)$ satisfy $|\psi-\widehat F|_\gamma<\delta$, and let
$\nu\ne\mu_\alpha$.  By \eqref{eq:pullback-Lip},
$g:=(\psi-\widehat F)\circ h_\beta$ satisfies
$|g|_{\lip(Y_\alpha)}<(C_KC_\theta)^{-1}$, so by \eqref{eq:wass},
\begin{align*}
  \int\psi\,d\widehat\mu_\alpha-\int\psi\,dH_\beta(\nu)
  &=\widehat D(\nu)+\int g\,d\mu_\alpha-\int g\,d\nu\\
  &\ge\bigl((C_KC_\theta)^{-1}-|g|_{\lip(Y_\alpha)}\bigr)W_\theta(\nu,\mu_\alpha)>0 .
\end{align*}
By (C2) it follows that
$\Mmax(U_\beta,\psi)=\Mmax(T_\beta,\psi)=\{\widehat\mu_\alpha\}$.  As
$\widehat\mu_\alpha$ is not a periodic orbit measure, neither
$(U_\beta,\psi)$ nor $(T_\beta,\psi)$ has the periodic optimization property,
for any $\psi$ in the open set $\{\psi:|\psi-\widehat F|_\gamma<\delta\}$, so
neither $U_\beta$ nor $T_\beta$ has TPO with respect to $C^{0,\gamma}(I)$.
\end{proof}

We note the relation between Theorem \ref{thm:E} and the critical set
$\mathrm{Crit}^\gamma(\beta)$ of \cite{HHJL}:

\begin{rem}\label{rem:E-critical}
For $\psi$ as in Theorem \ref{thm:E}, (C3) and unique ergodicity give
$\frac1n\sum_{i<n}\psi(U_\beta^i1)=\frac1n\sum_{i<n}(\psi\circ
h_\beta)(\sigma^i\bfa)\to\int\psi\,d\widehat\mu_\alpha$, so the $U_\beta$-orbit
of $1$ is maximizing.  Thus the open set of Theorem \ref{thm:E} lies in
$\mathrm{Crit}^\gamma(\beta)$, which can therefore have non-empty interior.
\end{rem}

We conclude this subsection by indicating how Theorems \ref{thm:A},
\ref{thm:C}, \ref{thm:D} and \ref{thm:E} extend to Sturmian sequences over
other alphabets.

\begin{rem}[Sturmian sequences over other alphabets]\label{rem:alphabets}
Let $\ell<m$ be non-negative integers, and let $\iota:\Sigma\to\Sigma_m$ be
the map which replaces each digit $0$ by $\ell$ and each digit $1$ by $m$;
we call the elements of the subshifts $\iota(X_\alpha)$, with
$\alpha\in(0,1)$ irrational, \emph{Sturmian sequences over the alphabet
$\{\ell,m\}$}.  Theorems \ref{thm:A}, \ref{thm:C}, \ref{thm:D} and
\ref{thm:E} extend to these sequences, with modifications which we now
indicate.  As $\iota$ is strictly increasing with respect to $\le_\lex$,
satisfies $\iota\circ\sigma=\sigma\circ\iota$, and is an isometry for all
metrics $d_\theta$, the sequence $\iota(\bfa)$ is the lexicographically
largest element of the minimal subshift $\iota(X_\alpha)$, and
$\sigma^j\iota(\bfa)<_\lex\iota(\bfa)$ for all $j\ge1$.  So there is a unique
$\beta>1$ with $d_\beta(1)=\iota(\bfa)$ (by \cite{Par60}), which lies in
$(m,m+1)$, and its beta-shift is $Y(\iota(\bfa))$, in the notation of
Definition \ref{def:graph}; this beta-shift has the specification property,
by the proof of Lemma \ref{lem:spec}.  Define $\Phi_{\ell,m}:\Sigma_m\to\R$
by $\Phi_{\ell,m}(x):=1$ if $x_0=m$, $\Phi_{\ell,m}(x):=0$ if $\ell\le
x_0<m$, and $\Phi_{\ell,m}(x):=-1$ if $x_0<\ell$, so that
$\Phi_{\ell,m}\circ\iota=\Phi$, and set
$F_{\ell,m}:=\Phi_{\ell,m}-d_\theta(\cdot,\iota(X_\alpha))$.  Then the
conclusions of Theorems \ref{thm:A} and \ref{thm:C} hold for
$Y(\iota(\bfa))$, $F_{\ell,m}$ and the measure $\mu_\alpha\circ\iota^{-1}$,
provided $C_\theta$ is replaced by $1+3(1-\theta)^{-3}$ when $\ell\ge1$.
Indeed, Theorem \ref{thm:D} carries over to $\iota(X_\alpha)$ via the
isometric conjugacy $\iota$.  The first return loops of the graph
$G(\iota(\bfa))$ are the words $\iota(\bfa)_{[0,n-1)}c$ with
$c<\iota(\bfa)_{n-1}$, and the sum of the values of $\Phi_{\ell,m}$ along
such a loop is at most $\lfloor n\alpha\rfloor$, as in Lemma
\ref{lem:returnwords}: if $a_{n-1}=0$ then $\iota(\bfa)_{[0,n-1)}$ contains
$\lfloor n\alpha\rfloor+1$ digits $m$, while $c<\ell$, so that the final
digit contributes $-1$.  It follows that $S_n\Phi_{\ell,m}\le n\alpha+1$ on
$Y(\iota(\bfa))$, as in Proposition \ref{prop:Phimax}, 
so Lemma
\ref{lem:rigid-stability} applies as in the proof of Theorem \ref{thm:C},
with $Y(\iota(\bfa))$, $\iota(X_\alpha)$, $\mu_\alpha\circ\iota^{-1}$ and
$\Phi_{\ell,m}$ in place of $Y_\alpha$, $X_\alpha$, $\mu_\alpha$ and $\Phi$;
if $\ell\ge1$ then $|\Phi_{\ell,m}|_\lip=2$, and the constant $C$ of the
lemma is $1+(3-\theta)(1-\theta)^{-3}\le1+3(1-\theta)^{-3}$.
The value $-1$ cannot be replaced
by $0$ when $\ell\ge1$: if $\bfa$ begins with $1^k0$ then $\alpha<k/(k+1)$,
whereas the periodic point $(m^k0)^\infty$ belongs to $Y(\iota(\bfa))$ and
the digit $m$
has frequency $k/(k+1)$ in it.  Similarly, Theorem \ref{thm:E}
holds with $\widehat F$ replaced by
$t\mapsto(\beta-1)t/(m-\ell)-M\operatorname{dist}(t,\Lambda)^\gamma$, where
$\Lambda$ is the closure of the $U_\beta$-orbit of $1$, $M:=1$ if
$\ell=m-1$, and $M:=2C_K$ otherwise, with $C_K$ as in Lemma \ref{lem:key}
for $K:=\iota(X_\alpha)$; here the proof uses the inequality
$(x_0-\ell)/(m-\ell)\le\Phi_{\ell,m}(x)+d_\theta(x,\iota(X_\alpha))$, valid
for all $x\in\Sigma_m$, in which the second term can be omitted when
$\ell=m-1$.  Finally, if $\ell=m-1$ then $\iota$ is the tuning map
$\zeta_m$ of Definition \ref{def:tuning}, so $Y(\iota(\bfa))$ is the tuned
Sturmian beta-shift $Y_{m,\alpha}$ of Definition \ref{def:tuned}, and the
failure of TPO for it is also a special case of Theorem \ref{thm:tuned}.
\end{rem}

\subsection{Proofs of Theorem \ref{thm:F} and Corollary
\ref{cor:G}}\label{subsec:transfer}

We now prove Theorem \ref{thm:F}, and deduce Corollary \ref{cor:G}.

Until the end of the proof of Theorem \ref{thm:F}, let $\beta$, $\gamma$ and
$\theta$ be as in Section~\ref{subsec:coding}: thus $\beta>1$ is not a
simple beta-number, $\gamma\in(0,1]$, $\theta=\beta^{-\gamma}$, and
$S_\beta$ is equipped with the metric $d_\theta$.  By \eqref{eq:pullback-Lip}
and (C1), $g\mapsto g\circ h_\beta$ is an injective bounded linear map from
$C^{0,\gamma}(I)$ to $\lip(S_\beta,d_\theta)$, and by (C2) the maximizing
measures of $g$ for $U_\beta$ are the images under $h_\beta$ of those of
$g\circ h_\beta$.  The image of this map lies in the closed subspace of
functions $f$ with $f(c0^\infty)=f((c-1)d_\beta(1))$ for all integers
$1\le c<\beta$, as
$h_\beta(c0^\infty)=c/\beta=h_\beta((c-1)d_\beta(1))$, and this subspace has
empty interior.  (In the Sturmian case, $0\bfa\in X_\alpha$, as
$\sigma(X_\alpha)=X_\alpha$ and $1\bfa>_\lex\bfa=\max_\lex X_\alpha$, so
$F(0\bfa)=0$, whereas $F(10^\infty)\ge1-\theta$; so all such $f$ satisfy
$|f-F|_{\lip(Y_\alpha)}\ge1-\theta>C_\theta^{-1}$, and none lies in the open
set $U$ of Theorem \ref{thm:A}.)  For this reason the proof of Theorem
\ref{thm:F} relies on the extension estimate \eqref{eq:extension} below.

\begin{proof}[Proof of Theorem \ref{thm:F}]
We write $\|\cdot\|_\lip$ and $|\cdot|_\lip$ for the norm and seminorm of
$\lip(S_\beta,d_\theta)$, use the notation of Lemma \ref{lem:key}, and for
$f\in C(S_\beta)$ let $Q(f)$, respectively $Q_K(f)$, be the maximum of $\int
f\,d\mu$ over all $\mu\in\m(S_\beta,\sigma)$, respectively over those with
$\mu(K)=1$.  By (C2) and Lemma \ref{lem:key}, $H_\beta(\mu)(\Lambda)=1$ if
and only if $\mu(K)=1$; so (a) says that $Q(f)=Q_K(f)$ for $f\in\cV$, and (b)
that $Q(g\circ h_\beta)=Q_K(g\circ h_\beta)$ for $g\in\cV'$.

For $u\in\lip(S_\beta,d_\theta)$, Lemma \ref{lem:key} shows that
$u\circ(h_\beta|_K)^{-1}$ is $\gamma$-H\"older on $\Lambda$ with constant
$\beta^{\gamma(N+1)}|u|_\lip$, so McShane's extension theorem \cite{McShane}
for the metric $|s-t|^\gamma$, followed by truncation at $\pm\|u\|_{C(S_\beta)}$,
gives $\bar u\in C^{0,\gamma}(I)$ with $\bar u\circ h_\beta=u$ on $K$,
$|\bar u|_\gamma\le\beta^{\gamma(N+1)}|u|_\lip$ and
$\|\bar u\|_{C(I)}\le\|u\|_{C(S_\beta)}$, so that
$\|\bar u\|_\gamma\le\beta^{\gamma(N+1)}\|u\|_\lip$.  Comparing $x\in S_\beta$
with some $z\in K$ satisfying $d_\theta(x,z)=\Psi_K(x)$, and using
\eqref{eq:pullback-Lip}, we obtain
\begin{equation}\label{eq:extension}
  |u-\bar u\circ h_\beta|\le\bigl(|u|_\lip+|\bar u|_\gamma\bigr)
  \Psi_K\le\bigl(1+\beta^{\gamma(N+1)}\bigr)|u|_\lip\Psi_K\qquad\text{on }S_\beta .
\end{equation}

(a)$\Rightarrow$(b).  Let $f\in\cV$ and $\varepsilon>0$ be such that $\cV$
contains all $g$ with $\|g-f\|_\lip<\varepsilon$.  With $\bar f$ as above,
let $\varphi:=\bar f-M\rho_\Lambda$, where
$M:=(1+\beta^{\gamma(N+1)})|f|_\lip C_K+1$, and
let $\cV'$ be the set of $\psi\in C^{0,\gamma}(I)$ with
$\|\psi-\varphi\|_\gamma<\varepsilon$.  Let $\psi\in\cV'$.  By
\eqref{eq:pullback-Lip}, $g:=f+(\psi-\varphi)\circ h_\beta$ satisfies
$\|g-f\|_\lip<\varepsilon$, so $Q(g)=Q_K(g)$; and $\psi\circ h_\beta=g$ on
$K$, so $Q_K(\psi\circ h_\beta)=Q(g)$.  On $S_\beta$, \eqref{eq:extension}
gives $\psi\circ h_\beta=g+(\bar f\circ h_\beta-f)-M\rho_\Lambda\circ
h_\beta\le g+(1+\beta^{\gamma(N+1)})|f|_\lip\Psi_K-M\rho_\Lambda\circ h_\beta$, so by
\eqref{eq:key} all $\nu\in\m(S_\beta,\sigma)$ satisfy
$\int\psi\circ h_\beta\,d\nu\le Q_K(\psi\circ h_\beta)-\int\rho_\Lambda\,dH_\beta(\nu)$.
Hence, for all $\psi\in\cV'$, all members of $\Mmax(U_\beta,\psi)$ give full
measure to $\Lambda$; in particular (b) holds.

(b)$\Rightarrow$(a).  Let $\varphi\in\cV'$ and $r>0$ be such that $\cV'$
contains all $\psi$ with $\|\psi-\varphi\|_\gamma<r$.  Let
$f:=\varphi\circ h_\beta-\Psi_K\in\lip(S_\beta,d_\theta)$, let $\cV$ be the set
of $g\in\lip(S_\beta,d_\theta)$ with $\|g-f\|_\lip<\varepsilon:=\min\{\beta^{-\gamma(N+1)}r,
(2+2\beta^{\gamma(N+1)})^{-1}\}$, and let $g\in\cV$.  With $u:=g-f$ and $\bar u$ as above,
$\psi:=\varphi+\bar u$ satisfies $\|\psi-\varphi\|_\gamma<r$, so $Q(\psi\circ
h_\beta)=Q_K(\psi\circ h_\beta)$; and $g=\psi\circ h_\beta$ on $K$, so
$Q_K(g)=Q(\psi\circ h_\beta)$.  
On $S_\beta$, \eqref{eq:extension} gives
$g=\psi\circ h_\beta-\Psi_K+(u-\bar u\circ h_\beta)\le\psi\circ
h_\beta-\frac12\Psi_K$, so all $\nu\in\m(S_\beta,\sigma)$ satisfy $\int
g\,d\nu\le Q_K(g)-\frac12\int\Psi_K\,d\nu$.  Hence, for all $g\in\cV$, all
maximizing measures of $g$ give full measure to $K$; in particular (a)
holds.

The two implications show that if (a) or (b) holds, then $\cV$ and $\cV'$
can be chosen such that all maximizing measures of all $f\in\cV$ give full
measure to $K$, and all members of $\Mmax(U_\beta,g)=\Mmax(T_\beta,g)$,
$g\in\cV'$, give full measure to $\Lambda$, as claimed.  Finally, if $K$
contains no periodic orbit then no periodic orbit measure gives full measure
to $K$, nor, by (C2) and Lemma \ref{lem:key}, to $\Lambda$; so, with $\cV$
and $\cV'$ chosen in this way, no $f\in\cV$ has the periodic optimization
property, and neither does $(U_\beta,g)$, nor $(T_\beta,g)$, for any
$g\in\cV'$.
\end{proof}

Corollary \ref{cor:G} is deduced from Theorem \ref{thm:F}, the structural
results of \cite{HHJL} and \cite{HJXZ}, and Theorem \ref{thm:nonemergent},
whose proof in Section \ref{subsec:nonemergent} does not use the results of
this section.
  Theorem \ref{thm:nonemergent} is needed for non-emergent $\beta$,
as Theorem \ref{thm:F} cannot be applied to all of them with
$K=\partial_MS_\beta$: for Lebesgue almost every $\beta$ the orbit of
$d_\beta(1)$ is dense in $S_\beta$ \cite[Thm.~E]{Sch97}, so that
$\partial_MS_\beta=S_\beta$, which contains the fixed point $0^\infty$,
whereas Theorem \ref{thm:F} requires that some word $0^N$ does not occur in any
element of $K$.

\begin{proof}[Proof of Corollary \ref{cor:G}]
Set $\theta:=\beta^{-\gamma}$.  We first show that $S_\beta$ has TPO with
respect to $\lip(S_\beta,d_\theta)$ if and only if $U_\beta$ has TPO with
respect to $C^{0,\gamma}(I)$.  If $\beta$ is non-emergent then both have
TPO, by Theorem \ref{thm:nonemergent} and \cite[Thm.~J]{HHJL}.  If $\beta$
is a simple beta-number then $S_\beta$ is a shift of finite type, as
$d^*_\beta(1)$ is periodic (see e.g.~\cite{blanchard}), so has TPO
\cite[Lem.~9.4]{HJXZ}, and $U_\beta$ has TPO \cite[Thm.~I]{HHJL}.  As
beta-numbers which are not simple are non-emergent \cite[Cor.~5.4]{HHJL},
we may therefore assume that $\beta$ is emergent and not a beta-number.

Let $\omega_\beta$ denote the $\omega$-limit set of $d_\beta(1)$, which is
the Markov boundary of $S_\beta$ by Lemma \ref{lem:markov-boundary}.  As
$d_\beta(1)$ is the upper $\beta$-expansion of $1$
\cite[Prop.~3.9(i)]{HHJL}, its orbit closure is minimal
\cite[Prop.~5.3]{HHJL}, so is equal to $\omega_\beta$, and, being infinite,
contains no periodic orbit.  In particular $0^\infty\notin\omega_\beta$, so
by compactness some word $0^N$ occurs in no element of $\omega_\beta$, and
$K:=\omega_\beta$ satisfies the hypotheses of Theorem \ref{thm:F}.

Suppose that $U_\beta$ does not have TPO.  The set
$\mathrm{Lock}^\gamma(U_\beta)$ of functions with the locking property
\cite[\S 1.5]{HHJL} is open and consists of functions with the periodic
optimization property, so is not dense, and by \cite[Thm.~5.22]{HHJL} the
non-empty open set $\cV'$ of functions not in its closure lies in
$\mathrm{Crit}^\gamma(\beta)$.  By \cite[Lem.~5.10]{HHJL}, each $g\in\cV'$
has a maximizing measure $H_\beta(\mu)$ with $\mu$ quasi-generated by the
orbit of $d_\beta(1)$, that is, $\mu$ is a weak$^*$ limit point of the
measures $n^{-1}\sum_{i<n}\delta_{\sigma^id_\beta(1)}$, $n\in\N$
\cite[Def.~5.9]{HHJL}; so $\mu(\omega_\beta)=1$ and
$H_\beta(\mu)(h_\beta(\omega_\beta))=1$.  So Theorem \ref{thm:F}(b) holds
with $K=\omega_\beta$, and as $\omega_\beta$ contains no periodic orbit,
$S_\beta$ does not have TPO.

Conversely, suppose that $S_\beta$ does not have TPO.  Then, by
\cite[Cor.~10.1]{HJXZ}, the set of $f\in\lip(S_\beta,d_\theta)$ having a
maximizing measure $\mu$ with $\mu(\omega_\beta)=1$ has non-empty interior,
that is, Theorem \ref{thm:F}(a) holds with $K=\omega_\beta$; so $U_\beta$
does not have TPO with respect to $C^{0,\gamma}(I)$.

It remains to consider $T_\beta$.  If $\beta$ is not a simple beta-number,
then $\m(I,T_\beta)=\m(I,U_\beta)$ and $\mathrm{Per}(T_\beta)=
\mathrm{Per}(U_\beta)$ \cite[Prop.~3.12(iv)]{HHJL}, and
$\Mmax(T_\beta,g)=\Mmax(U_\beta,g)$ for all $g\in C(I)$, by (C2); so
$(T_\beta,g)$ has the periodic optimization property if and only if
$(U_\beta,g)$ does, and $T_\beta$ has TPO if and only if $U_\beta$ does.  If
$\beta$ is a simple beta-number, then $T_\beta$ does not have TPO with
respect to $C^{0,\gamma}(I)$, as some non-empty open subset of
$C^{0,\gamma}(I)$ consists of functions having no maximizing measure for
$T_\beta$ \cite[Rem.~5.21]{HHJL}.
\end{proof}

We conclude this section by comparing the obstructions to TPO for
$S_\beta$ and for $U_\beta$.

\begin{rem}[Obstructions to TPO for emergent parameters]
\label{rem:G-converse}
For emergent $\beta$ which are not beta-numbers, Theorem \ref{thm:F} with
$K=\partial_MS_\beta$, the orbit closure of $d_\beta(1)$, relates the
obstructions to TPO in the two settings.  Condition (a) says that $S_\beta$
is not fragile in the sense of \cite[Def.~10.2]{HJXZ}, and condition (b) says
that $\mathrm{Crit}^\gamma(\beta)$ has non-empty interior, since by
\cite[Lems.~3.13, 5.6, 5.10, 5.11]{HHJL} a function $g\in C^{0,\gamma}(I)$
belongs to $\mathrm{Crit}^\gamma(\beta)$ if and only if some
$\nu\in\Mmax(U_\beta,g)$ satisfies $\nu(h_\beta(\partial_MS_\beta))=1$.  So
$S_\beta$ is fragile if and only if $\mathrm{Crit}^\gamma(\beta)$ has
empty interior, and as $\partial_MS_\beta$ contains no periodic orbit, each
of these two conditions is equivalent to TPO in the corresponding setting
\cite[Cor.~10.1]{HJXZ}, \cite[Thm.~5.22]{HHJL}.  In the Sturmian case,
Theorem \ref{thm:A} shows directly that $Y_\alpha$ is not fragile (Remark
\ref{rem:markov-boundary}), and Theorems \ref{thm:A} and \ref{thm:F},
together with the unique ergodicity of $X_\alpha$, give Theorem
\ref{thm:E}, apart from the explicit form of the open set.
\end{rem}

\section{Robustly maximizing measures, and parameters without TPO}\label{sec:dimension}

In this section we prove Theorems \ref{thm:nonemergent}, \ref{thm:robust},
\ref{thm:dimension} and \ref{thm:TZO}, in this order.  Whenever
$\beta>1$ and $\theta\in(0,1)$ are given, the beta-shift $S_\beta$ is
equipped with the metric $d_\theta$, and we write $\|\cdot\|_\lip$ and
$|\cdot|_\lip$ for the norm and seminorm of $\lip(S_\beta,d_\theta)$.

\subsection{Sub-actions on beta-shifts}\label{subsec:subactions}

The proofs of Theorems \ref{thm:nonemergent} and \ref{thm:robust} are
based on the following lemma, which constructs a bounded sub-action.

\begin{lem}\label{lem:subaction}
Let $\beta>1$ and $\theta\in(0,1)$, let $g\in\lip(S_\beta,d_\theta)$
and $\mu\in\Mmax(g)$, let $q:=g-\int g\,d\mu$, and define
$u:S_\beta\to[0,\infty]$ by
\begin{equation}\label{eq:subaction}
  u(x):=\sup\bigl\{S_jq(z):j\ge0,\ z\in S_\beta,\ \sigma^jz=x\bigr\} .
\end{equation}
Then $u$ is a Borel function, and:
\begin{enumerate}
\item[(i)] $0\le u\le\bigl(1+2\theta(1-\theta)^{-1}\bigr)|g|_\lip$;
\item[(ii)] $q\le u\circ\sigma-u$, with equality $\mu$-almost everywhere;
\item[(iii)] $u(y)-u(x)\le\theta(1-\theta)^{-1}|g|_\lip\,d_\theta(x,y)$ for
all $x,y\in S_\beta$ with $x\le_\lex y$.
\end{enumerate}
\end{lem}

\begin{proof}
We first show that $u$ is a Borel function satisfying
\eqref{eq:subaction-ineq} below, and then prove (iii), (i) and (ii), in this
order: the proof of (i) uses (iii), and the equality in (ii) uses (i).

For $x\in S_\beta$ and $j\ge0$ let $\sigma^{-j}(x):=\{z\in S_\beta:\sigma^jz=x\}$.
As the first symbol of $d^*_\beta(1)$ is non-zero, the point $0^jx$ belongs to
$S_\beta$, and hence to $\sigma^{-j}(x)$; and $\sigma^{-j}(x)$ is compact, so
$u_j(x):=\max\{S_jq(z):z\in\sigma^{-j}(x)\}$ is well defined, and
$u=\sup_{j\ge0}u_j$.  If $x_r\to x$ in $S_\beta$, and $z_r\in\sigma^{-j}(x_r)$
satisfies $S_jq(z_r)=u_j(x_r)$ for all $r$, then all limit points of the
sequence $(z_r)$ belong to $\sigma^{-j}(x)$, so $\limsup_ru_j(x_r)\le u_j(x)$,
by continuity of $S_jq$.  Thus $u_j$ is upper semicontinuous, and $u$ is a
Borel function.  If $z\in\sigma^{-j}(x)$ then $z\in\sigma^{-j-1}(\sigma x)$,
so $u(\sigma x)\ge S_{j+1}q(z)=S_jq(z)+q(x)$, and taking the supremum over
$j$ and $z$ gives
\begin{equation}\label{eq:subaction-ineq}
  u(x)+q(x)\le u(\sigma x)\qquad\text{for all }x\in S_\beta .
\end{equation}

(iii) Let $x,y\in S_\beta$ with $x\le_\lex y$, let $j\ge1$ and
$z\in\sigma^{-j}(y)$, and let $w:=z_{[0,j)}$, so that $z=wy$.  For
$0\le i<j$ we have $\sigma^i(wx)\le_\lex\sigma^i(wy)\le_\lex d^*_\beta(1)$,
while $\sigma^i(wx)=\sigma^{i-j}x$ for $i\ge j$, so $wx\in\sigma^{-j}(x)$.
Moreover $d_\theta(\sigma^i(wx),\sigma^i(wy))=\theta^{j-i}d_\theta(x,y)$ for
$0\le i<j$, so
\[
  S_jq(z)\le S_jq(wx)+|g|_\lip\,d_\theta(x,y)\sum_{i=0}^{j-1}\theta^{j-i}
  \le u(x)+\theta(1-\theta)^{-1}|g|_\lip\,d_\theta(x,y) .
\]
This inequality also holds for $j=0$, as $u(x)\ge0$, and taking the
supremum over $j$ and $z$ gives (iii).

(i) The lower bound holds since $S_0q=0$.  Let
$c:=\theta(1-\theta)^{-1}|g|_\lip$.  We first show that
$u(0^\infty)\le|g|_\lip+c$.  Let $j\ge1$ and $z\in\sigma^{-j}(0^\infty)$.  If
$z=0^\infty$ then $S_jq(z)=j\,q(0^\infty)\le0$, as $\delta_{0^\infty}$ is
invariant.  Otherwise let $k$ be the largest integer such that $z_{k-1}\ne0$,
so that $k\le j$ and $z=vb0^\infty$ for a word $v$ of length $k-1$ and a
digit $b\ge1$.  For all
$y\in S_\beta$ the point $v(b-1)y$ belongs to $S_\beta$: for $0\le i<k-1$ we
have $\sigma^i(v(b-1)y)<_\lex\sigma^i(vb0^\infty)\le_\lex d^*_\beta(1)$, while
$\sigma^{k-1}(v(b-1)y)=(b-1)y<_\lex d^*_\beta(1)$, as
$b0^\infty=\sigma^{k-1}z\le_\lex d^*_\beta(1)$, and
$\sigma^i(v(b-1)y)=\sigma^{i-k}y$ for $i\ge k$.  So $S_\beta$
contains the points $(v(b-1))^n0^\infty$, $n\ge1$, and hence, being closed, the
periodic point $p:=(v(b-1))^\infty$.  The measure
$k^{-1}\sum_{i=0}^{k-1}\delta_{\sigma^ip}$ belongs to $\m(S_\beta,\sigma)$,
and $\int q\,d\nu\le0$ for all $\nu\in\m(S_\beta,\sigma)$, so $S_kq(p)\le0$.
For $0\le i<k$ the points $\sigma^iz$ and $\sigma^ip$ have the same first
$k-1-i$ symbols, so, as $q(0^\infty)\le0$,
\[
  S_jq(z)=S_kq(z)+(j-k)\,q(0^\infty)
  \le S_kq(p)+|g|_\lip\sum_{i=0}^{k-1}\theta^{k-1-i}
  \le(1-\theta)^{-1}|g|_\lip=|g|_\lip+c .
\]
As $S_0q=0$, it follows that $u(0^\infty)\le|g|_\lip+c$.  Finally, all
$x\in S_\beta$ satisfy $0^\infty\le_\lex x$ and $d_\theta(0^\infty,x)\le1$,
so (iii) gives $u(x)\le u(0^\infty)+c\le|g|_\lip+2c$, which is the upper
bound in (i).

(ii) The inequality is \eqref{eq:subaction-ineq}.  So $u\circ\sigma-u-q$ is
a non-negative Borel function, which is bounded by (i), and whose integral
with respect to $\mu$ vanishes, as $\mu$ is invariant and $\int q\,d\mu=0$;
hence it vanishes $\mu$-almost everywhere.
\end{proof}

Lemma~\ref{lem:subaction} says that $u$ is a bounded, though not
necessarily continuous, \emph{sub-action} for $g$ (see e.g.~\cite{CLT01,Jen19}).
For beta-shifts with the specification property, a bound of the form given in
Lemma~\ref{lem:subaction}(i)
also follows from the proof of \cite[Prop.~10.9]{HJXZ}, but the above proof
of (i) requires no such hypothesis.  Part (i) is a symbolic version of
the upper bound obtained in the proof of \cite[Prop.~4.7(i)]{HHJL} for
beta-transformations,
while part (iii) is a symbolic version of the one-sided
inequality of \cite[Lem.~4.5]{HHJL}, and is based on the same property of
inverse branches (cf.~\cite[Rem.~4.11]{HHJL}): a word which can precede a
point of $S_\beta$ can also precede all lexicographically smaller points of
$S_\beta$.

\subsection{Proof of Theorem \ref{thm:nonemergent}}\label{subsec:nonemergent}

Using Lemma \ref{lem:subaction}, we are now able to prove Theorem
\ref{thm:nonemergent}.

\begin{proof}[Proof of Theorem \ref{thm:nonemergent}]
Let $\theta\in(0,1)$, and let $\mathcal Z$ be the set of beta-shifts
$S_{\beta'}$, where $\beta'\in(1,\beta)$ is a simple beta-number.  Each
$S_{\beta'}\in\mathcal Z$ is a subshift of $S_\beta$, 
as $d^*_{\beta'}(1)\le_\lex d^*_\beta(1)$ \cite[Prop.~3.9(xiii)]{HHJL}, and is a shift of finite type, as
$d^*_{\beta'}(1)$ is periodic (see e.g.~\cite{blanchard}), so has TPO
\cite[Lem.~9.4]{HJXZ}.  As $\mathcal Z$ is countable, it suffices, by
\cite[Thm.~6.13]{HJXZ}, to show that $\mathcal Z$ is a maximizable family in
the sense of \cite[Def.~5.5]{HJXZ}, that is, that each
$g\in\lip(S_\beta,d_\theta)$ has a maximizing measure supported by some
member of $\mathcal Z$.

Let $g\in\lip(S_\beta,d_\theta)$ and $\mu\in\Mmax(g)$, and let $q$ and $u$
be as in Lemma \ref{lem:subaction}.  If $n\ge1$ and $x\in S_\beta$ then
$x\in\sigma^{-n}(\sigma^nx)$, so $S_nq(x)\le u(\sigma^nx)$, and Lemma
\ref{lem:subaction}(i) shows that the Birkhoff sums $S_nq$, $n\ge1$, are
bounded above uniformly.  As $\int q\,d\nu\le0=\int q\,d\mu$ for all
$\nu\in\m(S_\beta,\sigma)$, a theorem of Morris \cite[Thm.~1]{Mor07} (see
also \cite[Rem.~4.6, Lem.~4.7]{HJXZ}) shows that $g$ has the
\emph{subordination property}: all $\nu\in\m(S_\beta,\sigma)$ whose support
is contained in that of $\mu$ belong to $\Mmax(g)$.  Let $K$ be a minimal
subshift contained in the support of $\mu$, and let
$\nu\in\m(S_\beta,\sigma)$ satisfy $\nu(K)=1$, so that $\nu\in\Mmax(g)$.

As $\beta$ is non-emergent, the orbit closure of $d^*_\beta(1)$ is not
minimal \cite[Prop.~5.3]{HHJL}, so $d^*_\beta(1)\notin K$, since otherwise
the minimality of $K$ would imply that this orbit closure is $K$.  As
$K\subseteq S_\beta$, it follows that $\mathbf k:=\max_\lex K$ satisfies
$\mathbf k<_\lex d^*_\beta(1)$.  
Since $d^*_{\beta'}(1)$ increases
to $d^*_\beta(1)$ as $\beta'$ increases to $\beta$
\cite[Prop.~3.9(xiii),(xiv)]{HHJL}, and simple beta-numbers are dense in
$(1,\infty)$ \cite{Par60}, there is a simple beta-number $\beta'<\beta$ with
$\mathbf k\le_\lex d^*_{\beta'}(1)$.
For all $x\in K$ and $j\ge0$ we have
$\sigma^jx\in K$, so $\sigma^jx\le_\lex\mathbf k\le_\lex d^*_{\beta'}(1)$;
hence $K\subseteq S_{\beta'}$, and the maximizing measure $\nu$ is supported
by $S_{\beta'}\in\mathcal Z$.
\end{proof}

\subsection{Proof of Theorem \ref{thm:robust}}\label{subsec:robust}

Lemma \ref{lem:subaction} is also the main ingredient in the proof of
Theorem \ref{thm:robust}.

\begin{proof}[Proof of Theorem \ref{thm:robust}]
The support $\Omega:=\mathrm{supp}(\mu)$ is a subshift, since $\mu$ is $\sigma$-invariant.
 By
hypothesis there are $f\in\lip(S_\beta,d_\theta)$ and $r>0$ such that
$\mu\in\Mmax(f+g)$ for all $g\in\lip(S_\beta,d_\theta)$ with
$\|g\|_\lip<r$.  Let $\delta:=r/2$.  If $g\in\lip(S_\beta,d_\theta)$
satisfies $|g|_\lip<\delta$, and $x_0\in S_\beta$, then
$\|g-g(x_0)\|_\lip\le2|g|_\lip<r$, as $S_\beta$ has $d_\theta$-diameter at
most $1$; since adding a constant does not change maximizing measures, it
follows that
\begin{equation}\label{eq:robust}
  \mu\in\Mmax(f+g)\qquad\text{for all }g\in\lip(S_\beta,d_\theta)\text{
  with }|g|_\lip<\delta .
\end{equation}
Let $L:=|f|_\lip+\delta$ and $M:=\bigl(1+2\theta(1-\theta)^{-1}\bigr)L$, and
fix $\ell\in\N$ such that $\theta^{\ell+1}(1-\theta)^{-1}L\le\delta/16$.

Let $k\in\N$, and set $\eta_k:=\delta\theta^k/2$ and
$\lambda_k:=\delta\theta^k/16$.  Let $\mathcal F_k$ be the set consisting of
$f$ and the functions $f_v:=f+\eta_k\ind_{[v]_{S_\beta}}$,
$v\in\w_k(\Omega)$.  If $\ind_{[v]_{S_\beta}}$ takes different values at
$x,y\in S_\beta$, then $x$ and $y$ differ in one of their first $k$
symbols, so $d_\theta(x,y)\ge\theta^{k-1}$; hence
$|\eta_k\ind_{[v]_{S_\beta}}|_\lip\le\eta_k\theta^{1-k}=\delta\theta/2<\delta$,
and \eqref{eq:robust} gives $\mu\in\Mmax(g)$ for all $g\in\mathcal F_k$,
while $|g|_\lip\le L$.  For $g\in\mathcal F_k$ let $q_g:=g-\int g\,d\mu$,
and let $u_g$ be the function \eqref{eq:subaction} for $g$, so that
$0\le u_g\le M$ by Lemma \ref{lem:subaction}(i).  Let $\mathcal Z_k$ be the
set of cylinder sets $[w]_\Omega$, $w\in\w_{k+\ell}(\Omega)$, which is a
partition of $\Omega$.  If $Z\in\mathcal Z_k$ and $x,y\in Z$, then
$d_\theta(x,y)\le\theta^{k+\ell}$, so Lemma \ref{lem:subaction}(iii) gives
\begin{equation}\label{eq:almost-monotone}
  u_g(y)-u_g(x)\le\lambda_k\qquad\text{for all }g\in\mathcal F_k,\
  Z\in\mathcal Z_k\text{ and }x,y\in Z\text{ with }x\le_\lex y .
\end{equation}

For $Z\in\mathcal Z_k$, $g\in\mathcal F_k$ and $x\in Z$, let
$b_{g,Z}(x):=\inf\{u_g(z):z\in Z,\ z\le_\lex x\}$.  Then $b_{g,Z}$ is
non-increasing on $Z$ with respect to $\le_\lex$, and by
\eqref{eq:almost-monotone},
\begin{equation}\label{eq:b-approx}
  0\le u_g-b_{g,Z}\le\lambda_k\quad\text{and}\quad0\le b_{g,Z}\le M
  \qquad\text{on }Z .
\end{equation}
For $i\in\N_0$ let $Z^i$ be the set of $x\in Z$ such that
$\sum_{g\in\mathcal F_k}\lfloor b_{g,Z}(x)/\lambda_k\rfloor=i$, and let
$\mathcal D_k$ be the set of all non-empty $Z^i$ with $Z\in\mathcal Z_k$ and
$i\in\N_0$, which is a partition of $\Omega$.  As
$|\mathcal F_k|\le|\w_k(\Omega)|+1\le2\lceil\beta\rceil^k$, the set $Z^i$ is empty for
$i>2\lceil\beta\rceil^kM/\lambda_k$, and as $|\mathcal Z_k|\le\lceil\beta\rceil^{k+\ell}$, it follows
that
\begin{equation}\label{eq:partition-count}
  |\mathcal D_k|\le\lceil\beta\rceil^{k+\ell}\bigl(1+2\lceil\beta\rceil^kM/\lambda_k\bigr)
  \le\Gamma\,(\lceil\beta\rceil^2/\theta)^k ,\qquad\text{where }\Gamma:=\lceil\beta\rceil^\ell(1+32M/\delta) .
\end{equation}
If $x,y\in Z^i$, say with $x\le_\lex y$, then
$\lfloor b_{g,Z}(y)/\lambda_k\rfloor\le\lfloor b_{g,Z}(x)/\lambda_k\rfloor$
for all $g\in\mathcal F_k$, and the sums of these integers over
$g\in\mathcal F_k$ are equal, so $0\le b_{g,Z}(x)-b_{g,Z}(y)<\lambda_k$ for
all $g\in\mathcal F_k$.  Together with \eqref{eq:b-approx} this gives
\begin{equation}\label{eq:small-oscillation}
  |u_g(x)-u_g(y)|<2\lambda_k\qquad\text{for all }g\in\mathcal F_k,\
  P\in\mathcal D_k\text{ and }x,y\in P .
\end{equation}

Now let $n\ge k$, and let $p:=n-k+1$.  By Lemma \ref{lem:subaction}(ii) and
the invariance of $\mu$, the set $\Omega_{k,n}$ of those $x\in\Omega$ such
that $u_g(\sigma^{j+1}x)-u_g(\sigma^jx)=q_g(\sigma^jx)$ for all
$g\in\mathcal F_k$ and $0\le j<p$ satisfies $\mu(\Omega_{k,n})=1$.  For
$w\in\w_n(\Omega)$ the set $[w]_\Omega$ is non-empty and open in $\Omega$,
so $\mu([w]_\Omega)>0$, and we may choose a point
$x(w)\in[w]_\Omega\cap\Omega_{k,n}$; summing over $0\le j<p$ then gives
\begin{equation}\label{eq:cocycle}
  S_pq_g(x(w))=u_g(\sigma^px(w))-u_g(x(w))\qquad\text{for all
  }g\in\mathcal F_k .
\end{equation}
For $P,P'\in\mathcal D_k$ let $\w^{P,P'}$ be the set of $w\in\w_n(\Omega)$
such that $x(w)\in P$ and $\sigma^px(w)\in P'$, so that $\w_n(\Omega)$ is the
union of these sets.  Let $w,w'\in\w^{P,P'}$ and $v\in\w_k(\Omega)$.  As
$q_{f_v}-q_f=\eta_k\bigl(\ind_{[v]_{S_\beta}}-\mu([v]_{S_\beta})\bigr)$,
applying \eqref{eq:cocycle} to $f_v$ and to $f$, at $x(w)$ and at $x(w')$,
gives
\[
  \eta_k\bigl(S_p\ind_{[v]_{S_\beta}}(x(w))-S_p\ind_{[v]_{S_\beta}}(x(w'))\bigr)
  =\Delta_{f_v}-\Delta_f ,
\]
where $\Delta_g:=u_g(\sigma^px(w))-u_g(\sigma^px(w'))-u_g(x(w))+u_g(x(w'))$
satisfies $|\Delta_g|<4\lambda_k=\eta_k/2$, by
\eqref{eq:small-oscillation}.  So the integers
$S_p\ind_{[v]_{S_\beta}}(x(w))$ and $S_p\ind_{[v]_{S_\beta}}(x(w'))$ differ
by less than $1$, and are therefore equal.  As $x(w)_{[0,n)}=w$, the first
of these integers is the number of $j\in\{0,\ldots,n-k\}$ such that
$w_{[j,j+k)}=v$, and similarly for the second.  Thus all $w\in\w^{P,P'}$
have the same set $A$ of subwords of length $k$, and $|A|\le p\le n$.  Each
$w\in\w^{P,P'}$ is the concatenation of $\lfloor n/k\rfloor$ words belonging
to $A$ and a word of length less than $k$, so $|\w^{P,P'}|\le n^{n/k}\lceil\beta\rceil^k$,
and \eqref{eq:partition-count} gives
\[
  |\w_n(\Omega)|\le|\mathcal D_k|^2\,n^{n/k}\,\lceil\beta\rceil^k
  \le\Gamma^2(\lceil\beta\rceil^5/\theta^2)^k\,n^{n/k} .
\]
For $n\ge2$ we may take $k:=\lceil\sqrt{n\log n}\,\rceil$, so that $k\le n$
and $(n/k)\log n\le\sqrt{n\log n}$; as $\sqrt{n\log n}\ge1$ and $\Gamma\ge1$,
this gives $\log|\w_n(\Omega)|\le C\sqrt{n\log n}$, where
$C:=2\log\Gamma+2\log(\lceil\beta\rceil^5/\theta^2)+1$.  Finally, the topological entropy of
$\Omega$ is $\lim_{n\to\infty}n^{-1}\log|\w_n(\Omega)|$ (see e.g.~\cite[Ch.~7]{Wal82}),
which is therefore zero.
\end{proof}

\subsection{Proof of Theorem \ref{thm:dimension}}\label{subsec:dimension}

Theorem \ref{thm:dimension} is deduced from Theorems \ref{thm:nonemergent}
and \ref{thm:robust} by means of a dimension estimate (Lemma
\ref{lem:zerodim}).

For $t>1$ and a bounded sequence $x=(x_i)_{i\ge0}$ of non-negative integers
let $h_t(x):=\sum_{i\ge0}x_it^{-i-1}$, as in
Section \ref{subsec:transformations}, and note that
$h_\beta(d^*_\beta(1))=1$ for all $\beta>1$, as $d^*_\beta(1)$ is an
expansion of $1$ in base $\beta$.  We first show, using an estimate for the
parameters $\beta$ for which $d^*_\beta(1)$ has a given prefix
(cf.~\cite{Sch97}), that the parameters $\beta$ for which the orbit closure
of $d^*_\beta(1)$ has zero topological entropy form a set of packing
dimension zero.

\begin{lem}\label{lem:param-cylinder}
Let $B\ge2$ be an integer, let $b\in(1,B)$ and $n\in\N$, and let
$\beta,\beta'\in[b,B)$ be such that
$d^*_\beta(1)$ and $d^*_{\beta'}(1)$ have the same first $n$ symbols.  Then
$|\beta-\beta'|\le B^2(B-1)b^{-n}/(b-1)$.
\end{lem}

\begin{proof}
Let $\mathbf e:=d^*_\beta(1)$ and $\mathbf e':=d^*_{\beta'}(1)$, whose
terms lie in $\{0,1,\ldots,B-1\}$ as $\beta,\beta'<B$, and define
$\varphi:[b,B]\to\R$ by $\varphi(t):=h_t(\mathbf e)$.  As $e_0\ge1$, the
function $\varphi$ is differentiable, with $\varphi'(t)\le-t^{-2}\le-B^{-2}$.
As $\varphi(\beta)=1=h_{\beta'}(\mathbf e')$, and $\mathbf e$ and
$\mathbf e'$ have the same first $n$ symbols, the mean value theorem gives
\begin{align*}
  B^{-2}|\beta-\beta'|\le|\varphi(\beta)-\varphi(\beta')|
  &=|h_{\beta'}(\mathbf e')-h_{\beta'}(\mathbf e)|\\
  &\le\sum_{i\ge n}(B-1)(\beta')^{-i-1}=\frac{(B-1)(\beta')^{-n}}{\beta'-1}
  \le\frac{(B-1)b^{-n}}{b-1} .\qedhere
\end{align*}
\end{proof}

Lemma \ref{lem:param-cylinder} gives the following dimension estimate.

\begin{lem}\label{lem:zerodim}
Let $E\subseteq(1,\infty)$ be such that for all $\beta\in E$ the orbit closure
$\Omega_\beta$ of $d^*_\beta(1)$ has zero topological entropy.  Then $E$ has
packing dimension zero, and hence Hausdorff dimension zero.
\end{lem}

\begin{proof}
The packing dimension of a countable union of sets is the supremum of their
packing dimensions, the packing dimension of a bounded set is at most its
upper box dimension, and Hausdorff dimension is at most packing dimension
\cite[Ch.~3]{Fal03}.  So it suffices to show that $E\cap[b,B)$ has packing
dimension zero for all integers $B\ge2$ and all $b\in(1,B)$.
Fix such $B$ and $b$, and $\varepsilon>0$, and for $n\in\N$ let
$E_n$ be the set of $\beta\in E\cap[b,B)$ such that
$|\w_n(\Omega_\beta)|\le e^{\varepsilon n}$.
As the topological entropy of $\Omega_\beta$ is
$\lim_{n\to\infty}n^{-1}\log|\w_n(\Omega_\beta)|$ \cite[Ch.~7]{Wal82}, which
vanishes for $\beta\in E$, the set $E\cap[b,B)$ is the union of the sets
$E_n$, $n\in\N$,
so it suffices to show that each non-empty $E_n$ has
upper box dimension at most $\varepsilon/\log b$.  Fix $n\in\N$, and let
$j\in\N$.  For $\beta\in E_n$, the word $d^*_\beta(1)_{[0,jn)}$ is a
concatenation of $j$ words belonging to $\w_n(\Omega_\beta)$, which is one
of the $2^{B^n}$ subsets of $\{0,1,\ldots,B-1\}^n$, and has at most $e^{\varepsilon n}$
elements.  So at most $2^{B^n}e^{\varepsilon jn}$ words arise as
$d^*_\beta(1)_{[0,jn)}$ with $\beta\in E_n$, and by Lemma
\ref{lem:param-cylinder}, $E_n$ is covered by at most
$2^{B^n}e^{\varepsilon jn}$ sets of diameter at most
$\delta_j:=B^2(B-1)b^{-jn}/(b-1)$.  As $\delta_{j+1}=b^{-n}\delta_j$, the
upper box dimension of $E_n$ may be computed along the sequence
$(\delta_j)_{j\in\N}$ \cite[Ch.~3]{Fal03}, so it is at most
\[
  \lim_{j\to\infty}\frac{\log\bigl(2^{B^n}e^{\varepsilon jn}\bigr)}{-\log\delta_j}
  =\frac{\varepsilon}{\log b} .\qedhere
\]
\end{proof}

\begin{rem}
The statement about Hausdorff dimension in    
Lemma \ref{lem:zerodim} can also be deduced from a theorem of Bowen
\cite[Thm.~2]{Bow73}, applied to the shift on $\Sigma_{B-1}$: the set of
$x\in\Sigma_{B-1}$ whose orbit quasi-generates some invariant measure of zero
entropy has zero topological entropy in the sense of \cite{Bow73}, and so has
Hausdorff dimension zero with respect to all metrics $d_\theta$, while by
Lemma \ref{lem:param-cylinder} the map $d^*_\beta(1)\mapsto\beta$, defined for
$\beta\in[b,B)$, is Lipschitz with respect to $d_{b^{-1}}$, for all
integers $B\ge2$ and all $b\in(1,B)$, and Lipschitz maps do not increase
Hausdorff dimension.
\end{rem}

The following lemma collects the consequences, for emergent parameters, of
the structural results of \cite{HJXZ} and of Theorem \ref{thm:robust}.

\begin{lem}\label{lem:emergent}
Let $\beta>1$ be emergent and not a beta-number, let $\theta\in(0,1)$,
and let $\Omega_\beta$ be the orbit closure of $d_\beta(1)$.
\begin{enumerate}
\item[(i)] Some open dense subset of $\lip(S_\beta,d_\theta)$ consists of
functions which either have the periodic optimization property, or have a
maximizing measure supported by $\Omega_\beta$.
\item[(ii)] If $S_\beta$ does not have TPO with respect to
$\lip(S_\beta,d_\theta)$, then
$\Omega_\beta$ is uniquely ergodic and has zero topological entropy, and its
invariant measure is a maximizing measure for all functions in some
non-empty open subset of $\lip(S_\beta,d_\theta)$.
\end{enumerate}
\end{lem}

\begin{proof}
As $\beta$ is not a beta-number, $d^*_\beta(1)=d_\beta(1)$ is not eventually
periodic, so $\Omega_\beta$ is infinite, and as $\beta$ is emergent,
$\Omega_\beta$ is minimal \cite[Prop.~5.3]{HHJL}.  So $\Omega_\beta$ is the
$\omega$-limit set of $d_\beta(1)$, which is the Markov boundary of $S_\beta$
(Lemma \ref{lem:markov-boundary}), and (i) follows from
\cite[Thm.~9.13]{HJXZ}.

(ii) As $\Omega_\beta$ is minimal and infinite, $0^\infty\notin\Omega_\beta$,
and by compactness some word $0^N$ does not occur in $d_\beta(1)$.  Hence
$S_\beta$ has the specification property \cite{Bertrand86,Bertrand88}; this
also follows from the proof of Lemma \ref{lem:spec}, which applies with
$d_\beta(1)$ in place of $\bfa$.  If $\Omega_\beta$ were not uniquely
ergodic, then $S_\beta$ would be a specimen shift space in the sense of
\cite[Def.~10.12]{HJXZ}, and so would have TPO \cite[Thm.~10.13]{HJXZ}.  So
$\Omega_\beta$ is uniquely ergodic, and its invariant measure $\mu$ has
support $\Omega_\beta$, by minimality.  As $S_\beta$ does not have TPO,
\cite[Cor.~10.1]{HJXZ} shows that the set of $f\in\lip(S_\beta,d_\theta)$
having a maximizing measure supported by $\Omega_\beta$, that is, the set of
$f$ for which $\mu$ is a maximizing measure, has non-empty interior.  So
$\Omega_\beta$ has zero topological entropy, by Theorem \ref{thm:robust}.
\end{proof}

Theorem \ref{thm:dimension} now follows from Lemmas \ref{lem:zerodim} and
\ref{lem:emergent}, together with Theorem \ref{thm:nonemergent} and
Corollary \ref{cor:G}.

\begin{proof}[Proof of Theorem \ref{thm:dimension}]
Let $E$ be the union of the three sets in Theorem \ref{thm:dimension}; it
suffices to show that $E$ has packing dimension zero.  Let $\beta\in E$ be
such that $\beta$ is not a simple beta-number.
Then there
exists $\theta\in(0,1)$ such that $S_\beta$ does not have TPO with respect to
$\lip(S_\beta,d_\theta)$: if $T_\beta$ or $U_\beta$ does not have TPO with
respect to $C^{0,\gamma}(I)$, then this holds with $\theta:=\beta^{-\gamma}$,
by Corollary \ref{cor:G}.
So $\beta$ is emergent, by Theorem
\ref{thm:nonemergent}, and is therefore not a beta-number
\cite[Cor.~5.4]{HHJL}, and Lemma \ref{lem:emergent}(ii) shows that the orbit
closure of $d^*_\beta(1)=d_\beta(1)$ has zero topological entropy.  By Lemma
\ref{lem:zerodim}, the parameters with this property form a set of
packing dimension zero, so $E$ is contained in the union of a set of
packing dimension zero and the countable set of simple beta-numbers, and
therefore has packing dimension zero, and hence Hausdorff dimension zero
\cite[Ch.~3]{Fal03}.
\end{proof}

\subsection{Proof of Theorem \ref{thm:TZO}}\label{subsec:TZO}

We now prove Theorem \ref{thm:TZO}, using Lemma \ref{lem:emergent},
together with \cite[Thm.~5.22]{HHJL} for $U_\beta$.

\begin{proof}[Proof of Theorem \ref{thm:TZO}]
As noted before the statement of Theorem \ref{thm:TZO}, a system with TPO
has typical zero-entropy optimization.  For $\beta>1$ which is emergent and
not a beta-number, let $\Omega_\beta$ denote the orbit closure of $d_\beta(1)$.

Let $\theta\in(0,1)$, and suppose that $S_\beta$ does not have TPO with
respect to $\lip(S_\beta,d_\theta)$.  Then $\beta$ is emergent, by Theorem
\ref{thm:nonemergent}, and is not a simple beta-number, since otherwise
$S_\beta$ would be a shift of finite type, and would have TPO
\cite{Con16}; so $\beta$ is not a beta-number
\cite[Cor.~5.4]{HHJL}.  By Lemma \ref{lem:emergent}(ii), $\Omega_\beta$ is
uniquely ergodic, with invariant measure $\mu_\beta$ say, and has zero
topological entropy.  For $\mu\in\m(S_\beta,\sigma)$, the set of
$f\in\lip(S_\beta,d_\theta)$ with $\mu\in\Mmax(f)$ is closed; let
$\mathcal L_\mu$ denote its interior.  All $f\in\mathcal L_\mu$ satisfy
$\Mmax(f)=\{\mu\}$: otherwise some $\nu\in\Mmax(f)$ differs from $\mu$, and
as Lipschitz functions separate measures there is
$h\in\lip(S_\beta,d_\theta)$ with $\int h\,d\nu>\int h\,d\mu$, so that
$\mu\notin\Mmax(f+\varepsilon h)$ for all $\varepsilon>0$.  Let $\mathcal E$
be the countable set consisting of $\mu_\beta$ and the periodic orbit
measures, and let $\mathcal G$ be the open dense set of Lemma
\ref{lem:emergent}(i), so that each $f\in\mathcal G$ has a maximizing
measure belonging to $\mathcal E$.  If
$\mathcal V\subseteq\lip(S_\beta,d_\theta)$ is non-empty and open, then
$\mathcal V\cap\mathcal G$ is a non-empty open set covered by the countably
many closed sets $\{f:\mu\in\Mmax(f)\}$, $\mu\in\mathcal E$, so by the
Baire category theorem $\mathcal V$ intersects $\mathcal L_\mu$ for some
$\mu\in\mathcal E$.  Hence the union of the sets $\mathcal L_\mu$,
$\mu\in\mathcal E$, is an open dense subset of $\lip(S_\beta,d_\theta)$,
each member of which has a unique maximizing measure, whose support is a
periodic orbit or $\Omega_\beta$.  As $\Omega_\beta$ has zero topological
entropy, all invariant measures supported by $\Omega_\beta$ have zero
entropy, by the variational principle \cite[Ch.~8]{Wal82}, so $S_\beta$ has
typical zero-entropy optimization with respect to $\lip(S_\beta,d_\theta)$.

Now let $\gamma\in(0,1]$, and suppose that $U_\beta$ does not have TPO with
respect to $C^{0,\gamma}(I)$.  Then $\beta$ is emergent and not a
beta-number \cite[Thms.~I, J]{HHJL}, and $S_\beta$ does not have TPO with
respect to $\lip(S_\beta,d_\theta)$, where $\theta:=\beta^{-\gamma}$, by
Corollary \ref{cor:G}; so, as above, $\Omega_\beta$ is uniquely ergodic, with
invariant measure $\mu_\beta$, and has zero topological entropy.
By \cite[Thm.~5.22]{HHJL}, $C^{0,\gamma}(I)$ is the union of
$\mathrm{Crit}^\gamma(\beta)$ and the closure of the set
$\mathrm{Lock}^\gamma(U_\beta)$ of functions with the locking property,
which is open and consists of functions with the periodic optimization
property \cite[\S 1.5]{HHJL}.  A non-empty open set disjoint from
$\mathrm{Lock}^\gamma(U_\beta)$ is therefore contained in
$\mathrm{Crit}^\gamma(\beta)$, and hence in its interior, so the open set
$\mathrm{Lock}^\gamma(U_\beta)\cup\mathrm{int}(\mathrm{Crit}^\gamma(\beta))$
is dense in $C^{0,\gamma}(I)$.
By \cite[Lem.~5.10]{HHJL}, each
$g\in\mathrm{Crit}^\gamma(\beta)$ has a maximizing measure $H_\beta(\mu)$,
where $H_\beta$ is as in \eqref{eq:Hbeta_def},
and
$\mu\in\m(S_\beta,\sigma)$ is quasi-generated by the orbit of
$d_\beta(1)$ (see the proof of Corollary \ref{cor:G}), so that
$\mu(\Omega_\beta)=1$, and hence $\mu=\mu_\beta$.  Thus all members of
$\mathrm{Crit}^\gamma(\beta)$ have $\nu_\beta:=H_\beta(\mu_\beta)$ as a
maximizing measure, and the argument used above for $S_\beta$, with
H\"older functions in place of Lipschitz functions, shows that $\nu_\beta$
is the unique maximizing measure for all members of
$\mathrm{int}(\mathrm{Crit}^\gamma(\beta))$.
As in the proof of Lemma \ref{lem:emergent}, some
word $0^N$ does not occur in any element of $\Omega_\beta$, so Lemma \ref{lem:key},
with $K:=\Omega_\beta$, shows that $h_\beta$ is injective on $\Omega_\beta$
and that $h_\beta(\Omega_\beta)\subseteq[\beta^{-N},1]$.  For
$x\in\Omega_\beta$ it follows that $\beta h_\beta(x)=x_0+h_\beta(\sigma x)$
lies in $(x_0,x_0+1]$, so that $U_\beta(h_\beta(x))=h_\beta(\sigma x)$.  Thus
$h_\beta$ restricts to a homeomorphism from $\Omega_\beta$ onto
$h_\beta(\Omega_\beta)$ which conjugates $\sigma$ to $U_\beta$, so the only
$U_\beta$-invariant measure supported by $h_\beta(\Omega_\beta)$, which is
the support of $\nu_\beta$, is $\nu_\beta$ itself, and it has zero entropy.
As all members of $\mathrm{Lock}^\gamma(U_\beta)$ have the periodic
optimization property, it follows that $U_\beta$ has typical zero-entropy
optimization with respect to $C^{0,\gamma}(I)$.

Finally, suppose that $\beta$ is not a simple beta-number.  By (C2),
$\m(I,T_\beta)=\m(I,U_\beta)$, and $\Mmax(T_\beta,g)=\Mmax(U_\beta,g)$ for all
$g\in C(I)$.  The maps $T_\beta$ and $U_\beta$ differ only at the points
$c/\beta$ with $c\in\N$ and $c\le\beta$, and $T_\beta^{-1}(\{0\})$ contains
$0$ and these points.  So, by invariance, all $\nu\in\m(I,T_\beta)$ give
zero measure to these points, hence $T_\beta$ and $U_\beta$ agree
$\nu$-almost everywhere, and $\nu$ has the same entropy with respect to
$T_\beta$ as with respect to $U_\beta$.  Consequently, if $g\in C(I)$ has a
unique maximizing measure $\mu$ for $U_\beta$, such that all
$U_\beta$-invariant measures supported by $\mathrm{supp}(\mu)$ have zero
entropy, then $\mu$ is also the unique maximizing measure of $g$ for
$T_\beta$, and all $T_\beta$-invariant measures supported by
$\mathrm{supp}(\mu)$ are $U_\beta$-invariant, so have zero entropy with
respect to $U_\beta$, and hence with respect to $T_\beta$.  As $U_\beta$
has typical zero-entropy optimization with respect to $C^{0,\gamma}(I)$ for
all $\gamma\in(0,1]$, it follows that so does $T_\beta$.
\end{proof}

The arguments in the proof of Theorem \ref{thm:TZO} also identify the
robustly maximizing measures of Theorem \ref{thm:robust}, as follows.

\begin{rem}\label{rem:robust}
Let $\beta>1$ and $\theta\in(0,1)$, and let $\mu\in\m(S_\beta,\sigma)$
be a maximizing measure for all functions in a non-empty open set
$\mathcal U\subseteq\lip(S_\beta,d_\theta)$.  If $\mu$ is not a periodic
orbit measure, then $\beta$ is emergent and not a beta-number, and the orbit
closure $\Omega_\beta$ of $d_\beta(1)$ is uniquely ergodic, with $\mu$ as its
invariant measure.  To see this, suppose that $\mu$ is not a periodic orbit
measure.  Then no $f\in\mathcal U$ has the periodic optimization property,
since otherwise $\Mmax(f)=\{\mu\}$; in particular $S_\beta$ does not have
TPO.  So, as in the proof of Theorem \ref{thm:TZO}, $\beta$ is emergent and
not a beta-number, and Lemma \ref{lem:emergent}(ii) shows that $\Omega_\beta$
is uniquely ergodic, with invariant measure $\mu_\beta$ say.  If $\mathcal G$
is the open dense set of Lemma \ref{lem:emergent}(i), then all members of the
non-empty open set $\mathcal U':=\mathcal U\cap\mathcal G$ have a maximizing
measure supported by $\Omega_\beta$, which is $\mu_\beta$.  If
$h\in\lip(S_\beta,d_\theta)$, $f\in\mathcal U'$, and $\varepsilon>0$ is such
that $f+\varepsilon h\in\mathcal U'$, then $\mu$ and $\mu_\beta$ are
maximizing measures for both $f$ and $f+\varepsilon h$, so
$\int h\,d\mu=\int h\,d\mu_\beta$; as Lipschitz functions separate measures,
it follows that $\mu=\mu_\beta$.  By Theorem \ref{thm:A}, this case occurs
whenever $d_\beta(1)$ is a Sturmian sequence.
\end{rem}

\section{Tuned Sturmian beta-shifts, and density of parameters without TPO}\label{sec:density}

By Theorem \ref{thm:A}, the beta-shift $S_\beta$ does not have TPO whenever
$d_\beta(1)$ is a Sturmian sequence, but the set of such parameters $\beta$
has nowhere dense closure (Remark \ref{rem:sturmian-closure}), so Theorem
\ref{thm:A} alone does not suffice to prove Theorem \ref{thm:density}.  In
this section we enlarge this set of parameters by means of a tuning construction,
which replaces the digits $0$ and $1$ of a maximal Sturmian sequence by two
words of equal length, and deduce Theorem \ref{thm:density}.  The first step
towards proving Theorem \ref{thm:density}
is combinatorial: the resulting tuned Sturmian
parameters accumulate at all simple beta-numbers, and are therefore dense in
$(1,\infty)$ (Corollary \ref{cor:tuned-dense}).  The second step is a
version of Theorems \ref{thm:A} and \ref{thm:C} for the corresponding tuned
Sturmian beta-shifts, showing that, like the Sturmian ones, they do not have
TPO (Theorem \ref{thm:tuned}).  As for Theorem \ref{thm:C}, its proof
applies Lemma \ref{lem:rigid-stability}, now combined with a version of
Theorem \ref{thm:D} for tuned Sturmian subshifts (Theorem
\ref{thm:tuned-rigidity}), which is deduced from Theorem \ref{thm:D} itself,
and a version of Proposition \ref{prop:Phimax} (Proposition
\ref{prop:tuned-max}).

The overall construction is an analogue, for beta-shifts, of Douady--Hubbard
tuning of external angles of the Mandelbrot set, which involves replacing the digits $0$ and $1$ in the binary
expansion of an angle by the periodic blocks of the binary expansions of the
two external angles of the root of a hyperbolic component
(see \cite{Dou86,DH85}). Bullett \& Sentenac
\cite[Part~2, \S 2]{bullettsentenac} used compositions of this operation to
construct minimal closed invariant subsets for the doubling map with a
prescribed sequence of rotation numbers.  
The maximal Sturmian sequences also describe external angles of the
Mandelbrot set: for irrational $\alpha\in(0,1)$, $\sigma(\bfa(\alpha))$ is
the binary expansion of the external angle of the point of the boundary of
the main cardioid with internal argument $\alpha$, by
\cite[\S 6, Thm.~1]{Dou86} and \eqref{eq:prefixsum}.
Substitutions of this kind also
describe renormalisation in the kneading theory of Lorenz maps \cite{GS93},
and appear in the study of the doubling map and beta-transformations
with a hole \cite{AK23,GlS15}, where sequences satisfy
lexicographic bounds from both above and below.  The condition defining a
beta-shift is a bound from above only, and this allows the two words used
here to differ in a single digit.

\subsection{Tuning words}

We begin with the words used in the construction.

\begin{defn}\label{def:tuning}
A \emph{tuning word} is a word $w=w_0w_1\cdots w_{p-1}$ of length
$p\in\N$, whose letters are non-negative integers, such that $w_0\ge1$ and
$\sigma^iw^\infty<_\lex w^\infty$ for $0<i<p$.  For a tuning word $w$, let
$r$ be the largest index with $w_r\ge1$, let $w^-$ be the word obtained from
$w$ by replacing the digit $w_r$ by $w_r-1$, and let the \emph{tuning map}
$\zeta_w:\Sigma\to\Sigma_{w_0}$ be the map which replaces each digit $0$ by
the word $w^-$ and each digit $1$ by the word $w$; we use the same notation
$\zeta_w(v)$ for finite words $v$ over the digits $0$ and $1$.
\end{defn}

Note that $w^-$ differs from $w$ only at the index $r$, so 
$w^-<_\lex w$. The digits of a tuning word $w$ are at most $w_0$,
since $\sigma^iw^\infty\le_\lex w^\infty$ for all $i\ge0$, so $\zeta_w$ does
indeed take values in $\Sigma_{w_0}$.
If $p=1$ then $w$ is a single digit $w_0\ge1$
and $w^-$ is the digit $w_0-1$, so the tuning map $\zeta_w$ is the map
$\iota$ of Remark \ref{rem:alphabets} with $\ell=w_0-1$ and $m=w_0$; in
particular
$\zeta_w(\bfa)$ lies in $\iota(X_\alpha)$, so is a Sturmian sequence over the
alphabet $\{w_0-1,w_0\}$.

The required combinatorial properties of tuning words are described in the following Lemma~\ref{lem:tuning}
(note that part (ii) constitutes
a version of
\cite[Part~2, Lem.~2]{bullettsentenac}):

\begin{lem}\label{lem:tuning}
Let $w$ be a tuning word of length $p$, let $r$, $w^-$ and $\zeta:=\zeta_w$
be as in Definition \ref{def:tuning}, and for irrational $\alpha\in(0,1)$
let $\bfa(\alpha)$ be the maximal Sturmian sequence of rotation number
$\alpha$.  Then:
\begin{enumerate}
\item[(i)] for $0<i<p$, both $w_{[i,p)}$ and $w^-_{[i,p)}$ are
lexicographically smaller than $w_{[0,p-i)}$;
\item[(ii)] $\sigma^p\circ\zeta=\zeta\circ\sigma$, and if $y,y'\in\Sigma$
are distinct and $k$ is the smallest index with $y_k\ne y'_k$, then $pk+r$
is the smallest index at which $\zeta(y)$ and $\zeta(y')$ differ, and
$\zeta(y)<_\lex\zeta(y')$ if $y<_\lex y'$; in particular
$d_\theta(\zeta(y),\zeta(y'))=\theta^r\,d_{\theta^p}(y,y')$ for all
$y,y'\in\Sigma$ and $\theta\in(0,1)$, 
so $\zeta$ is continuous and strictly increasing with respect to $\le_\lex$,
and in particular injective;
\item[(iii)] for all $y\in\Sigma$ and $j\ge0$, we have
$\zeta(y)_{[j,j+p)}=w$ if and only if $j=pk$ for some $k\ge0$ with $y_k=1$;
\item[(iv)] for all irrational $\alpha\in(0,1)$, the sequence
$\zeta(\bfa(\alpha))$ satisfies $\sigma^j\zeta(\bfa(\alpha))<_\lex
\zeta(\bfa(\alpha))$ for all $j\ge1$, so there is a unique
$\beta_{w,\alpha}>1$ with $d_{\beta_{w,\alpha}}(1)=\zeta(\bfa(\alpha))$, and
$\beta_{w,\alpha}$ is not a beta-number;
\item[(v)] the map $\alpha\mapsto\beta_{w,\alpha}$, defined on the set of irrationals in
$(0,1)$, is strictly increasing; if $\beta_1>1$ satisfies
$d^*_{\beta_1}(1)=w^\infty$, then $\beta_{w,\alpha}<\beta_1$ for all
irrational $\alpha\in(0,1)$; and if $\beta'>1$ satisfies
$d^*_{\beta'}(1)<_\lex w^\infty$, then there exists $\alpha_0\in(0,1)$ such
that $\beta_{w,\alpha}>\beta'$ for all irrational $\alpha\in(\alpha_0,1)$.
\end{enumerate}
\end{lem}

\begin{proof}
(i) Let $0<i<p$.  The sequences $\sigma^iw^\infty$ and $w^\infty$ begin with
the words $w_{[i,p)}$ and $w_{[0,p-i)}$ respectively, both of length $p-i$,
and $\sigma^iw^\infty<_\lex w^\infty$; so it suffices to show that these two
words are distinct, as then $w_{[i,p)}<_\lex w_{[0,p-i)}$.  If they were
both equal to a word $u$, then $\sigma^iw^\infty=u\,w^\infty$ and
$w^\infty=u\,\sigma^{p-i}w^\infty$, so $\sigma^iw^\infty<_\lex w^\infty$
would give $w^\infty<_\lex\sigma^{p-i}w^\infty$, contradicting Definition
\ref{def:tuning} with $p-i$ in place of $i$.  Finally, each digit of $w^-$ is
at most the corresponding digit of $w$, so $w^-_{[i,p)}\le_\lex
w_{[i,p)}<_\lex w_{[0,p-i)}$.

(ii) As $w$ and $w^-$ both have length $p$, the map $\zeta$ replaces each
digit by a block of length $p$, so $\sigma^p\circ\zeta=\zeta\circ\sigma$, and
the first $pk$ symbols of $\zeta(y)$ form the word $\zeta(y_{[0,k)})$.  Let
$y,y'$ and $k$ be as in the statement, and suppose without loss of generality that
$y<_\lex y'$, so that $y_k=0$ and $y'_k=1$.  Then $\zeta(y)$ and $\zeta(y')$
have the same first $pk$ symbols, and their next $p$ symbols form the words
$w^-$ and $w$ respectively.  As $w^-$ and $w$ differ only at the index $r$,
where $w^-_r=w_r-1$, it follows that $pk+r$ is the smallest index at which
$\zeta(y)$ and $\zeta(y')$ differ, and that $\zeta(y)<_\lex\zeta(y')$.  Hence
$d_\theta(\zeta(y),\zeta(y'))=\theta^{pk+r}=\theta^r\,d_{\theta^p}(y,y')$.
This equality clearly also holds if $y=y'$. 
So $\zeta$ is continuous and strictly increasing, hence injective.

(iii) Write $j=pk+i$ with $k\ge0$ and $0\le i<p$.  If $i=0$ then
$\zeta(y)_{[j,j+p)}=\zeta(y_k)$, and this word is $w$ if $y_k=1$, and is $w^-\ne w$
if $y_k=0$.  If $0<i<p$ then $\zeta(y)_{[j,j+p)}$ begins with
$\zeta(y_k)_{[i,p)}$, which is 
either $w_{[i,p)}$ or $w^-_{[i,p)}$, and so it differs
from $w_{[0,p-i)}$ by (i); hence $\zeta(y)_{[j,j+p)}\ne w$, so
(iii) is proved.

(iv) Let $\bfa:=\bfa(\alpha)$, let $j\ge1$, and write $j=pk+i$ with $k\ge0$
and $0\le i<p$.  If $i=0$ then $k\ge1$, and
$\sigma^j\zeta(\bfa)=\zeta(\sigma^k\bfa)<_\lex\zeta(\bfa)$, by (ii) and Lemma
\ref{lem:sturmian}(v).  If $0<i<p$ then $\sigma^j\zeta(\bfa)$ begins with
$\zeta(a_k)_{[i,p)}$, which is lexicographically smaller than $w_{[0,p-i)}$
by (i), while $\zeta(\bfa)$ begins with $\zeta(a_0)=w$, as $a_0=1$; so again
$\sigma^j\zeta(\bfa)<_\lex\zeta(\bfa)$.  As the first digit of $\zeta(\bfa)$
is $w_0\ge1$, Parry's criterion \cite{Par60} gives a unique
$\beta_{w,\alpha}>1$ with $d_{\beta_{w,\alpha}}(1)=\zeta(\bfa)$.  If
$\zeta(\bfa)$ were eventually periodic, with some period $P\in\N$ from some
index $n\ge0$ onwards, then $pP$ would also be such a period, and for $k\ge0$ with
$pk\ge n$, (ii) would give
$\zeta(\sigma^{k+P}\bfa)=\sigma^{pk+pP}\zeta(\bfa)=\sigma^{pk}\zeta(\bfa)
=\zeta(\sigma^k\bfa)$, and hence $\sigma^{k+P}\bfa=\sigma^k\bfa$, as $\zeta$
is injective; but $\bfa$ is not eventually periodic, as its orbit closure
$X_\alpha$ is minimal and infinite (Lemma \ref{lem:sturmian}(i)).  So $\zeta(\bfa)$ is not eventually periodic, in other words
$\beta_{w,\alpha}$ is not a beta-number.

(v) Let $\alpha,\alpha'\in(0,1)$ be irrational,
with $\alpha<\alpha'$. 
By \eqref{eq:prefixsum},
$S_j\Phi(\bfa(\alpha))=\lceil j\alpha\rceil\le\lceil j\alpha'\rceil
=S_j\Phi(\bfa(\alpha'))$ for all $j\ge0$, with strict inequality whenever
$j(\alpha'-\alpha)>1$; so at the smallest index at which $\bfa(\alpha)$ and
$\bfa(\alpha')$ differ, the digit of $\bfa(\alpha)$ is $0$ and that of
$\bfa(\alpha')$ is $1$, so $\bfa(\alpha)<_\lex\bfa(\alpha')$.
So $d_{\beta_{w,\alpha}}(1)<_\lex
d_{\beta_{w,\alpha'}}(1)$, by (ii).  By (iv), neither parameter is a simple
beta-number, so these greedy expansions of $1$ are also the quasi-greedy
ones, and as $\beta\mapsto d^*_\beta(1)$ is strictly increasing
\cite[Prop.~3.9(xiii)]{HHJL}, it follows that
$\beta_{w,\alpha}<\beta_{w,\alpha'}$.  Next, if $d^*_{\beta_1}(1)=w^\infty$,
then $d^*_{\beta_{w,\alpha}}(1)=\zeta(\bfa(\alpha))<_\lex\zeta(1^\infty)=
w^\infty$, by (ii), so $\beta_{w,\alpha}<\beta_1$.  Finally, let $\beta'>1$
satisfy $d^*_{\beta'}(1)<_\lex w^\infty$, and choose an integer $n\ge2$ such
that $d^*_{\beta'}(1)_{[0,pn)}<_\lex w^n$.  If $\alpha\in(1-1/n,1)$ is
irrational and $1\le j\le n$, then $j-1\le j-j/n<j\alpha<j$, so $\lceil
j\alpha\rceil=j$, and \eqref{eq:prefixsum} shows that $\bfa(\alpha)$ begins
with $1^n$.  Then $\zeta(\bfa(\alpha))$ begins with $w^n$, so
$d^*_{\beta'}(1)<_\lex d^*_{\beta_{w,\alpha}}(1)$, and therefore
$\beta'<\beta_{w,\alpha}$.  So the last statement of (v) holds with
$\alpha_0:=1-1/n$.
\end{proof}

By Lemma \ref{lem:tuning}(v), if $w$ is a tuning word and $\beta_1>1$
satisfies $d^*_{\beta_1}(1)=w^\infty$, then $\beta_{w,\alpha}$ tends to
$\beta_1$ from below as $\alpha$ tends to $1$.  We show below that every
simple beta-number $\beta_1$ satisfies $d^*_{\beta_1}(1)=w^\infty$ for some
tuning word $w$; as simple beta-numbers are dense in $(1,\infty)$, this gives
the following.

\begin{cor}[Density of the parameters $\beta_{w,\alpha}$]\label{cor:tuned-dense}
Let $J\subseteq(1,\infty)$ be a non-empty open interval.  Then there is a
tuning word $w$, and $\alpha_0\in(0,1)$, such that for all irrational
$\alpha\in(\alpha_0,1)$, the parameter $\beta_{w,\alpha}$ of Lemma
\ref{lem:tuning}(iv), which satisfies
$d_{\beta_{w,\alpha}}(1)=\zeta_w(\bfa(\alpha))$, lies in $J$.  In particular
$J$ contains uncountably many such parameters, none of which is a
beta-number.
\end{cor}
\begin{proof}
As simple beta-numbers are dense in $(1,\infty)$ (see \cite{Par60}), there are
$\beta',\beta_1\in J$ with $\beta'<\beta_1$ such that $\beta_1$ is a simple
beta-number.  Then $d^*_{\beta_1}(1)$ is periodic, and
$\sigma^id^*_{\beta_1}(1)\le_\lex d^*_{\beta_1}(1)$ for all $i\ge0$ (see
Remark \ref{rem:general}).  Let $w$ be the shortest word such that
$d^*_{\beta_1}(1)=w^\infty$, and let $p$ be its length.  Then $w$ is a tuning
word: its first digit is $\lceil\beta_1\rceil-1\ge1$, and for $0<i<p$ we have
$\sigma^iw^\infty\ne w^\infty$, since otherwise $w^\infty$ would have the
periods $i$ and $p$, and hence the period $\gcd(i,p)<p$, contradicting the
choice of $w$; so $\sigma^iw^\infty<_\lex w^\infty$.  As $\beta\mapsto
d^*_\beta(1)$ is strictly increasing (see \cite[Prop.~3.9(xiii)]{HHJL}), we have
$d^*_{\beta'}(1)<_\lex w^\infty$, so by Lemma \ref{lem:tuning}(v) we can find
$\alpha_0\in(0,1)$ such that $\beta'<\beta_{w,\alpha}<\beta_1$, and hence
$\beta_{w,\alpha}\in J$, for all irrational $\alpha\in(\alpha_0,1)$.  By
Lemma \ref{lem:tuning}(v), the map $\alpha\mapsto\beta_{w,\alpha}$ is
strictly increasing, so these parameters are pairwise distinct, and there are
uncountably many of them. None of them is a beta-number, by Lemma
\ref{lem:tuning}(iv).
\end{proof}

By Corollaries \ref{cor:tuned-dense} and \ref{cor:G}, Theorem
\ref{thm:density} will follow once we show that none of the beta-shifts
$S_{\beta_{w,\alpha}}$ has TPO.  This is Theorem \ref{thm:tuned} below, and
most of the rest of this section is devoted to its proof.

\subsection{Tuned Sturmian subshifts}

By Lemma \ref{lem:tuning}(iv), tuning maps send maximal Sturmian sequences
to expansions $d_\beta(1)$, motivating the following terminology.

\begin{defn}\label{def:tuned}
Let $w$ be a tuning word of length $p$, let $\alpha\in(0,1)$ be
irrational, and let $\beta_{w,\alpha}$ be as in Lemma \ref{lem:tuning}(iv),
so that $d_{\beta_{w,\alpha}}(1)=\zeta_w(\bfa(\alpha))$.  We call
$\zeta_w(\bfa(\alpha))$ the \emph{maximal tuned Sturmian sequence},
$\beta_{w,\alpha}$ the \emph{tuned Sturmian parameter}, and the beta-shift
$Y_{w,\alpha}:=S_{\beta_{w,\alpha}}$ the \emph{tuned Sturmian beta-shift},
corresponding to $w$ and $\alpha$.
The \emph{tuned
Sturmian subshift} $X_{w,\alpha}$ and the \emph{tuned Sturmian measure}
$\mu_{w,\alpha}$, corresponding to $w$ and $\alpha$, are defined by
\[
  X_{w,\alpha}:=\bigcup_{i=0}^{p-1}\sigma^i\zeta_w(X_\alpha),\qquad
  \mu_{w,\alpha}:=\frac1p\sum_{i=0}^{p-1}\mu_\alpha\circ(\sigma^i\circ\zeta_w)^{-1} ,
\]
and the elements of $X_{w,\alpha}$ are called \emph{tuned Sturmian
sequences}.
\end{defn}

By Lemma \ref{lem:tuning}(iv), the tuned Sturmian parameter
$\beta_{w,\alpha}$ is not a beta-number, so
$d^*_{\beta_{w,\alpha}}(1)=d_{\beta_{w,\alpha}}(1)=\zeta_w(\bfa(\alpha))$,
whose first term $w_0$ is $\lceil\beta_{w,\alpha}\rceil-1$; hence the tuned
Sturmian beta-shift is
\begin{equation}\label{eq:Ywalpha}
  Y_{w,\alpha}=\bigl\{x\in\Sigma_{w_0}:\sigma^jx\le_\lex\zeta_w(\bfa(\alpha))
  \ \text{ for all }j\ge0\bigr\} .
\end{equation}
Taking $w=1$ in Definition \ref{def:tuned} gives the Sturmian case: $\zeta_1$
is the identity map on $\Sigma$, so $X_{1,\alpha}=X_\alpha$,
$\mu_{1,\alpha}=\mu_\alpha$, and $Y_{1,\alpha}=Y_\alpha$ by
\eqref{eq:Ywalpha} and \eqref{eq:Yalpha}.

The tuned Sturmian subshift has the following properties
(cf.~\cite[Part~2, Lem.~1, Thm.~7]{bullettsentenac}):

\begin{lem}\label{lem:tuned-subshift}
Let $w$ be a tuning word of length $p$, and let $w^-$ and $\zeta:=\zeta_w$
be as in Definition \ref{def:tuning}.  Let
$\alpha\in(0,1)$ be irrational, with $N_\alpha :=\lceil1/\alpha\rceil$ and
$\bfa:=\bfa(\alpha)$, and let
$Y_{w,\alpha}$, $X_{w,\alpha}$ and $\mu_{w,\alpha}$ be the tuned Sturmian
beta-shift, subshift and measure 
corresponding to $w$ and $\alpha$.  Then:
\begin{enumerate}
\item[(i)] $X_{w,\alpha}$ is the orbit closure of the maximal tuned Sturmian
sequence $\zeta(\bfa)$, and $X_{w,\alpha}\subseteq Y_{w,\alpha}$;
\item[(ii)] for all $x\in X_{w,\alpha}$ there is a unique
$i(x)\in\{0,1,\ldots,p-1\}$ with $x\in\sigma^{i(x)}\zeta(X_\alpha)$; the
word $w$ occurs in $x_{[0,(N_\alpha+1)p)}$; all $j\ge0$ with $x_{[j,j+p)}=w$
satisfy
$j\equiv-i(x)\pmod p$; and all $j\ge0$ with $j\equiv-i(x)\pmod p$ satisfy
$x_{[j,j+p)}\in\{w,w^-\}$;
\item[(iii)] $i(x)$ depends only on $x_{[0,(N_\alpha+1)p)}$, $i(\sigma x)\equiv
i(x)+1\pmod p$ for all $x\in X_{w,\alpha}$, and $\zeta$ restricts to a
homeomorphism from $X_\alpha$ onto $\{x\in X_{w,\alpha}:i(x)=0\}$;
\item[(iv)] $X_{w,\alpha}$ is minimal, infinite and uniquely ergodic, its
invariant measure is $\mu_{w,\alpha}$, which is not a periodic orbit
measure, and $\mu_{w,\alpha}([w])=\alpha/p$.
\end{enumerate}
\end{lem}

\begin{proof}
(i) For $k\ge0$ and $0\le i<p$, Lemma \ref{lem:tuning}(ii) gives
$\sigma^{pk+i}\zeta(\bfa)=\sigma^i\zeta(\sigma^k\bfa)$, which lies in
$X_{w,\alpha}$; and $X_{w,\alpha}$ is closed, since $\zeta$ is continuous and
$X_\alpha$ is compact.  Conversely, if $y\in X_\alpha$ and $0\le i<p$, then
$\sigma^{k_n}\bfa\to y$ for some sequence $(k_n)$, and
$\sigma^i\zeta(y)=\lim_n\sigma^{pk_n+i}\zeta(\bfa)$.
So $X_{w,\alpha}$ is
the orbit closure of $\zeta(\bfa)$.  Moreover $\zeta(\bfa)\in Y_{w,\alpha}$,
by Lemma \ref{lem:tuning}(iv) and \eqref{eq:Ywalpha}, and $Y_{w,\alpha}$ is
closed and satisfies $\sigma(Y_{w,\alpha})\subseteq Y_{w,\alpha}$, so
$X_{w,\alpha}\subseteq Y_{w,\alpha}$.

(ii) Let $x=\sigma^i\zeta(y)$ with $0\le i<p$ and $y\in X_\alpha$.  The
word $v:=y_{[1,N_\alpha+1)}$ has length $N_\alpha\ge1/\alpha$ and occurs in
$X_\alpha$, so contains the digit $1$, by Lemma \ref{lem:sturmian}(iv):
$y_h=1$ for some $h$ with $1\le h\le N_\alpha$.
As
$x_{[0,(N_\alpha+1)p)}=\zeta(y)_{[i,i+(N_\alpha+1)p)}$ contains
$\zeta(v)=\zeta(y)_{[p,(N_\alpha+1)p)}$, the word $w=\zeta(y_h)$ occurs in
$x_{[0,(N_\alpha+1)p)}$.  By Lemma \ref{lem:tuning}(iii), if
$x_{[j,j+p)}=\zeta(y)_{[i+j,i+j+p)}$ equals $w$ then $i+j\equiv0\pmod p$;
and if $i+j=pk$ then $x_{[j,j+p)}=\zeta(y_k)\in\{w,w^-\}$.  If also
$x\in\sigma^{i'}\zeta(X_\alpha)$ with $0\le i'<p$, then a position $j$ at
which $w$ occurs in $x$ satisfies $j\equiv-i$ and $j\equiv-i'\pmod p$, so
$i=i'$.  This proves (ii).

(iii) By (ii), $i(x)\equiv-j\pmod p$ for all $j$ with $x_{[j,j+p)}=w$,
and there is such a $j\le N_\alpha p$, so $i(x)$ depends only on
$x_{[0,(N_\alpha+1)p)}$.  If
$x=\sigma^i\zeta(y)$ with $0\le i<p$ and $y\in X_\alpha$, then
$\sigma x=\sigma^{i+1}\zeta(y)$ if $i<p-1$, and $\sigma
x=\sigma^p\zeta(y)=\zeta(\sigma y)$ if $i=p-1$, so $i(\sigma x)\equiv
i(x)+1\pmod p$.  By the uniqueness in (ii), $\{x\in
X_{w,\alpha}:i(x)=0\}=\zeta(X_\alpha)$, and $\zeta$ is continuous and
injective by Lemma \ref{lem:tuning}(ii), so restricts to a homeomorphism from
the compact set $X_\alpha$ onto $\zeta(X_\alpha)$.

(iv) Let $x=\sigma^i\zeta(y)$ with $0\le i<p$ and $y\in X_\alpha$.  The
orbit closure of $x$ contains $\sigma^{p-i}x=\zeta(\sigma y)$, therefore, since $\zeta$
is continuous and $\sigma^p\circ\zeta=\zeta\circ\sigma$, it contains the
image under $\zeta$ of the orbit closure of $\sigma y$, which is equal to
$\zeta(X_\alpha)$ by minimality of $X_\alpha$; hence it contains
$X_{w,\alpha}$.  Thus $X_{w,\alpha}$ is minimal, and it is infinite, since
$\zeta$ is injective and $X_\alpha$ is infinite.

Now let $\nu\in\m(X_{w,\alpha},\sigma)$ and $f\in C(X_{w,\alpha})$, and
define $\hat f\in C(X_\alpha)$ by $\hat f(y):=S_pf(\zeta(y))$, so that
$\int\hat f\,d\mu_\alpha=\sum_{i<p}\int f\circ\sigma^i\circ\zeta\,d\mu_\alpha
=p\int f\,d\mu_{w,\alpha}$.  For $y\in X_\alpha$, $0\le i<p$ and $n\in\N$,
the sum $S_{pn}f(\sigma^i\zeta(y))$ differs by at most
$2p\|f\|_{C(X_{w,\alpha})}$ from $S_{pn}f(\zeta(y))$, which equals
$S_n\hat f(y)$ since $\sigma^p\circ\zeta=\zeta\circ\sigma$.  
As $X_\alpha$ is
uniquely ergodic, $\frac{1}{n}S_n\hat f\to\int\hat f\,d\mu_\alpha$ uniformly on
$X_\alpha$ (cf.~\cite[Thm.~6.19]{Wal82}), so $\frac1{pn}S_{pn}f\to\int
f\,d\mu_{w,\alpha}$ uniformly on $X_{w,\alpha}$.  Since
$\int\frac1{pn}S_{pn}f\,d\nu=\int f\,d\nu$ for all $n\in\N$, by invariance
of $\nu$, it follows that $\int f\,d\nu=\int f\,d\mu_{w,\alpha}$.
As $f$ is arbitrary, $\nu=\mu_{w,\alpha}$; and as
$\m(X_{w,\alpha},\sigma)$ is non-empty, $X_{w,\alpha}$ is uniquely ergodic,
with invariant measure $\mu_{w,\alpha}$.  This is not a periodic orbit
measure, since the minimal infinite set $X_{w,\alpha}$ does not contain any periodic
orbit.  Finally, for $f=\ind_{[w]}$, Lemma \ref{lem:tuning}(iii) gives
$\hat f(y)=y_0=\Phi(y)$ for $y\in X_\alpha$, so
$p\mu_{w,\alpha}([w])=\int\Phi\,d\mu_\alpha=\alpha$, by Lemma
\ref{lem:sturmian}(iii).
\end{proof}

We now show that, for tuning words of length at least $2$, tuned Sturmian
subshifts are not topologically conjugate to Sturmian subshifts.

\begin{prop}\label{prop:tuned-not-sturmian}
Let $w$ be a tuning word of length $p\ge2$, let $\alpha\in(0,1)$ be
irrational, and let $X_{w,\alpha}$ be the tuned Sturmian subshift corresponding to $w$ and
$\alpha$.  Then $X_{w,\alpha}$ is not topologically conjugate to the
Sturmian subshift $X_{\alpha'}$ for any irrational $\alpha'\in(0,1)$.
Consequently no element of $X_{w,\alpha}$ is a Sturmian sequence over an
alphabet $\{\ell,m\}$, in the sense of Remark \ref{rem:alphabets}; in
particular, $d_{\beta_{w,\alpha}}(1)=\zeta_w(\bfa(\alpha))$ is not a Sturmian sequence over any $\{\ell,m\}$.
\end{prop}
\begin{proof}
By Lemma \ref{lem:tuned-subshift}(iii), the set
$\zeta_w(X_\alpha)=\{x\in X_{w,\alpha}:i(x)=0\}$ is clopen in
$X_{w,\alpha}$, as $i(x)$ depends only on $x_{[0,(N_\alpha+1)p)}$.  For
$0\le j<p$ and $y\in X_\alpha$ we have $i(\sigma^j\zeta_w(y))=j$, by Lemma
\ref{lem:tuned-subshift}(ii), so $\sigma^j\zeta_w(y)\in\zeta_w(X_\alpha)$ if
and only if $j=0$, and the definition of $\mu_{w,\alpha}$ gives
$\mu_{w,\alpha}(\zeta_w(X_\alpha))=1/p$.

Now let $\alpha'\in(0,1)$ be irrational.  By Lemmas \ref{lem:sturmian}(ii)
and \ref{lem:cuts}, with $\alpha'$ in place of $\alpha$, each cylinder set
$[v]_{X_{\alpha'}}$ has $\mu_{\alpha'}$-measure $|I_v|$, where $I_v$ is
either empty or an interval with endpoints in
$\{1\}\cup\{\{-j\alpha'\}:j\ge0\}$, which is contained in $\Z+\alpha'\Z$.
If $D$ is a clopen subset of $X_{\alpha'}$, then $D$ is a finite disjoint
union of cylinder sets, as $D$ is compact and open, so
$\mu_{\alpha'}(D)\in\Z+\alpha'\Z$; in particular $\mu_{\alpha'}(D)\ne1/p$,
as $p\ge2$ and $\alpha'$ is irrational.

If $h:X_{w,\alpha}\to X_{\alpha'}$ were a topological conjugacy, then
$\mu_{w,\alpha}\circ h^{-1}$ would be invariant, as $\mu_{w,\alpha}$ is
(Lemma \ref{lem:tuned-subshift}(iv)), and hence equal to $\mu_{\alpha'}$, by
unique ergodicity of $X_{\alpha'}$ (Lemma \ref{lem:sturmian}(i)); so the
clopen set $h(\zeta_w(X_\alpha))$ would have $\mu_{\alpha'}$-measure
$\mu_{w,\alpha}(\zeta_w(X_\alpha))=1/p$, contradicting the above,
so the first statement is proved.

Lastly, suppose that some $x\in X_{w,\alpha}$ is a Sturmian sequence over
an alphabet $\{\ell,m\}$,  i.e.~$x\in\iota(X_{\alpha'})$ for some
irrational $\alpha'\in(0,1)$, with $\iota$ as in Remark
\ref{rem:alphabets}.  As $X_{w,\alpha}$ and $\iota(X_{\alpha'})$ are
minimal (Lemma \ref{lem:tuned-subshift}(iv) and Remark
\ref{rem:alphabets}), both are equal to the orbit closure of $x$.  So
$X_{w,\alpha}=\iota(X_{\alpha'})$, and $\iota$ restricts to a topological
conjugacy from $X_{\alpha'}$ onto $X_{w,\alpha}$, which contradicts the
first statement.
\end{proof}

\subsection{Lipschitz cohomology of tuned Sturmian subshifts}

The proof of Theorem \ref{thm:tuned} has two main ingredients: a version of
Theorem \ref{thm:D} for tuned Sturmian subshifts, which we establish in this
subsection (Theorem \ref{thm:tuned-rigidity}), and a version of Proposition
\ref{prop:Phimax} for tuned Sturmian beta-shifts, which we establish in the
next subsection (see Proposition \ref{prop:tuned-max}).

By Lemma \ref{lem:tuned-subshift}(ii) and (iii), the tuned Sturmian
subshift $X_{w,\alpha}$ is the union of the $p$ pairwise disjoint clopen
sets $\sigma^i\zeta_w(X_\alpha)$, $0\le i<p$, which $\sigma$ permutes
cyclically, and the first return map of $\sigma$ to $\zeta_w(X_\alpha)$ is
$\sigma^p$, which $\zeta_w$ conjugates to $\sigma$ on $X_\alpha$.  Passing
from a function $f$ on $X_{w,\alpha}$ to the function $S_pf\circ\zeta_w$ on
$X_\alpha$, and applying Theorem \ref{thm:D} to $X_\alpha$ with the metric
$d_{\theta^p}$,
we obtain the following version of Theorem
\ref{thm:D}, in which the indicator function of the cylinder set $[w]$
plays the role of the digit function: by Lemma \ref{lem:tuning}(iii), the
word $w$ occurs in $\zeta_w(y)$ precisely at the positions $pk$ with
$y_k=1$, so that $S_p\ind_{[w]}(\zeta_w(y))=\Phi(y)$ for all $y\in\Sigma$.

\begin{thm}[\textbf{Lipschitz cohomological rigidity for tuned Sturmian
subshifts}]\label{thm:tuned-rigidity}
Let $w$ be a tuning word of length $p$, let $\alpha\in(0,1)$ be
irrational, with $N_\alpha:=\lceil1/\alpha\rceil$, let $X_{w,\alpha}$ and
$\mu_{w,\alpha}$ be the tuned Sturmian subshift and measure corresponding to $w$ and
$\alpha$, and let $\theta\in(0,1)$.  Then for
all $f\in\lip(X_{w,\alpha},d_\theta)$ there
is a unique constant $c_f\in\R$, and a function $\tilde
f\in\lip(X_{w,\alpha},d_\theta)$ which is unique up to an additive constant,
such that
\begin{equation}\label{eq:tuned-cohomology}
  f-\int f\,d\mu_{w,\alpha}=c_f\Bigl(\ind_{[w]}-\frac\alpha p\Bigr)+\tilde
  f\circ\sigma-\tilde f\qquad\text{on }X_{w,\alpha}.
\end{equation}
Moreover,
\begin{equation}\label{eq:tuned-bounds}
  |c_f|\le\frac{p\,\theta^{1-p}\,|f|_{\lip(X_{w,\alpha})}}{(1-\theta^p)^2}
  \qquad\text{and}\qquad
  |\tilde f|_{\lip(X_{w,\alpha})}\le\frac{5p^2\theta^{2-(N_\alpha+2)p}\,|f|_{\lip(X_{w,\alpha})}}{(1-\theta^p)^3} .
\end{equation}
Consequently the quotient of $\lip(X_{w,\alpha},d_\theta)$ by the subspace
spanned by the constants and the Lipschitz coboundaries is one-dimensional,
and is spanned by the equivalence class of $\ind_{[w]}$.  The functional
$f\mapsto c_f$ is the unique linear functional on $\lip(X_{w,\alpha},d_\theta)$
which vanishes on the constants and on the Lipschitz coboundaries and
satisfies $c_{\ind_{[w]}}=1$.
\end{thm}

\begin{proof}
Write $\zeta:=\zeta_w$, $\mu:=\mu_{w,\alpha}$ and
$L:=|f|_{\lip(X_{w,\alpha})}$.  The proof is in four steps.

\emph{Step 1: reduction to Theorem \ref{thm:D}.}
Define $\hat f:X_\alpha\to\R$ by $\hat f(y):=S_pf(\zeta(y))$.  
By
\eqref{eq:Lip-shift} and Lemma \ref{lem:tuning}(ii),
$d_\theta(\sigma^j\zeta(y),\sigma^j\zeta(y'))\le\theta^{-j}d_\theta(\zeta(y),\zeta(y'))
\le\theta^{-j}d_{\theta^p}(y,y')$ for $y,y'\in\Sigma$ and $j\ge0$.
Hence $|\hat f|_{\lip(X_\alpha,d_{\theta^p})}\le
L\sum_{j<p}\theta^{-j}\le p\,\theta^{1-p}L$, and Theorem \ref{thm:D}, with
$\theta^p$ in place of $\theta$, gives $c\in\R$ and $\hat
g\in\lip(X_\alpha,d_{\theta^p})$ such that $\hat f-\int\hat
f\,d\mu_\alpha=c(\Phi-\alpha)+\hat g\circ\sigma-\hat g$ on $X_\alpha$, and
\begin{equation}\label{eq:hat-bounds}
  |c|\le\frac{p\,\theta^{1-p}L}{(1-\theta^p)^2}
  \qquad\text{and}\qquad
  |\hat g|_{\lip(X_\alpha,d_{\theta^p})}\le\frac{p\,\theta L}{(1-\theta^p)^3} .
\end{equation}
By the definition of $\mu$, $\int\hat f\,d\mu_\alpha=\sum_{i<p}\int
f\circ\sigma^i\circ\zeta\,d\mu_\alpha=p\int f\,d\mu$, and by Lemma
\ref{lem:tuning}(iii), $S_p\ind_{[w]}(\zeta(y))=y_0=\Phi(y)$ for $y\in
X_\alpha$.  So the function $R:=f-\int f\,d\mu-c(\ind_{[w]}-\alpha/p)$ on
$X_{w,\alpha}$ satisfies
\begin{equation}\label{eq:block-sum}
  S_pR(\zeta(y))=\hat f(y)-\int\hat f\,d\mu_\alpha-c\bigl(\Phi(y)-\alpha\bigr)
  =\hat g(\sigma y)-\hat g(y)\qquad(y\in X_\alpha) .
\end{equation}

\emph{Step 2: the transfer function.}
For $x\in X_{w,\alpha}$ let $n(x):=p-i(x)$, with $i(x)$ as in Lemma
\ref{lem:tuned-subshift}(ii), so that $n(x)$ is the least $n\ge1$ with
$\sigma^nx\in\zeta(X_\alpha)$, by Lemma \ref{lem:tuned-subshift}(iii).
Define
\[
  \tilde f(x):=\hat g\bigl(\zeta^{-1}(\sigma^{n(x)}x)\bigr)-S_{n(x)}R(x) .
\]
For $y\in X_\alpha$ we have $n(\zeta(y))=p$, so \eqref{eq:block-sum} gives
$\tilde f(\zeta(y))=\hat g(\sigma y)-S_pR(\zeta(y))=\hat g(y)$.  Thus
$\tilde f=\hat g\circ\zeta^{-1}$ on $\zeta(X_\alpha)$, and therefore $\tilde
f(x)=\tilde f(\sigma^{n(x)}x)-S_{n(x)}R(x)$ for all $x\in X_{w,\alpha}$.  Now
let $x\in X_{w,\alpha}$.  If $n(x)\ge2$ then $n(\sigma x)=n(x)-1$, so
$\tilde f(\sigma x)=\tilde f(\sigma^{n(x)}x)-S_{n(x)-1}R(\sigma x)$, and this identity
clearly also holds if $n(x)=1$.
Subtracting the expression for $\tilde f(x)$ gives $\tilde f(\sigma
x)-\tilde f(x)=S_{n(x)}R(x)-S_{n(x)-1}R(\sigma x)=R(x)$.  So $R=\tilde
f\circ\sigma-\tilde f$ on $X_{w,\alpha}$, which is
\eqref{eq:tuned-cohomology} with $c_f:=c$.

\emph{Step 3: the Lipschitz bound.}
Set $M:=(N_\alpha+1)p$ and $Q:=p^2\theta^{1-p}(1-\theta^p)^{-3}$, so that each of 
$|\hat g|_{\lip(X_\alpha,d_{\theta^p})}$, $p\,\theta^{1-p}L$ and $p|c|$ is
at most $QL$, by \eqref{eq:hat-bounds}.  Let $x,x'\in X_{w,\alpha}$ be
distinct, and let $\ell$ be the smallest index with $x_\ell\ne x'_\ell$, so
that $d_\theta(x,x')=\theta^\ell$.

Suppose first that $\ell<M$, so that $d_\theta(x,x')\ge\theta^{M-1}$.  As
$X_\alpha$ and $X_{w,\alpha}$ have diameter at most $1$, all values of
$\hat g$ lie in an interval of length $|\hat
g|_{\lip(X_\alpha,d_{\theta^p})}$, and all values of $f$ lie in an interval
of length $L$, which contains $\int f\,d\mu$.  So $|R|\le L+|c|$, and the
definition of $\tilde f$ gives
\[
  |\tilde f(x)-\tilde f(x')|\le|\hat g|_{\lip(X_\alpha,d_{\theta^p})}+2p(L+|c|)
  \le5QL\le5\theta^{1-M}QL\,d_\theta(x,x') .
\]

Suppose now that $\ell\ge M$.  Then $i(x)=i(x')$, by Lemma
\ref{lem:tuned-subshift}(iii), so $n(x)=n(x')=:n$, and $\tilde f(x)-\tilde
f(x')=\hat g(z)-\hat g(z')-\bigl(S_nR(x)-S_nR(x')\bigr)$, where
$z:=\zeta^{-1}(\sigma^nx)$ and $z':=\zeta^{-1}(\sigma^nx')$.  By Lemma
\ref{lem:tuning}(ii) and \eqref{eq:Lip-shift},
\[
  d_{\theta^p}(z,z')=\theta^{-r}d_\theta(\sigma^nx,\sigma^nx')
  \le\theta^{-r-n}d_\theta(x,x')\le\theta^{1-2p}d_\theta(x,x') ,
\]
with $r<p$ as in Definition \ref{def:tuning}.  For $0\le j<n$ the points
$\sigma^jx$ and $\sigma^jx'$ have the same first $\ell-j\ge p$ symbols, so
$\ind_{[w]}(\sigma^jx)=\ind_{[w]}(\sigma^jx')$, and therefore
$|R(\sigma^jx)-R(\sigma^jx')|=|f(\sigma^jx)-f(\sigma^jx')|\le
L\theta^{-j}d_\theta(x,x')$.  Hence, as $M\ge2p$,
\[
  |\tilde f(x)-\tilde f(x')|\le\bigl(\theta^{1-2p}|\hat
  g|_{\lip(X_\alpha,d_{\theta^p})}+p\,\theta^{1-p}L\bigr)\,d_\theta(x,x')
  \le2\theta^{1-M}QL\,d_\theta(x,x') .
\]

In both cases
\[
  |\tilde f(x)-\tilde f(x')|\le5\theta^{1-M}QL\,d_\theta(x,x')
  =\frac{5p^2\theta^{2-(N_\alpha+2)p}L}{(1-\theta^p)^3}\,d_\theta(x,x') ,
\]
which, together with \eqref{eq:hat-bounds}, proves the required \eqref{eq:tuned-bounds}.

\emph{Step 4: uniqueness, the functional $f\mapsto c_f$, and the quotient.}
Suppose that $\ind_{[w]}-\alpha/p=u\circ\sigma-u$ on $X_{w,\alpha}$ for
some $u\in\lip(X_{w,\alpha},d_\theta)$.  For $y\in X_\alpha$, summing this
identity over the points $\sigma^j\zeta(y)$ with $0\le j<p$, and using
$S_p\ind_{[w]}(\zeta(y))=\Phi(y)$ and $\sigma^p\circ\zeta=\zeta\circ\sigma$,
gives $\Phi(y)-\alpha=u(\zeta(\sigma y))-u(\zeta(y))$.  As $u\circ\zeta$ lies
in $\lip(X_\alpha,d_{\theta^p})$, by Lemma \ref{lem:tuning}(ii), it follows
that $\Phi-\alpha$ is a Lipschitz coboundary on $(X_\alpha,d_{\theta^p})$,
contradicting Proposition \ref{prop:rigidity}(iv), with $\theta^p$ in place
of $\theta$.  So $\ind_{[w]}-\alpha/p$ is not a Lipschitz coboundary on
$X_{w,\alpha}$.

Now suppose that \eqref{eq:tuned-cohomology} also holds with $c_f$ and
$\tilde f$ replaced by $c'\in\R$ and $\tilde f'\in\lip(X_{w,\alpha},d_\theta)$.
Then $h:=\tilde f'-\tilde f$ satisfies
$(c_f-c')(\ind_{[w]}-\alpha/p)=h\circ\sigma-h$, so $c'=c_f$ and
$h\circ\sigma=h$; hence $h$ is constant, as it is continuous and
$X_{w,\alpha}$ is minimal (Lemma \ref{lem:tuned-subshift}(iv)).  This
proves that $c_f$ is unique, and that $\tilde f$ is unique up to an additive
constant.

By the uniqueness of $c_f$, the functional $f\mapsto c_f$ is linear.  It
vanishes on the constants and on the Lipschitz coboundaries, as for these
functions \eqref{eq:tuned-cohomology} holds with $c_f=0$, since $\mu$
is invariant; and $c_{\ind_{[w]}}=1$, as \eqref{eq:tuned-cohomology}
holds for $f=\ind_{[w]}$ with $c_f=1$ and $\tilde f=0$, since
$\int\ind_{[w]}\,d\mu=\alpha/p$ by Lemma \ref{lem:tuned-subshift}(iv).  By
\eqref{eq:tuned-cohomology}, for all $f\in\lip(X_{w,\alpha},d_\theta)$ the
function $f-c_f\ind_{[w]}$ is the sum of a constant and a Lipschitz
coboundary.  Hence all linear functionals on $\lip(X_{w,\alpha},d_\theta)$
which vanish on the constants and on the Lipschitz coboundaries, and take
the value $1$ at $\ind_{[w]}$, take the value $c_f$ at $f$; and the
quotient of $\lip(X_{w,\alpha},d_\theta)$ by the subspace spanned by the
constants and the Lipschitz coboundaries is spanned by the equivalence class
of $\ind_{[w]}$, which is non-zero, as $c_{\ind_{[w]}}=1$, so this quotient
is one-dimensional.  This proves the last two statements of the theorem.
\end{proof}

\begin{rem}[Relation between Theorems \ref{thm:tuned-rigidity} and
\ref{thm:D}]\label{rem:tuned-rigidity}
For the tuning word $w=1$ we have $p=1$, $X_{w,\alpha}=X_\alpha$,
$\mu_{w,\alpha}=\mu_\alpha$ and $\ind_{[w]}=\Phi$, so in this case Theorem
\ref{thm:tuned-rigidity} reduces to Theorem \ref{thm:D}, except that its
bound on $|\tilde f|_{\lip(X_\alpha)}$ is weaker than \eqref{eq:u-bound}.
The one-dimensionality in Theorem \ref{thm:D} is invariant under topological
conjugacy: if $\pi:X\to X'$ is a topological conjugacy between subshifts,
then by uniform continuity there exists $k\in\N$ such that $\pi(x)_0$
depends only on $x_{[0,k)}$, so that $d_\theta(\pi x,\pi
y)\le\theta^{1-k}d_\theta(x,y)$ for all $x,y\in X$, and similarly for
$\pi^{-1}$; hence $f\mapsto f\circ\pi$ maps $\lip(X',d_\theta)$ onto
$\lip(X,d_\theta)$, and the constants and Lipschitz coboundaries on $X'$
onto those on $X$, so it induces an isomorphism between the corresponding
quotients.  For $p\ge2$ the tuned Sturmian subshift $X_{w,\alpha}$ is not
topologically conjugate to a Sturmian subshift, by Proposition
\ref{prop:tuned-not-sturmian}, so Theorem \ref{thm:tuned-rigidity} does not
follow from Theorem \ref{thm:D} by a change of coordinates.
\end{rem}

\subsection{The maximizing property}

The following version of Proposition \ref{prop:Phimax} is the second main
ingredient in the proof of Theorem \ref{thm:tuned} below.  Its proof
considers orbit segments in the tuned Sturmian beta-shift $Y_{w,\alpha}$, all
of whose words of length $(N_\alpha+1)p+1$ occur in $X_{w,\alpha}$, where $p$
is the length of the tuning word $w$.  Along such a segment, the blocks $w$
and $w^-$ follow each other in the same manner
as the digits $1$ and $0$ do in a point of the
Sturmian beta-shift $Y_\alpha$, so Proposition \ref{prop:Phimax} bounds the
number of blocks $w$.

\begin{prop}\label{prop:tuned-max}
Let $w$ be a tuning word of length $p$, let $\alpha\in(0,1)$ be
irrational, with $N_\alpha:=\lceil1/\alpha\rceil$, let $Y_{w,\alpha}$,
$X_{w,\alpha}$ and $\mu_{w,\alpha}$ be
respectively the tuned Sturmian beta-shift,
subshift and measure corresponding to $w$ and $\alpha$, and
let $\theta\in(0,1)$.  Define
$\Psi_w,\Phi_w:Y_{w,\alpha}\to\R$ by $\Psi_w(x):=d_\theta(x,X_{w,\alpha})$
and $\Phi_w:=\ind_{[w]}-3\theta^{-(N_\alpha+1)p}\Psi_w$.  Then
\begin{equation}\label{eq:tuned-uniform}
  S_n\Phi_w(x)\le\frac{n\alpha}p+2\qquad\text{for all }x\in Y_{w,\alpha}\text{
  and all }n\in\N .
\end{equation}
Consequently $\int\Phi_w\,d\nu\le\alpha/p$ for all
$\nu\in\m(Y_{w,\alpha},\sigma)$, with equality for $\nu=\mu_{w,\alpha}$, so
that $\mu_{w,\alpha}\in\Mmax(\Phi_w)$.
\end{prop}
\begin{proof}
Fix $x\in Y_{w,\alpha}$ and $n\in\N$, and set $M:=(N_\alpha+1)p$.  We call
$j\in\{0,1,\ldots,n-1\}$ \emph{good} if the word $x_{[j,j+M+1)}$ occurs in
$X_{w,\alpha}$, and \emph{bad} otherwise.  If $j$ is bad, then all $z\in
X_{w,\alpha}$ satisfy $z_{[0,M+1)}\ne x_{[j,j+M+1)}$, and hence
$d_\theta(\sigma^jx,z)\ge\theta^M$; so $\Psi_w(\sigma^jx)\ge\theta^M$, and
therefore $\Phi_w(\sigma^jx)\le1-3=-2$.  If $j$ is good, then
$\Phi_w(\sigma^jx)\le\ind_{[w]}(\sigma^jx)$, since $\Psi_w\ge0$.

We claim that if $0\le s<t\le n$ and all $j\in\{s,\ldots,t-1\}$ are good,
then
\begin{equation}\label{eq:good-run}
  S_{t-s}\ind_{[w]}(\sigma^sx)<\frac{(t-s)\alpha}p+2 ,
\end{equation}
that is, fewer than $(t-s)\alpha/p+2$ integers $j\in\{s,\ldots,t-1\}$
satisfy $x_{[j,j+p)}=w$.  Given the claim, let $B$ be the number of bad
$j<n$.  The good $j<n$ form at most $B+1$ maximal sets of consecutive
integers, and applying \eqref{eq:good-run} to each of them gives
\[
  S_n\Phi_w(x)\le\sum_{\substack{0\le j<n\\ j\text{ good}}}\ind_{[w]}(\sigma^jx)-2B
  \le\frac{(n-B)\alpha}p+2(B+1)-2B\le\frac{n\alpha}p+2 ,
\]
which is \eqref{eq:tuned-uniform}.

To prove the claim, let $0\le s<t\le n$ be such that all
$j\in\{s,\ldots,t-1\}$ are good.
Let $i(z)\in\{0,1,\ldots,p-1\}$ be as in Lemma \ref{lem:tuned-subshift}(ii),
defined for $z\in X_{w,\alpha}$, which depends only on $z_{[0,M)}$, by Lemma
\ref{lem:tuned-subshift}(iii).  For $s\le j<t$, choose $z\in X_{w,\alpha}$
with $z_{[0,M+1)}=x_{[j,j+M+1)}$, which is possible because $j$ is good and
$\sigma(X_{w,\alpha})=X_{w,\alpha}$, and set $i_j:=i(z)$, which depends only
on $x_{[j,j+M)}$, and not on the choice of $z$.  As $z_{[0,p)}=x_{[j,j+p)}$
and $0\le i_j<p$, Lemma \ref{lem:tuned-subshift}(ii) shows that if
$x_{[j,j+p)}=w$ then $i_j=0$, and that if $i_j=0$ then
$x_{[j,j+p)}\in\{w,w^-\}$.  If moreover $j+1<t$, then $\sigma z\in
X_{w,\alpha}$ and $(\sigma z)_{[0,M)}=x_{[j+1,j+1+M)}$, so
$i_{j+1}=i(\sigma z)\equiv i_j+1\pmod p$, by Lemma
\ref{lem:tuned-subshift}(iii).

Hence $i_j\equiv i_s+j-s\pmod p$ for $s\le j<t$, so the integers
$j\in\{s,\ldots,t-1\}$ with $i_j=0$ are those congruent to $s-i_s$ modulo
$p$.  If there are no such $j$, then $x_{[j,j+p)}\ne w$ for $s\le j<t$, so
the left-hand side of \eqref{eq:good-run} vanishes.  Otherwise, these
integers are $j_k:=j_0+pk$ for $0\le k\le K$, where $j_0\ge s$ and $K\ge0$
satisfy $j_0+pK\le t-1$, so that $K+1<(t-s)/p+1$.  As
$x_{[j_k,j_k+p)}\in\{w,w^-\}=\{\zeta_w(1),\zeta_w(0)\}$ for $0\le k\le K$,
there is a word $b=b_0b_1\cdots b_K\in\{0,1\}^{K+1}$ such that
$x_{[j_0,j_0+(K+1)p)}=\zeta_w(b)$; and as $x_{[j,j+p)}\ne w$ for all other
$j\in\{s,\ldots,t-1\}$, the left-hand side of \eqref{eq:good-run} is the
number $|b|_1$ of digits $1$ in $b$.

We show that $b0^\infty\in Y_\alpha$, writing $\bfa=\bfa(\alpha)=(a_i)_{i\ge0}$.
Let $0\le k\le K$, and suppose that $b_{[k,K+1)}>_\lex a_{[0,K+1-k)}$; let
$h$ be the smallest index with $b_{k+h}\ne a_h$, so that $b_{k+h}=1$ and
$a_h=0$.  Then $\sigma^{j_k}x$ begins with
$\zeta_w(b_{[k,k+h+1)})=\zeta_w(a_{[0,h)})\,w$, while the maximal tuned
Sturmian sequence $\zeta_w(\bfa)$ begins with
$\zeta_w(a_{[0,h+1)})=\zeta_w(a_{[0,h)})\,w^-$; since $w^-<_\lex w$, it
follows that $\sigma^{j_k}x>_\lex\zeta_w(\bfa)$, which contradicts $x\in
Y_{w,\alpha}$, by \eqref{eq:Ywalpha}.  Hence $b_{[k,K+1)}\le_\lex
a_{[0,K+1-k)}$ for $0\le k\le K$.  As $\sigma^k(b0^\infty)=b_{[k,K+1)}0^\infty$
for $0\le k\le K$, and $\sigma^k(b0^\infty)=0^\infty$ for $k>K$, it follows
that $\sigma^k(b0^\infty)\le_\lex\bfa$ for all $k\ge0$, that is,
$b0^\infty\in Y_\alpha$.  By Proposition \ref{prop:Phimax},
\[
  |b|_1=S_{K+1}\Phi(b0^\infty)\le(K+1)\alpha+1<\Bigl(\frac{t-s}p+1\Bigr)\alpha+1
  <\frac{(t-s)\alpha}p+2 ,
\]
which proves the claimed inequality
\eqref{eq:good-run}, and hence \eqref{eq:tuned-uniform}.

Finally, for $\nu\in\m(Y_{w,\alpha},\sigma)$,   \eqref{eq:tuned-uniform}
yields $\int\Phi_w\,d\nu\le\alpha/p+2/n$ for all $n\in\N$, so
$\int\Phi_w\,d\nu\le\alpha/p$.  By Lemma \ref{lem:tuned-subshift}(i), (iv),
the measure $\mu_{w,\alpha}$ lies in $\m(Y_{w,\alpha},\sigma)$ and satisfies
$\mu_{w,\alpha}([w])=\alpha/p$; as $\Psi_w$ vanishes on $X_{w,\alpha}$, it
follows that $\int\Phi_w\,d\mu_{w,\alpha}=\alpha/p$.
\end{proof}

\begin{rem}\label{rem:penalty}
Unlike the digit function on a Sturmian beta-shift (Proposition
\ref{prop:Phimax}), the function $\ind_{[w]}$ need not be maximized by the
tuned Sturmian measure, so the penalty term in the definition of $\Phi_w$
in Proposition~\ref{prop:tuned-max}
cannot be omitted.
For example, let $w$ be the tuning word $10$, so that $p=2$ and $w^-=00$,
and let $\alpha\in(1/3,2/5)$ be irrational.  Then $\bfa$ begins with $101$,
by \eqref{eq:prefixsum}, so $\zeta_w(\bfa)$ begins with $100010$, and every
shift of the periodic sequence $(10000)^\infty$ is lexicographically smaller
than $\zeta_w(\bfa)$.  So $(10000)^\infty\in Y_{w,\alpha}$, and its periodic
orbit measure gives the cylinder set $[w]$ measure
$1/5>\alpha/2=\mu_{w,\alpha}([w])$.
\end{rem}

\subsection{Stability of tuned Sturmian measures}

We now prove the version of Theorems \ref{thm:A} and \ref{thm:C} for
tuned Sturmian beta-shifts, and deduce Theorem \ref{thm:density}.

\begin{thm}[\textbf{Failure of TPO for tuned Sturmian
beta-shifts}]\label{thm:tuned}
Let $w$ be a tuning word of length $p$, let $\alpha\in(0,1)$ be
irrational, with $N_\alpha:=\lceil1/\alpha\rceil$, let $Y_{w,\alpha}$,
$X_{w,\alpha}$ and $\mu_{w,\alpha}$ be respectively the tuned Sturmian beta-shift,
subshift and measure corresponding to $w$ and $\alpha$, and
let $\theta\in(0,1)$.  Define
$F_w\in\lip(Y_{w,\alpha},d_\theta)$ by
\[
  F_w(x):=\ind_{[w]}(x)-\bigl(1+3\theta^{-(N_\alpha+1)p}\bigr)\,d_\theta(x,X_{w,\alpha}),
\]
and set $\delta:=\theta^{(N_\alpha+2)p}(1-\theta^p)^3/(15p^2)$.  Then
\begin{equation}\label{eq:wasserstein-tuned}
  \int F_w\,d\mu_{w,\alpha}-\int F_w\,d\nu\ \ge\ \delta\,W_\theta(\nu,\mu_{w,\alpha})
  \qquad\text{for all }\nu\in\m(Y_{w,\alpha},\sigma),
\end{equation}
so all $f\in\lip(Y_{w,\alpha},d_\theta)$ with
$|f-F_w|_{\lip(Y_{w,\alpha})}<\delta$ satisfy $\Mmax(f)=\{\mu_{w,\alpha}\}$,
where $\mu_{w,\alpha}$ is not a periodic orbit measure.  In particular,
$Y_{w,\alpha}$ does not have TPO with respect to
$\lip(Y_{w,\alpha},d_\theta)$.
\end{thm}

\begin{proof}
Write $\mu:=\mu_{w,\alpha}$ and $M:=(N_\alpha+1)p$, and let $\Psi_w$ and
$\Phi_w$ be as in Proposition \ref{prop:tuned-max}, so that
$F_w=\Phi_w-\Psi_w$.  As $\ind_{[w]}(x)$ depends only on $x_{[0,p)}$, we
have $|\ind_{[w]}|_{\lip(Y_{w,\alpha})}\le\theta^{1-p}$, and
$|\Psi_w|_{\lip(Y_{w,\alpha})}\le1$, so
$|\Phi_w|_{\lip(Y_{w,\alpha})}\le\theta^{1-p}+3\theta^{-M}\le4\theta^{-M}$
and $F_w\in\lip(Y_{w,\alpha},d_\theta)$.

We apply Lemma \ref{lem:rigid-stability} with $Y:=Y_{w,\alpha}$,
$X:=X_{w,\alpha}$ and $\phi:=\Phi_w$, for which $\Psi_X=\Psi_w$ and
$\phi-\Psi_X=F_w$.  By Proposition \ref{prop:tuned-max},
$\int\Phi_w\,d\nu\le\alpha/p=\int\Phi_w\,d\mu$ for all
$\nu\in\m(Y_{w,\alpha},\sigma)$.  As $\Phi_w=\ind_{[w]}$ on $X_{w,\alpha}$,
Theorem \ref{thm:tuned-rigidity} shows that all
$f\in\lip(X_{w,\alpha},d_\theta)$ satisfy \eqref{eq:abstract-decomp} with
$c:=c_f$ and $u:=\tilde f$, and by \eqref{eq:tuned-bounds} we may take
\[
  A:=\frac{p\,\theta^{1-p}}{(1-\theta^p)^2},\qquad
  B:=\frac{5p^2\theta^{2-(N_\alpha+2)p}}{(1-\theta^p)^3} .
\]
As $A\le4\theta^{-M}A$ and $1+\theta^{-1}\le2\theta^{-1}$, the constant $C$
of Lemma \ref{lem:rigid-stability} satisfies
\[
  C\le1+4\theta^{-M}A+2\theta^{-1}B
  =1+\frac{4p\,\theta^{1-(N_\alpha+2)p}}{(1-\theta^p)^2}
  +\frac{10p^2\theta^{1-(N_\alpha+2)p}}{(1-\theta^p)^3}\le\delta^{-1} ,
\]
so \eqref{eq:abstract-gap} gives \eqref{eq:wasserstein-tuned}.

By Lemma \ref{lem:stability}, it follows that $\Mmax(f)=\{\mu\}$ for all
$f\in\lip(Y_{w,\alpha},d_\theta)$ with
$|f-F_w|_{\lip(Y_{w,\alpha})}<\delta$.  These $f$ form a non-empty open
subset of $\lip(Y_{w,\alpha},d_\theta)$, as in the proof of Theorem
\ref{thm:A}, and $\mu$ is not a periodic orbit measure, by Lemma
\ref{lem:tuned-subshift}(iv); so this open set does not intersect
$\cP(Y_{w,\alpha})$, thus $Y_{w,\alpha}$ does not have TPO with respect to
$\lip(Y_{w,\alpha},d_\theta)$.
\end{proof}

It is instructive to compare Theorem \ref{thm:tuned} with Theorem \ref{thm:A}:

\begin{rem}[Tuned Sturmian and Sturmian beta-shifts]\label{rem:tuned-new}
If $p=1$, so that $w$ is a single digit $w_0$, then the tuned Sturmian
beta-shift $Y_{w,\alpha}$ is the Sturmian beta-shift $Y_\alpha$ if $w_0=1$,
and is the beta-shift $Y(\iota(\bfa))$ of Remark \ref{rem:alphabets}, with
$\ell=w_0-1$ and $m=w_0$, if $w_0\ge2$.  So in this case Theorem
\ref{thm:tuned} gives no new parameters $\beta$, but its open set of
functions differs from the one given by Theorem \ref{thm:A} if $w_0=1$, and
from the one given by Remark \ref{rem:alphabets} if $w_0\ge2$.
If $p\ge2$,
then $d_{\beta_{w,\alpha}}(1)$ is not a Sturmian sequence over any alphabet
$\{\ell,m\}$, by Proposition \ref{prop:tuned-not-sturmian}, so Theorem
\ref{thm:tuned} gives parameters $\beta$ for which $S_\beta$ does not have
TPO, and which are not covered by Theorem \ref{thm:A} or Remark
\ref{rem:alphabets}.  
Tuned Sturmian sequences are nevertheless \emph{quasi-Sturmian} in the
sense of Cassaigne \cite[\S 5.2]{Cas98}: for each tuned Sturmian sequence
$x$ there is an integer $c$ such that, for all sufficiently large $n$,
exactly $n+c$ words of length $n$ occur in $x$.  This follows from
\cite[Prop.~8]{Cas98}, as $x$ has the form $u\,\zeta_w(v)$, with $u$ a
finite word and $v$ a Sturmian sequence, and $\zeta_w(01)=w^-w\ne
ww^-=\zeta_w(10)$.
\end{rem}

We are now able to prove Theorem \ref{thm:density}, by combining Corollary
\ref{cor:tuned-dense}, Theorem \ref{thm:tuned} and Corollary \ref{cor:G}.

\begin{proof}[Proof of Theorem \ref{thm:density}]
Let $J\subseteq(1,\infty)$ be a non-empty open interval.  By Corollary
\ref{cor:tuned-dense}, there are a tuning word $w$ and $\alpha_0\in(0,1)$
such that the tuned Sturmian parameter $\beta_{w,\alpha}$ lies in $J$ for
all irrational $\alpha\in(\alpha_0,1)$, and there are uncountably many such
parameters.  Let $\alpha\in(\alpha_0,1)$ be irrational, and set
$\beta:=\beta_{w,\alpha}$.  By Theorem \ref{thm:tuned}, the tuned Sturmian
beta-shift $S_\beta=Y_{w,\alpha}$ does not have TPO with respect to
$\lip(S_\beta,d_\theta)$, for all $\theta\in(0,1)$.  As $\beta$ is not a
beta-number (Lemma \ref{lem:tuning}(iv)), Corollary \ref{cor:G}, applied
with $\theta:=\beta^{-\gamma}$, shows that for all $\gamma\in(0,1]$ neither
$U_\beta$ nor $T_\beta$ has TPO with respect to $C^{0,\gamma}(I)$.  So $\beta$
belongs to each of the three sets in Theorem \ref{thm:dimension}, which are
therefore dense in $(1,\infty)$.
\end{proof}

By contrast, the Sturmian parameters themselves are not dense in any interval:

\begin{rem}\label{rem:sturmian-closure}
The set of \emph{Sturmian parameters}, i.e.~those $\beta>1$ such that
$d_\beta(1)$ is a Sturmian sequence, has nowhere dense closure.  Indeed,
these parameters are the numbers $\Delta(\alpha)$ with $\alpha\in(0,1)$
irrational, where $\Delta:[0,\infty)\to[1,\infty)$ is the function of Kwon
\cite{Kwo08}, defined by prescribing $d_{\Delta(\alpha)}(1)$ in terms of the
Sturmian and Christoffel words of slope $\alpha$, and satisfying
$d_{\Delta(\alpha)}(1)=\bfa(\alpha)$ for irrational $\alpha\in(0,1)$; by
\cite[Thm.~2.7]{Kwo08}, $\Delta$ is strictly increasing, and continuous at
each irrational point.  So if $\beta$ is the limit of a sequence
$\Delta(\alpha_n)$ with $\alpha_n\in(0,1)$ irrational, then, passing to a
monotone subsequence of $(\alpha_n)$, with limit $\alpha\in[0,1]$, we see
that $\beta$ is a one-sided limit of $\Delta$ at $\alpha$, and that
$\beta=\Delta(\alpha)$ if $\alpha$ is irrational.  Hence the closure of the
set of Sturmian parameters is contained in the union of that set with the
countable set of one-sided limits of $\Delta$ at rational points of
$[0,1]$.  So this closure has zero Lebesgue measure, by Theorem
\ref{thm:dimension} or by \cite[Thm.~2.7]{Kwo08}, and being closed, it is
nowhere dense.
\end{rem}

\end{document}